\documentclass[12pt]{article}
\usepackage{amssymb,amsmath,amsthm,tikz}
\usetikzlibrary{arrows}

\title{Spherical Tiling by $12$ Congruent Pentagons}
\author{Honghao Gao, Nan Shi, Min Yan\thanks{Research was supported by Hong Kong RGC General Research Fund 605610 and 606311.} \\ 
Hong Kong University of Science and Technology}

\newtheorem{theorem}{Theorem}
\newtheorem{lemma}[theorem]{Lemma}

\newtheorem{proposition}[theorem]{Proposition}

\theoremstyle{definition}
\newtheorem{definition}[theorem]{Definition}
\theoremstyle{remark}

\numberwithin{equation}{section}

\begin{document}

\maketitle

\begin{abstract}
The tilings of the $2$-dimensional sphere by congruent triangles have been extensively studied, and the edge-to-edge tilings have been completely classified. However, not much is known about the tilings by other congruent polygons. In this paper, we classify the simplest case, which is the edge-to-edge tilings of the $2$-dimensional sphere by $12$ congruent pentagons. We find one major class allowing two independent continuous parameters and four classes of isolated examples. The classification is done by first separately classifying the combinatorial, edge length, and angle aspects, and then combining the respective classifications together.
\end{abstract}

\section{Introduction}

Tilings have been studied by mathematicians for more than 100 years. Some major achievements include the solution to Hilbert's 18th problem \cite{b, h, r}, the classification of wallpaper groups \cite{f1,p} and crystallographic groups \cite{ba,f2,s1}, the classification of isohedral tilings of the plane \cite{gs2}, and the classification of edge-to-edge monohedral tilings of the 2-sphere by triangles \cite{d,paper_orin,paper_tri}. We refer the reader for a more complete history of the subject to the excellent 1987 monograph by Gr\"unbaum and Shephard \cite{gs}.

In this paper, we study {\em monohedral} tilings, which are tilings that all tiles are geometrically congruent. For monohedral tilings of the plane by polygons, it is easy to see that any triangle or quadrilateral can be the tile. It is also known that convex tiles cannot have more than six edges, and there are exactly three classes of convex hexagonal tiles \cite{r1} and at least $14$ classes of convex pentagonal tiles \cite{gs,sc}. Our main result assumes straight line edges (actually great arcs) but not the convexity. Moreover, most of the other results in this paper allow any reasonably nice curves to be the edges. 

The tilings we study are also {\em edge-to-edge}, which means that no vertex of a tile lies in the interior of an edge of another tile. The assumption simplifies the classification of tiling patterns, or all the possible ways the tiles fit together to cover the whole sphere. It is easy to see that any triangle and quadrilateral can be the tile of an edge-to-edge monohedral tiling of the plane. Moreover, $8$ of the $14$ convex pentagon classes can be used for edge-to-edge tilings \cite{ba1}. On the other hand, there is no general classification of plane tiling patterns, even for edge-to-edge and monohedral tilings by convex polygons. The known classifications often assume certain symmetry \cite{gs2}, or certain special geometric property of the tiles \cite{so1, so2}. The difficulty with classifying the plane tiling patterns is that the number of tiles is infinite. By restricting to spherical tilings, the problem becomes finite and more manageable. Our classification does not assume any symmetry or special geometric property.

We also restrict the study to the tilings such that all vertices have {\em degree $\ge 3$}. This avoids some trivial examples obtained by artificially adding extra vertices to the edges, or getting new tilings by ``wiggling modifications'' of the edges. The assumption further simplifies the classification.

Now we come to the subject of this paper, the edge-to-edge monohedral tilings of the sphere by polygons, subject to degree $\ge 3$ vertex condition. It is not hard to see that the tile can only be triangle, quadrilateral, or pentagon \cite[Proposition 4]{paper_quadr}. The first work in this direction is Sommerville's 1923 partial classification of edge-to-edge monohedral spherical tilings by triangles \cite{paper_orin}. Davies \cite{d} outlined a complete classification in 1967, and the complete proof was finally given by Ueno and Agaoka \cite{paper_tri} in 2002. See \cite{db} for the further study of monohedral non-edge-to-edge spherical tilings by triangles. As for quadrilateral spherical tilings, Ueno and Agaoka \cite{paper_quadr} showed that the classification can be very complicated. For works on the spherical tilings by special types of quadrilaterals, see \cite{ab,as}.

As far as we know, there have been virtually no study of tilings of the sphere by congruent pentagons. However, we are of the opinion that the spherical pentagon tilings should be easier to study than the quadrilateral ones, because among triangle, quadrilateral and pentagon, pentagon is the ``other extreme''. We feel that the spherical pentagon should be compared to the planar hexagon, and the spherical quadrilateral should be compared to the planar pentagon. To test our conviction, we classify the minimal case of the edge-to-edge monohedral tiling of the sphere by pentagons. 

\bigskip

\noindent {\bf Main Theorem}
{\em Any edge-to-edge spherical tiling by $12$ congruent pentagons must belong to one of the five classes in Figure \ref{completeclassify}, subject to the condition that the sum of angles at any vertex is $2\pi$.}

\bigskip

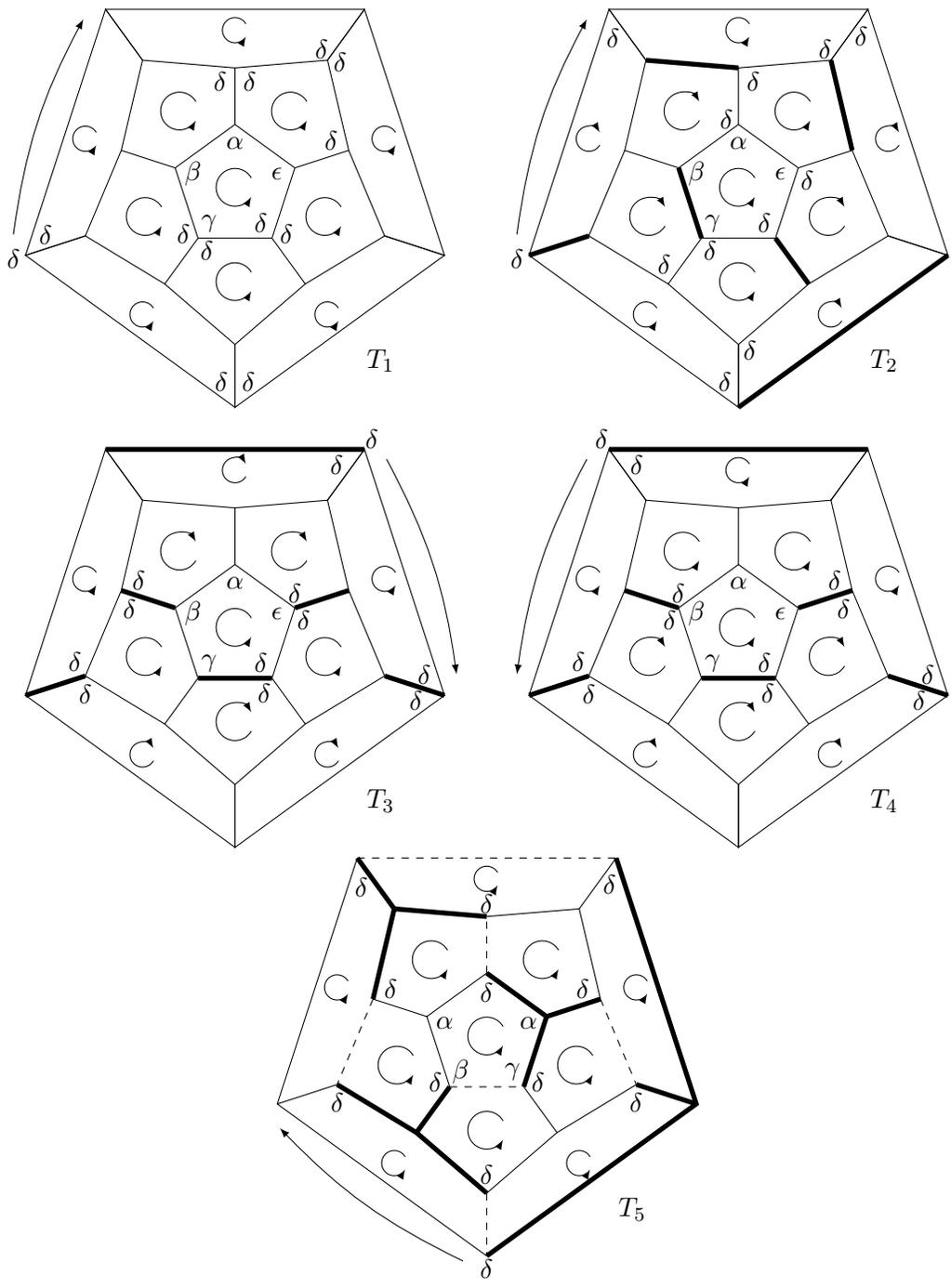
\begin{figure}[htp]
\centering
    \begin{tikzpicture}[>=latex, scale=0.9]

\foreach \x in {1,...,5}
    \draw 
    		(-54+72*\x:1) -- (18+72*\x:1)
    		(-54+72*\x:1) -- (-54+72*\x:1.9) -- (-18+72*\x:2.5) -- (18+72*\x:1.9)
		(-18+72*\x:2.5) -- (-18+72*\x:3.5) -- (54+72*\x:3.5) -- (54+72*\x:2.5);

    \node at (90:0.7) {\small $\alpha$};    
    \node at (162:0.7) {\small $\beta$};    
    \node at (234:0.7) {\small $\gamma$};
    \node at (18:0.7) {\small $\epsilon$};
    \node at (-54:0.7) {\small $\delta$};    
    \node at (58:2.6) {\small $\delta$};    
    \node at (50:2.6) {\small $\delta$};       
    \node at (-86:3.1) {\small $\delta$};        
    \node at (-94:3.1) {\small $\delta$};
    \node at (-114:1.1) {\small $\delta$};  
    \node at (222:1.1) {\small $\delta$};   
    \node at (-42:1.1) {\small $\delta$};
    \node at (82:1.7) {\small $\delta$};  
    \node at (98:1.7) {\small $\delta$};     
    \node at (26:1.7) {\small $\delta$};        
    \node at (194:3.1) {\small $\delta$};
	\node at (198:3.7) {\small $\delta$};
        
    \node at (310:3.6) {\small $T_1$};

	\draw[->]
		(30:0.3) arc (30:330:0.3);
	\draw[rotate=72,->]
		(120:3.6) to[out=10, in=170] (60:3.6);
	\foreach \x in {0,...,4}
	{
	\draw[shift={(54+72*\x:1.5)},->]
		(30:0.3) arc (30:330:0.3);
	\draw[shift={(90+72*\x:2.5)},->]
		(30:0.2) arc (30:330:0.2);
	}


\begin{scope}[xshift=8cm]

	\foreach \x in {1,...,5}
    \draw 
    		(-54+72*\x:1) -- (18+72*\x:1)
    		(-54+72*\x:1) -- (-54+72*\x:1.9) -- (-18+72*\x:2.5) -- (18+72*\x:1.9)
		(-18+72*\x:2.5) -- (-18+72*\x:3.5) -- (54+72*\x:3.5) -- (54+72*\x:2.5);

	\draw[line width=0.7mm]
		(162:1) -- (234:1)
		(-54:1) -- (-54:1.9)
		(18:1.9) -- (54:2.5)
		(90:1.9) -- (126:2.5)
		(198:2.5) -- (198:3.5)
		(-18:3.5) -- (-90:3.5);
  
    \node at (90:0.7) {\small $\alpha$};    
    \node at (162:0.7) {\small $\beta$};    
    \node at (234:0.7) {\small $\gamma$};
    \node at (18:0.7) {\small $\epsilon$};
    \node at (-54:0.7) {\small $\delta$}; 
    \node at (6:1.1) {\small $\delta$};   
    \node at (-114:1.1) {\small $\delta$};         
    \node at (51:3.1) {\small $\delta$};  
    \node at (130:3.1) {\small $\delta$}; 
    \node at (82:1.7) {\small $\delta$};    
    \node at (100:1.1) {\small $\delta$};
    \node at (226:1.7) {\small $\delta$};      
    \node at (-86:2.6) {\small $\delta$};     
    \node at (-94:3.1) {\small $\delta$};     
    \node at (58:2.6) {\small $\delta$};
	\node at (198:3.7) {\small $\delta$};   
        
    \node at (310:3.6) {\small $T_2$};

	\draw[->]
		(30:0.3) arc (30:330:0.3);
	\draw[rotate=72,->]
		(120:3.6) to[out=10, in=170] (60:3.6);
	\foreach \x in {0,2}
	{
	\draw[shift={(-90+72*\x:1.5)},->]
		(30:0.3) arc (30:330:0.3);
	\draw[shift={(90+72*\x:2.5)},->]
		(30:0.2) arc (30:330:0.2);
	}
	\foreach \x in {0,2,3}
	{
	\draw[shift={(-18+72*\x:1.5)},<-]
		(30:0.3) arc (30:330:0.3);
	\draw[shift={(162+72*\x:2.5)},<-]
		(30:0.2) arc (30:330:0.2);
	}
	
\end{scope}

    
\begin{scope}[yshift=-7cm]

	\foreach \x in {1,...,5}
    \draw 
    		(-54+72*\x:1) -- (18+72*\x:1)
    		(-54+72*\x:1) -- (-54+72*\x:1.9) -- (-18+72*\x:2.5) -- (18+72*\x:1.9)
		(-18+72*\x:2.5) -- (-18+72*\x:3.5) -- (54+72*\x:3.5) -- (54+72*\x:2.5);

	\draw[line width=0.7mm]
		(234:1) -- (-54:1)
		(18:1) -- (18:1.9)
		(162:1) -- (162:1.9)
		(-18:2.5) -- (-18:3.5)
		(198:2.5) -- (198:3.5)
		(126:3.5) -- (54:3.5);
    
    \node at (310:3.6) {\small $T_3$};
    
    \node at (90:0.7) {\small $\alpha$};    
    \node at (162:0.7) {\small $\beta$};    
    \node at (234:0.7) {\small $\gamma$};
    \node at (18:0.7) {\small $\epsilon$};
    \node at (-54:0.7) {\small $\delta$};       
    \node at (-22:3.1) {\small $\delta$};        
    \node at (-14:3.1) {\small $\delta$};   
    \node at (154:1.7) {\small $\delta$};  
    \node at (170:1.7) {\small $\delta$};   
    \node at (30:1.1) {\small $\delta$};  
    \node at (6:1.1) {\small $\delta$};        
    \node at (58:3.05) {\small $\delta$};     
    \node at (-66:1.15) {\small $\delta$};          
    \node at (192:2.6) {\small $\delta$};     
    \node at (204:2.6) {\small $\delta$};   
	\node at (54:3.7) {\small $\delta$};
    
  	\draw[->]
		(30:0.3) arc (30:330:0.3);
	\draw[rotate=-72,->]
		(120:3.6) to[out=10, in=170] (60:3.6);
	\foreach \x in {0,2}
	{
	\draw[shift={(198+72*\x:1.5)},->]
		(30:0.3) arc (30:330:0.3);
	\draw[shift={(18+72*\x:2.5)},->]
		(30:0.2) arc (30:330:0.2);
	}
	\foreach \x in {0,2,3}
	{
	\draw[shift={(-90+72*\x:1.5)},<-]
		(30:0.3) arc (30:330:0.3);
	\draw[shift={(90+72*\x:2.5)},<-]
		(30:0.2) arc (30:330:0.2);
	}
	
\end{scope}  
    

\begin{scope}[shift={(8cm,-7cm)}]

	\foreach \x in {1,...,5}
    \draw 
    		(-54+72*\x:1) -- (18+72*\x:1)
    		(-54+72*\x:1) -- (-54+72*\x:1.9) -- (-18+72*\x:2.5) -- (18+72*\x:1.9)
		(-18+72*\x:2.5) -- (-18+72*\x:3.5) -- (54+72*\x:3.5) -- (54+72*\x:2.5);

	\draw[line width=0.7mm]
		(234:1) -- (-54:1)
		(18:1) -- (18:1.9)
		(162:1) -- (162:1.9)
		(-18:2.5) -- (-18:3.5)
		(198:2.5) -- (198:3.5)
		(126:3.5) -- (54:3.5);

    \node at (310:3.6) {\small $T_4$};
   
    \node at (90:0.7) {\small $\alpha$};    
    \node at (162:0.7) {\small $\beta$};    
    \node at (234:0.7) {\small $\gamma$};
    \node at (18:0.7) {\small $\epsilon$};
    \node at (-54:0.7) {\small $\delta$};   
    \node at (26:1.7) {\small $\delta$};  
    \node at (10:1.7) {\small $\delta$};  
    \node at (150:1.1) {\small $\delta$};  
    \node at (174:1.1) {\small $\delta$};  
    \node at (-66:1.15) {\small $\delta$};       
    \node at (-22:3.1) {\small $\delta$};        
    \node at (-14:3.1) {\small $\delta$};        
    \node at (122:3.05) {\small $\delta$};       
    \node at (192:2.6) {\small $\delta$};     
    \node at (204:2.6) {\small $\delta$}; 
	\node at (126:3.7) {\small $\delta$};

	\draw[->]
		(30:0.3) arc (30:330:0.3);
	\draw[rotate=72,<-]
		(120:3.6) to[out=10, in=170] (60:3.6);
	\foreach \x in {0,1}
	{
	\draw[shift={(54+72*\x:1.5)},->]
		(30:0.3) arc (30:330:0.3);
	\draw[shift={(234+72*\x:2.5)},<-]
		(30:0.2) arc (30:330:0.2);
	}
	\foreach \x in {0,1,2}
	{
	\draw[shift={(198+72*\x:1.5)},<-]
		(30:0.3) arc (30:330:0.3);
	\draw[shift={(18+72*\x:2.5)},->]
		(30:0.2) arc (30:330:0.2);
	}

\end{scope}


\begin{scope}[shift={(4cm,-13.5cm)}]
    
	\draw
    		(234:1) -- (162:1) -- (90:1)
		(162:1) -- (162:1.9)
		(18:1.9) -- (54:2.5) -- (90:1.9)
		(-54:1) -- (-54:1.9)
		(-18:2.5) -- (-54:1.9) -- (270:2.5)
		(54:2.5) -- (54:3.5)
		(198:2.5) -- (198:3.5)
		(126:3.5) -- (198:3.5) -- (270:3.5);
	\draw[line width=0.7mm] 
    		(-54:1) -- (18:1) -- (90:1)
		(18:1) -- (18:1.9)
		(90:1.9) -- (126:2.5) -- (162:1.9)
		(234:1) -- (234:1.9)
		(198:2.5) -- (234:1.9) -- (270:2.5)
		(-18:2.5) -- (-18:3.5)
		(126:2.5) -- (126:3.5)
		(54:3.5) -- (-18:3.5) -- (270:3.5);
	\draw[dashed] 
    		(234:1) -- (-54:1)
		(90:1) -- (90:1.9)
		(18:1.9) -- (-18:2.5)
		(162:1.9) -- (198:2.5)
		(270:2.5) -- (270:3.5)
		(126:3.5) -- (54:3.5);
    
    \node at (310:3.6) {\small $T_5$};

    \node at (18:0.7) {\small $\alpha$};    
    \node at (162:0.7) {\small $\alpha$};    
    \node at (234:0.7) {\small $\beta$};
    \node at (-54:0.7) {\small $\gamma$};     
    \node at (90:0.7) {\small $\delta$};     
    \node at (26:1.7) {\small $\delta$};   
    \node at (154:1.7) {\small $\delta$};   
    \node at (-42:1.1) {\small $\delta$};   
    \node at (222:1.1) {\small $\delta$};     
    \node at (-90:2.2) {\small $\delta$};      
    \node at (90:2.1) {\small $\delta$};       
    \node at (51:3.1) {\small $\delta$};       
    \node at (130:3.1) {\small $\delta$};      
    \node at (-24:2.6) {\small $\delta$};      
    \node at (204:2.6) {\small $\delta$}; 
	\node at (-90:3.7) {\small $\delta$};

	\draw[->]
		(30:0.3) arc (30:330:0.3);
	\draw[rotate=144,->]
		(120:3.6) to[out=10, in=170] (60:3.6);
	\foreach \x in {0,...,4}
	{
	\draw[shift={(54+72*\x:1.5)},->]
		(30:0.3) arc (30:330:0.3);
	\draw[shift={(90+72*\x:2.5)},->]
		(30:0.2) arc (30:330:0.2);
	}
    
    \end{scope}

    \end{tikzpicture}
\caption{All the possible spherical tilings by $12$ congruent pentagons.}
\label{completeclassify}
\end{figure}

Here is how to read the tilings in Figure \ref{completeclassify}: Combinatorially, the tilings are the dodecahedron, and are illustrated on the flat plane after punching a hole in one tile. The center tile shows the angles in a tile. The angles in the other tiles are determined by the location of $\delta$ and the angle arrangement orientation. As for the edge lengths, the usual solid lines have the same edge length $a$, the thick lines have the same length $b$, and the dotted lines have the same length $c$.

Although the classification contains five classes, only the fifth class is the essential one, because the other four classes only provide isolated examples (and perhaps the regular dodecahedron only). Figure \ref{generalclass} is this essential class in its full glory.

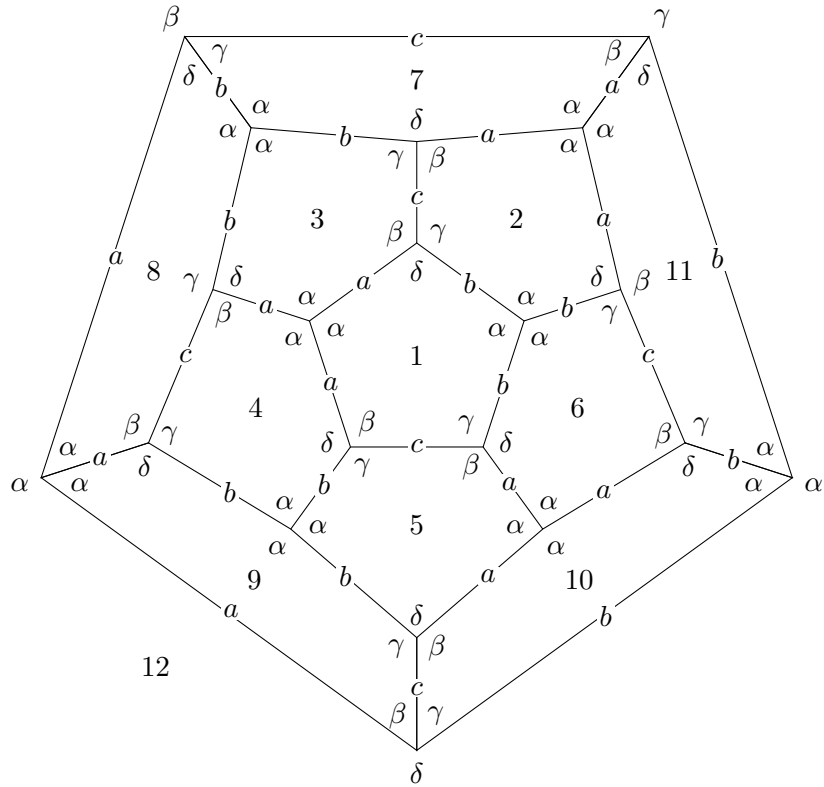
\begin{figure}[htp]
\centering
    \begin{tikzpicture}[>=latex, scale=1.5]
    
    \foreach \x in {1,...,5}
    \draw 
    		(-54+72*\x:1) -- (18+72*\x:1)
    		(-54+72*\x:1) -- (-54+72*\x:1.9) -- (-18+72*\x:2.5) -- (18+72*\x:1.9)
		(-18+72*\x:2.5) -- (-18+72*\x:3.5) -- (54+72*\x:3.5) -- (54+72*\x:2.5);

    \node at (0:0) {\small $1$};
    \node at (54:1.5) {\small $2$};
    \node at (126:1.5) {\small $3$};
    \node at (198:1.5) {\small $4$};
    \node at (-90:1.5) {\small $5$};
    \node at (-18:1.5) {\small $6$};
    \node at (90:2.45) {\small $7$};
    \node at (162:2.45) {\small $8$};
    \node at (234:2.45) {\small $9$};
    \node at (-54:2.45) {\small $10$};
    \node at (18:2.45) {\small $11$};
    \node at (230:3.6) {\small $12$};


	\node[fill=white,inner sep=1] at (54:0.8) {\small $b$};
	\node[fill=white,inner sep=1] at (126:0.8) {\small $a$};
	\node[fill=white,inner sep=1] at (198:0.8) {\small $a$};
	\node[fill=white,inner sep=1] at (-90:0.8) {\small $c$};
	\node[fill=white,inner sep=1] at (-18:0.8) {\small $b$};
	

	\node[fill=white,inner sep=1] at (18:1.4) {\small $b$};
	\node[fill=white,inner sep=1] at (90:1.4) {\small $c$};
	\node[fill=white,inner sep=1] at (162:1.4) {\small $a$};
	\node[fill=white,inner sep=1] at (234:1.4) {\small $b$};
	\node[fill=white,inner sep=1] at (-54:1.4) {\small $a$};
	

	\node[fill=white,inner sep=1] at (0:2.05) {\small $c$};
	\node[fill=white,inner sep=1] at (36:2.05) {\small $a$};
	\node[fill=white,inner sep=1] at (72:2.05) {\small $a$};
	\node[fill=white,inner sep=1] at (108:2.05) {\small $b$};
	\node[fill=white,inner sep=1] at (144:2.05) {\small $b$};
	\node[fill=white,inner sep=1] at (180:2.05) {\small $c$};
	\node[fill=white,inner sep=1] at (-144:2.05) {\small $b$};
	\node[fill=white,inner sep=1] at (-108:2.05) {\small $b$};
	\node[fill=white,inner sep=1] at (-72:2.05) {\small $a$};
	\node[fill=white,inner sep=1] at (-36:2.05) {\small $a$};
	

	\node[fill=white,inner sep=1] at (54:2.95) {\small $a$};
	\node[fill=white,inner sep=1] at (126:2.95) {\small $b$};
	\node[fill=white,inner sep=1] at (198:2.95) {\small $a$};
	\node[fill=white,inner sep=1] at (-90:2.95) {\small $c$};
	\node[fill=white,inner sep=1] at (-18:2.95) {\small $b$};

	
	\node[fill=white,inner sep=1] at (18:2.8) {\small $b$};
	\node[fill=white,inner sep=1] at (90:2.8) {\small $c$};
	\node[fill=white,inner sep=1] at (162:2.8) {\small $a$};
	\node[fill=white,inner sep=1] at (234:2.8) {\small $a$};
	\node[fill=white,inner sep=1] at (-54:2.85) {\small $b$};

    \node at (18:0.75) {\small $\alpha$};
    \node at (8:1.1) {\small $\alpha$};
    \node at (28:1.1) {\small $\alpha$};
    
    \node at (90:0.75) {\small $\delta$};
    \node at (80:1.1) {\small $\gamma$};
    \node at (100:1.1) {\small $\beta$};
    
    \node at (162:0.75) {\small $\alpha$};
    \node at (152:1.1) {\small $\alpha$};
    \node at (172:1.1) {\small $\alpha$};
    
    \node at (234:0.75) {\small $\beta$};
    \node at (224:1.1) {\small $\delta$};
    \node at (244:1.1) {\small $\gamma$};
    
    \node at (-54:0.75) {\small $\gamma$};
    \node at (-64:1.1) {\small $\beta$};
    \node at (-44:1.1) {\small $\delta$};

    \node at (18:2.1) {\small $\beta$};
    \node at (12:1.75) {\small $\gamma$};
    \node at (24:1.75) {\small $\delta$};
    
    \node at (90:2.1) {\small $\delta$};
    \node at (84:1.75) {\small $\beta$};
    \node at (96:1.75) {\small $\gamma$};
    
    \node at (162:2.1) {\small $\gamma$};
    \node at (156:1.75) {\small $\delta$};
    \node at (168:1.75) {\small $\beta$};
    
    \node at (234:2.1) {\small $\alpha$};
    \node at (228:1.75) {\small $\alpha$};
    \node at (240:1.75) {\small $\alpha$};
    
    \node at (-54:2.1) {\small $\alpha$};
    \node at (-48:1.75) {\small $\alpha$};
    \node at (-60:1.75) {\small $\alpha$};

    \node at (54:2.3) {\small $\alpha$};
    \node at (50:2.6) {\small $\alpha$};
    \node at (58:2.6) {\small $\alpha$};
    
    \node at (126:2.3) {\small $\alpha$};
    \node at (122:2.6) {\small $\alpha$};
    \node at (130:2.6) {\small $\alpha$};
    
    \node at (198:2.3) {\small $\gamma$};
    \node at (194:2.6) {\small $\beta$};
    \node at (202:2.6) {\small $\delta$};

    \node at (-90:2.3) {\small $\delta$};
    \node at (-94:2.6) {\small $\gamma$};
    \node at (-86:2.6) {\small $\beta$};
    
    \node at (-18:2.3) {\small $\beta$};
    \node at (-22:2.6) {\small $\delta$};
    \node at (-14:2.6) {\small $\gamma$};

    \node at (54:3.7) {\small $\gamma$};
    \node at (51:3.2) {\small $\delta$};
    \node at (57:3.2) {\small $\beta$};

    \node at (126:3.7) {\small $\beta$};
    \node at (123:3.2) {\small $\gamma$};
    \node at (129:3.2) {\small $\delta$};

    \node at (198:3.7) {\small $\alpha$};
    \node at (195:3.2) {\small $\alpha$};
    \node at (201:3.2) {\small $\alpha$};

    \node at (-90:3.7) {\small $\delta$};
    \node at (-93:3.2) {\small $\beta$};
    \node at (-87:3.2) {\small $\gamma$};

    \node at (-18:3.7) {\small $\alpha$};
    \node at (-15:3.2) {\small $\alpha$};
    \node at (-21:3.2) {\small $\alpha$};

    \end{tikzpicture}
\caption{The general class $T_5$ of spherical tilings by $12$ congruent pentagons.}
\label{generalclass}
\end{figure}

The tile in the ``essential class'' $T_5$ is illustrated on the left of Figure \ref{pentagontile}. The angle sum condition is
\[
3\alpha=\beta+\gamma+\delta=2\pi.
\]
This implies that the area of the pentagonal tile is $\frac{\pi}{3}$, or one twelfth of the area of the sphere.

It is easy to see that the tile in $T_5$ allows two free parameters. Divide the tile as in the middle of Figure \ref{pentagontile}. Since $\alpha=\frac{2\pi}{3}$ is already fixed, the area $A(a)$ of the triangle $\triangle ABC$ is completely determined by $a$. The function $A(a)$ is strictly increasing, and $A(b)$ is the area of $\triangle ADE$. When $A(a)+A(b)\le \frac{\pi}{3}$ and $A(a)+A(b)$ is sufficiently close to $\frac{\pi}{3}$, we can always find a suitable angle $\angle CAD$ (equivalent to finding suitable $\delta=\angle BAE$) so that the area of the pentagon is $\frac{\pi}{3}$. This shows that we can freely choose $a$ and $b$ within certain range to get a tile that fits into the tiling, and the tile is completely determined by the choice.

\begin{figure}[htp]
\centering
    \begin{tikzpicture}[scale=2]

	\draw
		(0,0) -- node[fill=white,inner sep=2] {\small $c$} (0.735,0)
		(0,0) -- node[fill=white,inner sep=2] {\small $a$}  ++(100:1) -- node[fill=white,inner sep=1] {\small $a$} ++(100-72:1)
		(0.735,0) -- node[fill=white,inner sep=1] {\small $b$} ++(55:0.9) -- node[fill=white,inner sep=1] {\small $b$} ++(55+72:0.9);

    \node at (-0.05,0.95) {\small $\alpha$};
    \node at (0.67,1.3) {\small $\delta$};
    \node at (1.1,0.75) {\small $\alpha$};
    \node at (0.1,0.1) {\small $\beta$};
    \node at (0.7,0.1) {\small $\gamma$};

	\begin{scope}[xshift=2cm]
	
	\draw
		(0,0) -- (0.735,0)
		(0,0) node[below] {\small $C$} -- node[fill=white,inner sep=2] {\small $a$}  ++(100:1) node[left] {\small $B$} -- node[fill=white,inner sep=1] {\small $a$} ++(100-72:1) node[above] {\small $A$} -- cycle
		(0.735,0) node[below] {\small $D$} -- node[fill=white,inner sep=1] {\small $b$} ++(55:0.9) node[right] {\small $E$} -- node[fill=white,inner sep=1] {\small $b$} ++(55+72:0.9) -- cycle;
	
	\node at (-0.05,0.95) {\small $\alpha$};
    \node at (1.1,0.75) {\small $\alpha$};

	\end{scope}

	\begin{scope}[shift={(4.5cm,1.1cm)}]
	
	\draw
		(0,0) node[above] {\small $A$}
		-- node[fill=white,inner sep=2] {\small $a$}
		++(20:0.8) node[below] {\small $\alpha$} 
			node[above] {\small $B$}  
		-- node[fill=white,inner sep=2] {\small $a$} 
		++(-50:0.8) node[left] {\small $\epsilon$} 
			node[right] {\small $C$}
		-- node[fill=white,inner sep=2] {\small $a$}
		++(-110:0.8) node[above left=-1] {\small $\delta$} 
			node[below] {\small $D$}
		-- node[fill=white,inner sep=2] {\small $a$}
		++(-180:0.8) node[above right=-1] {\small $\gamma$} 
			node[below] {\small $E$}
		-- ++(110:1.3)  
			node[left] {\small $A'$};

	\end{scope}

    \end{tikzpicture}
\caption{The pentagonal tile.}
\label{pentagontile}
\end{figure}
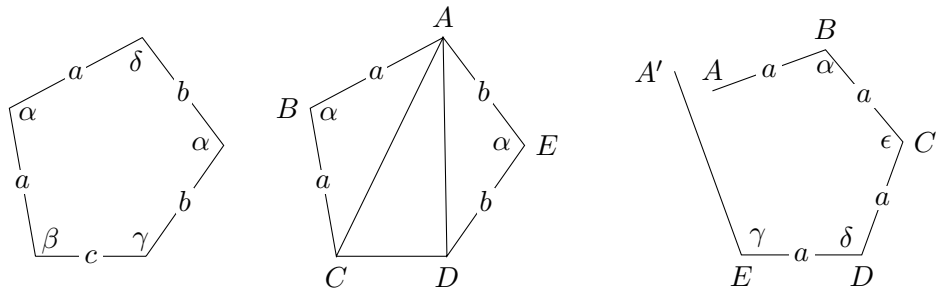

Suppose we put the cube in Figure \ref{example} inside a sphere, so that the cube and the sphere have the same center. If we put a light source at the center, then the projection of the cube to the sphere is a $T_5$ class tiling with $\beta+\gamma=\delta=\pi$.

\begin{figure}[htp]
\centering
    \begin{tikzpicture}[scale=1]
	
	\draw
		(3,4) -- (1,6) -- (-2,5) -- (-2,2) -- (0,0) -- (3,1)  -- (3,4) -- (0,3) -- (-2,5)
		(0,0) -- (0,3) node[below right] {\small $\alpha$}
		(0,1) -- (3,3) 
		(-0.66,0.66) -- (-1.67,4.67)
		(2,3.66) node[below] {\small $\pi$} -- (-1,5.33);
		
	\node[fill=white,inner sep=1] at (0,2) {\small $a$};
	\node[fill=white,inner sep=1] at (1,3.33) {\small $a$};
	\node[fill=white,inner sep=1] at (2.5,3.85) {\small $b$};
	\node[fill=white,inner sep=1] at (3,3.45) {\small $b$};
	\node[fill=white,inner sep=1] at (1.5,2) {\small $c$};
	\node at (2.8,3.7) {\small $\alpha$};
	\node at (0.25,1.4) {\small $\beta$};
	\node at (2.75,3.1) {\small $\gamma$};

    \end{tikzpicture}
\caption{A spherical pentagon tiling.}
\label{example}
\end{figure}
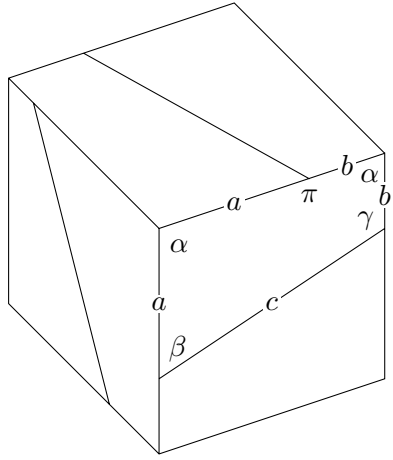

For the tiling class $T_2$, the angle sum condition is $3\alpha=2\beta+\epsilon=2\gamma+\delta=\alpha+\delta+\epsilon=2\pi$. This means
\[
\alpha=\frac{2\pi}{3},\quad
\beta=\pi-\frac{1}{2}\epsilon,\quad
\gamma=\frac{\pi}{3}+\frac{1}{2}\epsilon,\quad
\delta=\frac{4\pi}{3}-\epsilon,
\]
and should lead to only isolated examples. The reason is illustrated on the right of Figure \ref{pentagontile}. Suppose $a$ and $\epsilon$ are given. Fix a point $A$ on the sphere and a direction for the edge $AB$. Starting from $A$ and travelling along the given direction by distance $a$, we arrive at $B$. Turning an angle $\alpha$ at $B$ and travelling by distance $a$, we arrive at $C$. Turning an angle $\epsilon$ at $C$ and travelling by distance $a$, we arrive at $D$. Turning an angle $\delta=\frac{4\pi}{3}-\epsilon$ at $D$ and travelling by distance $a$, we arrive at $E$. Turning an angle $\gamma=\frac{\pi}{3}+\frac{1}{2}\epsilon$ at $E$ and travelling along, we get a line (actually a great arc) $EA'$. Then we want $A$ to lie on the line $EA'$, which imposes one relation between $a$ and $\epsilon$, and reduces the number of free parameters to one. Once $A$ lies on $EA'$, we further need the pentagon to have area $\frac{\pi}{3}$, which is the same as $\angle EAB=\beta=\pi-\frac{1}{2}\epsilon$. This imposes one more relation and we are left with no free parameter. Therefore the pentagonal tiles fitting into the tiling class $T_2$ must appear in isolated way. 

Similar discussion can be made for $T_3$ and $T_4$. For $T_1$, we have even one less degree of freedom. The problem is crystalized below (the angles are renamed in four cases in order to get the same expressions for the angle sum equations).

\medskip

\noindent {\bf Problem.} Find spherical pentagons in Figure \ref{regularpentagon} satisfying $3\alpha=2\beta+\gamma=2\delta+\epsilon=\alpha+\gamma+\epsilon=2\pi$. Are there any examples besides the pentagon in the regular dodecahedron?

\medskip

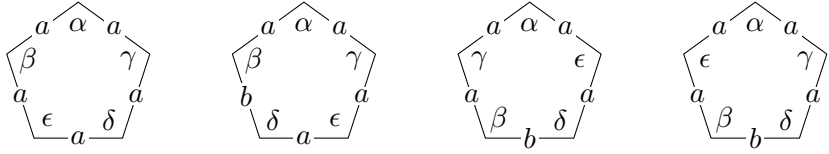
\begin{figure}[htp]
\centering
    \begin{tikzpicture}[scale=1]


	\draw (18:1) -- node[fill =white,inner sep=1] {\small $a$}
		(90:1) -- node[fill=white,inner sep=1] {\small $a$}
		(162:1) -- node[fill=white,inner sep=1] {\small $a$}
		(234:1) -- node[fill=white,inner sep=1] {\small $a$}
		(306:1) -- node[fill=white,inner sep=1] {\small $a$} (18:1);

    \node at (18:0.7) {\small $\gamma$};
    \node at (90:0.7) {\small $\alpha$};
    \node at (162:0.7) {\small $\beta$};
    \node at (234:0.7) {\small $\epsilon$};
    \node at (306:0.7) {\small $\delta$};

       
    \begin{scope}[xshift=3cm] 

	\draw (18:1) -- node[fill =white,inner sep=1] {\small $a$}
		(90:1) -- node[fill=white,inner sep=1] {\small $a$}
		(162:1) -- node[fill=white,inner sep=1] {\small $b$}
		(234:1) -- node[fill=white,inner sep=1] {\small $a$}
		(306:1) -- node[fill=white,inner sep=1] {\small $a$} (18:1);
    
    \node at (18:0.7) {\small $\gamma$};
    \node at (90:0.7) {\small $\alpha$};
    \node at (162:0.7) {\small $\beta$};
    \node at (234:0.7) {\small $\delta$};
    \node at (306:0.7) {\small $\epsilon$};
    
    \end{scope}

       
    \begin{scope}[xshift=6cm] 

	\draw (18:1) -- node[fill =white,inner sep=1] {\small $a$}
		(90:1) -- node[fill=white,inner sep=1] {\small $a$}
		(162:1) -- node[fill=white,inner sep=1] {\small $a$}
		(234:1) -- node[fill=white,inner sep=1] {\small $b$}
		(306:1) -- node[fill=white,inner sep=1] {\small $a$} (18:1);
    
    \node at (18:0.7) {\small $\epsilon$};
    \node at (90:0.7) {\small $\alpha$};
    \node at (162:0.7) {\small $\gamma$};
    \node at (234:0.7) {\small $\beta$};
    \node at (306:0.7) {\small $\delta$};
    
  	\end{scope}

       
    \begin{scope}[xshift=9cm] 

	\draw (18:1) -- node[fill =white,inner sep=1] {\small $a$}
		(90:1) -- node[fill=white,inner sep=1] {\small $a$}
		(162:1) -- node[fill=white,inner sep=1] {\small $a$}
		(234:1) -- node[fill=white,inner sep=1] {\small $b$}
		(306:1) -- node[fill=white,inner sep=1] {\small $a$} (18:1);
    
    \node at (18:0.7) {\small $\gamma$};
    \node at (90:0.7) {\small $\alpha$};
    \node at (162:0.7) {\small $\epsilon$};
    \node at (234:0.7) {\small $\beta$};
    \node at (306:0.7) {\small $\delta$};
    
    \end{scope}

    \end{tikzpicture}
\caption{Must these spherical pentagons be regular?}
\label{regularpentagon}
\end{figure}

To prove the main theorem, we first separate and classify different aspects of tilings. At the most crude level, we study tilings by combinatorially congruent tiles in Section \ref{combinatorial}, which means that all tiles have the same number of edges, and edge lengths and angles are ignored. Proposition \ref{tile_pattern} says that the dodecahedron is the unique spherical tiling by $12$ combinatorially congruent pentagons. Combinatorial tilings have been extensively studied. See \cite{s2} for a recent survey. Our study is actually a topological one, in the sense that the lines do not even need to be straight. See \cite{yan1} for our further work for the non-minimal case.

In Sections \ref{edgetile1} and \ref{edgetile2}, we study tilings by edge congruent tiles, which means that all tiles have the same edge length combination and arrangement. Propositions \ref{edge_pattern2}, \ref{edge_pattern3}, \ref{edge_pattern5}, \ref{edge_pattern4} (plus the case all edges have the same length) completely classify spherical tilings by $12$ edge congruent pentagons. Our further work in the non-minimal case can be found in \cite{ccy1}.

In Sections \ref{angletile1} and \ref{angletile2}, we study tilings by angle congruent tiles, which means that all tiles have the same angle combination and arrangement. Proposition \ref{angle} classifies the numerics in spherical tilings by $12$ angle congruent pentagons. Propositions \ref{angle_pattern3}, \ref{angle_pattern4}, \ref{angle_pattern5} then further classify the locations of the angles in three major cases. The complete classification would involve another major but rather complicated case. Since this case will not be needed in the proof of the main classification theorem, we only present some examples instead of the whole proof.

In Section \ref{proofmain}, we combine the edge congruent and angle congruent classifications to prove the main theorem. While our edge congruent and angle congruent results do not require geometry, in that the edges are not necessarily great arcs, in combining the two, we use geometry to cut down the possible number of edge-angle combinations. So our main theorem does require the edges to be great arcs. For the non-minimal case, see \cite{ccy2} for further partial results.

Finally, we would like to thank the referees for the careful reading of the paper and many helpful suggestions.

\section{Combinatorial Conditions}
\label{combinatorial}

In this section, we study the combinatorial aspect of tilings. This means that we ignore the edge lengths and the angles, and only require the tiles to have the same number of edges. Most results of the section are well known.

The results certainly apply to the edge-to-edge tilings as defined in \cite{gs}. However, some results can be applied to more general tilings. Specifically, we consider a connected graph nicely embedded in the sphere. We assume that all vertices in the graph have degree $\ge 3$. The interior of a tile is a connected component of the complement of the graph. A tile is the closure of its interior. The boundary of a tile is what is between the tile and its interior, and is a union of some edges in the graph.

The exclusion of vertices of degree $2$ means that we do not view an $n$-gon as an $(n+k)$-gon by artificially choosing another $k$ points on the boundary as extra vertices. Therefore all tiles have naturally defined vertices and edges that are already given in the graph, and the tiling is naturally edge-to-edge. 

Figure \ref{tile} shows some possible tiles. The first one does not have any boundary points identified. The others have some boundary points (even boundary intervals) identified. In case some boundary intervals are identified, the interior of the tile is different from the topological version of the interior.

\begin{figure}[htp]
\centering
    \begin{tikzpicture}[scale=1]

    \filldraw[fill=gray!30]
    		(-1.2,0) to[out=80,in=200] (0,1.2)
			to[out=-50,in=180] (1.2,0.5)
			to[out=-90,in=50] (0.9,-1)
			to[out=150,in=40] (-0.7,-1)
			to[out=90,in=-20] (-1.2,0);

    \begin{scope}[xshift=3cm]
    
    \filldraw[fill=gray!30]
		(1,0) to[out=30,in=0] (0.3,1.2)
			to[out=180,in=90] (-1.2,0)
			to[out=-90,in=180] (0.3,-1.2)
			to[out=0,in=-30] (1,0);
	\filldraw[fill=white]
		(0.5,0) circle (0.5);
	
	\end{scope}

	\begin{scope}[xshift=6cm]
    
    \filldraw[fill=gray!30]
		(1,0) to[out=30,in=0] (0.3,1.2)
			to[out=180,in=90] (-1.2,0)
			to[out=-90,in=180] (0.3,-1.2)
			to[out=0,in=-30] (1,0);
	\filldraw[fill=white]
		(0,0) circle (0.5);
	\draw
		(0.5,0) -- (1,0);
	
	\end{scope}

	\begin{scope}[xshift=9cm]
    
    \filldraw[fill=gray!30]
		(1,0) to[out=30,in=0] (0.3,1.2)
			to[out=180,in=90] (-1.2,0)
			to[out=-90,in=180] (0.3,-1.2)
			to[out=0,in=-30] (1,0);
	\filldraw[fill=white]
		(0.6,0) circle (0.4)
		(-0.5,0) circle (0.4);
	\draw
		(-0.1,0) -- (0.2,0);
	
	\end{scope}

    \end{tikzpicture}
\caption{Tiles.}
\label{tile}
\end{figure}

Figure \ref{notile} shows some examples that are not tiles. They are unions of several tiles at finitely many boundary points.

\begin{figure}[htp]
\centering
    \begin{tikzpicture}[scale=1]

    \filldraw[fill=gray!30]
    		(0,0) to[out=160,in=-70] (-1,1.2)
			to[out=-110,in=20] (-2,0)
			to[out=-20,in=110] (-1,-1.2)
			to[out=70,in=200] (0,0)
		(1,0) ellipse (1 and 1.2);
	\filldraw[rounded corners, fill=white]
		(0.7,0) ellipse (0.7 and 0.5);

    \begin{scope}[xshift=4.5cm]
    
    \filldraw[fill=gray!30]
		(0,0) ellipse (2 and 1.2);
	\filldraw[fill=white]
		(0,0) ellipse (0.5 and 1.2)
		(1.1,0) circle (0.4);
	\draw 
		(0.5,0) -- (0.7,0);
	
	\end{scope}

	\begin{scope}[xshift=9cm]
    
    \filldraw[fill=gray!30]
		(0,0) ellipse (2 and 1.2);
	\filldraw[fill=white]
		(0,1.2) to[out=-50,in=50] (0,-1.2)
			to[out=120,in=-20] (-2,0)
			to[out=20,in=240] (0,1.2);
	
	\end{scope}

    \end{tikzpicture}
\caption{Not tiles.}
\label{notile}
\end{figure}
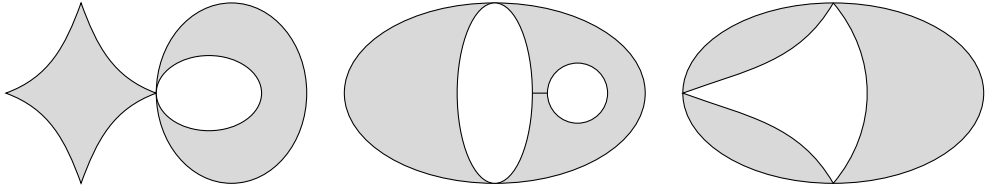

\begin{lemma}\label{noself}
In any (naturally generated) spherical tiling, there is at least one tile for which no boundary points are identified.
\end{lemma}

\begin{proof}
We call a closed subset of the sphere a {\em simple region} if it is enclosed by a nicely embedded simple closed curve. The boundary points of a simple region are not identified. The lemma basically says that some tile is a simple region.

It is a topological fact that if $R$ is the whole sphere or a simple region, and $P$ is a tile in $R$ with some boundary points identified, then the closure of some connected component of $R-P$ is a simple region. 

Let $P_1$ be a tile, such that some boundary points are identified. Then the complement $S^2-P_1$ has a connected component $C_1$, such that the closure $\bar{C}_1$ is a simple region, and is actually a union of tiles. If $\bar{C}_1$ is a tile, then we are done. Otherwise, $\bar{C}_1$ contains a tile $P_2$ as a proper subset. If $P_2$ is a simple region, then we are also done. If $P_2$ is not a simple region, then some boundary points of $P_2$ are identified, and the complement $\bar{C}_1-P_2$ has a connected component $C_2$, such that the closure $\bar{C}_2$ is a simple region. Keep going. As long as we do not encounter tiles that are simple regions along the way, we find a strictly decreasing sequence of simple regions $\bar{C}_1\supset \bar{C}_2\supset \bar{C}_3\supset \dotsb$. Since there are only finitely many tiles in the tiling, the sequence must stop. This means that some tile found in the process must be a simple region. 
\end{proof}

An immediate consequence of the lemma is that in case all tiles are congruent (full congruence, including edge lengths and angles), the boundary points of any tile are never identified. Since the property is needed in the subsequent discussion, we include it in the following definition.

\begin{definition}
A tiling is {\em combinatorially monohedral} if all tiles have the same number of edges and every tile has no boundary points identified. The tiles in such a tiling are {\em combinatorially congruent}.
\end{definition}

Let $f$, $e$, $v$ be the numbers of tiles (or faces), edges and vertices in a combinatorially monohedral tiling. Let $n$ be the number of edges in a tile. Then we have the Euler equation and the Dehn-Sommerville equation
\[
v - e + f = 2, \quad
2e = nf.
\]
Note that the Dehn-Sommerville equation assumes that each edge is shared by two {\em different} faces, which follows from the condition that no boundary points of a tile are identified. (Actually we only need that no boundary intervals are identified.) Moreover, if $v_i$ is the number of vertices of degree $i$, then
\[
v = v_3 + v_4 + v_5 + \dotsb, \quad
2e = 3v_3+4v_4+5v_5 +\dotsb. 
\]
It is a simple consequence of all these equations that $n=3,4,5$. See Proposition 4 of \cite{paper_quadr}.

Now we concentrate on the case $n=5$. The equations become
\begin{equation}\label{pent_eq}
2e = 5f, \quad
2v = 3f + 4, \quad
v_3 - 20 = 2v_4+5v_5+8v_6 +\dotsb. 
\end{equation}

\begin{lemma}\label{deg3_4}
Any spherical tiling by combinatorially congruent pentagons contains a tile in which at least four vertices have degree $3$.
\end{lemma}

Figure \ref{highdeg} gives a combinatorial pentagon tiling of the sphere in which each tile has a degree $4$ vertex.

\begin{figure}[htp]
\begin{center}
\begin{tikzpicture}[>=latex,scale=1.2]

\foreach \x in {0,1,2,3}
\draw[scale=1.07,rotate=90*\x]
	(0,0) -- (0.3,0) -- (0.4,0.2) -- (0.2,0.4) -- (0,0.3)
	(0.4,0.2) -- (0.6,0.2) -- (0.8,-0.4) -- (0.4,-0.2)
	(0.6,0.2) -- (0.8,0.4) -- (0.4,0.8)
	(0.8,0.4) -- (1,0.4) -- (1,-0.4) -- (0.8,-0.4)
	(1,0.4) -- (1.1,1.1) -- (0.4,1)
	(1.1,1.1) -- (1.4,1.4);

\end{tikzpicture}
\end{center}
\caption{Every tile has a vertex of degree $>3$.}
\label{highdeg}
\end{figure}
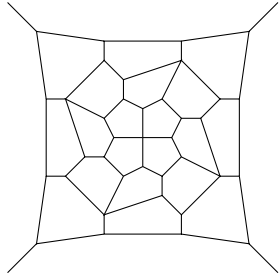

\begin{proof}
Let $t_j$ be the number of degree $3$ vertices in the $j$-th tile. Then $3v_3=\sum t_j$. On the other hand, we have 
\begin{align*}
3f=2v-4
&=2v_3+2v_4+2v_5+2v_6+\dotsb-4 \\
&\le 2v_3+2v_4+5v_5+8v_6+\dotsb-4 \\
&= 2v_3 + (v_3-20) - 4 
<3v_3=\sum t_j.
\end{align*}
This implies that some $t_i>3$.
\end{proof}

\begin{lemma}\label{vertex3}
A spherical tiling by combinatorially congruent pentagons has at least $12$ tiles. Moreover, the minimum $12$ is reached if and only if all vertices have degree $3$, and if and only if the number of vertices is $20$.
\end{lemma}

\begin{proof}
The third equation in \eqref{pent_eq} implies that $v\ge v_3\ge 20$, and $v=20$ if and only if all vertices have degree $3$. On the other hand, by the second equation, $v\ge 20$ is the same as $f\ge 12$, and $v=20$ is the same as $f=12$.
\end{proof}

\begin{proposition}\label{tile_pattern}
The spherical tiling by $12$ combinatorially congruent pentagons is uniquely given by Figure \ref{tile_pattern_pic}. 
\end{proposition}

\begin{figure}[htp]
\centering
    \begin{tikzpicture}[scale=0.7]

    \foreach \x in {1,...,5}
      \draw (-54+\x*72:1) -- (18+\x*72:1) -- (18+\x*72:1.75) -- (-18+\x*72:2.4) -- (-54+\x*72:1.75) -- cycle
      (-18+\x*72:2.4) -- (-18+\x*72:3.25) -- (-90+\x*72:3.25) -- (-90+\x*72:2.4);

    \node at(0:0){\small $1$};
    \foreach \y in {2,...,6}
      \node at (-90+\y*72:1.4){\small $\y$};
    \foreach \z in {7,...,11}
      \node at (-54+\z*72:2.2){\small $\z$};
      
     \node at (240:3.5){\small $12$};

    \end{tikzpicture}
\caption{Combinatorial structure of the minimal pentagon tiling.}
\label{tile_pattern_pic}
\end{figure}
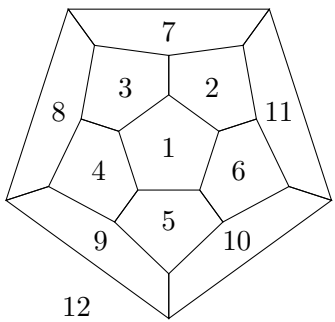

\begin{proof}
By Lemma \ref{vertex3}, all vertices have degree $3$.

We start with a tile $P_1$. Each of the five edges of $P_1$ is shared by $P_1$ and exactly another tile. We denote these tiles by $P_2$, $P_3$, $P_4$, $P_5$, $P_6$ and assume that they are arranged in counterclockwise order. 

Consider the vertex $A$ of $P_1$ on the left of Figure \ref{deg3fit}. It is also a vertex of $P_2$ and $P_3$. Since $A$ has degree $3$, besides the two edges of $P_1$ connected to $A$, there is exactly one more edge $x$ also connected to $A$. The three edges divide the neighborhood of $A$ into three corners, each belonging to one tile. This local picture implies that $x$ is a common edge of $P_2$ and $P_3$. The similar situation happens between $P_3$ and $P_4$, between $P_4$ and $P_5$, between $P_5$ and $P_6$, and between $P_6$ and $P_2$. Therefore the first six tiles are glued together in the way described in Figure \ref{tile_pattern_pic}. Of course this does not yet exclude the possibility that additional identification may happen among these tiles. But at least the relations described in Figure \ref{tile_pattern_pic} already exist between these tiles.

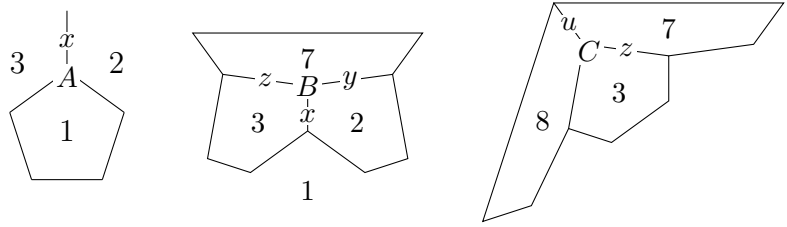
\begin{figure}[htp]
\centering
\begin{tikzpicture}[scale=0.8]

    \foreach \x in {1,...,5}
	\draw (18+72*\x:1) -- (90+72*\x:1);
	\draw (90:1) -- node[fill=white,inner sep=1] {\small $x$} (90:2);

	\node at (0:0) {$1$};
	\node at (54:1.4) {$2$};
	\node at (126:1.4) {$3$};	
	\node[fill=white,inner sep=1] at (90:0.9) {\small $A$};
	
\begin{scope}[shift={(4cm,-1cm)}]
	
	\draw (90:1) -- 
      	(90:1.75) -- node[fill=white,inner sep=1] {\small $y$}
        (54:2.4) -- (18:1.75) -- (18:1) -- cycle;
    \draw (90:1.75) -- node[fill=white,inner sep=1] {\small $z$}
      	(126:2.4) -- (162:1.75) -- (162:1) -- (90:1);
    \draw (54:2.4) -- (54:3.25) -- (126:3.25) -- (126:2.4);

    \node at (0:0) {\small $1$};
    \node at (54:1.4) {\small $2$};
    \node at (126:1.4) {\small $3$};
    \node at (90:2.2) {\small $7$};
	\node[fill=white,inner sep=1] at (90:1.75) {\small $B$};
	\node[fill=white,inner sep=1] at (90:1.25) {\small $x$};

\end{scope}

\begin{scope}[shift={(10cm,-0.5cm)}]

	\draw (90:1.75) -- node[fill=white,inner sep=1] {\small $z$} (126:2.4) -- (162:1.75) -- (162:1) -- (90:1) -- cycle;
    \draw (90:1.75) -- (54:2.4) -- (54:3.25) -- (126:3.25);
    \draw (126:2.4) -- node[fill=white,inner sep=1] {\small $u$} (126:3.25) -- (198:3.25) -- (198:2.4) -- (162:1.75);

    \node at (126:1.4) {\small $3$};
    \node at (90:2.2) {\small $7$};
    \node at (162:2.2) {\small $8$};
	\node[fill=white,inner sep=1] at (126:2.25) {\small $C$};
	  
\end{scope}	  

\end{tikzpicture}
\caption{Neighborhood of degree $3$ vertex.}
\label{deg3fit}
\end{figure}

The edge $x$ has another end vertex $B$ besides $A$. The three edges connected to $B$ are $x$, $y$, $z$, where $y$ is an edge of $P_2$ adjacent to $x$ and $z$ is an edge of $P_3$ adjacent to $x$. See middle of Figure \ref{deg3fit}. Now $y$ is shared by $P_2$ and another tile $P_7$, and $z$ is shared by $P_3$ and another tile $P_7'$. The three edges $x,y,z$ divide the neighborhood of $B$ into three corners, each belonging to one tile. Since two corners already belong to $P_2$ and $P_3$, the third corner belongs to the same tile. This implies that $P_7=P_7'$. Thus the $7$-th tile is established and is related to $P_2$ and $P_3$ as in Figure \ref{tile_pattern_pic}. Similarly, $P_8$ can be defined from $P_3$ and $P_4$, $P_9$ can be defined from $P_4$ and $P_5$, $P_{10}$ can be defined from $P_5$ and $P_6$, and $P_{11}$ can be defined from $P_6$ and $P_2$. 

The edge $z$ has another end vertex $C$ besides $B$. It is also a vertex of $P_7$ and $P_8$. There are three edges connected to $C$. Two are edges of $P_3$ and the third one is $u$ on the right of Figure \ref{deg3fit}. The situation is the same as the left of Figure \ref{deg3fit}, with $P_3$, $P_7$, $P_8$, $C$, $u$ in place of $P_1$, $P_2$, $P_3$, $A$, $x$. We get the similar conclusion that the edge $u$ is shared by $P_7$ and $P_8$. We can deduce similar relations between $P_8$ and $P_9$, etc. Then the tiles $P_1$ through $P_{11}$ are glued together as in Figure \ref{tile_pattern_pic}. 

After constructing the first eleven tiles, we find that $P_7$, $P_8$, $P_9$, $P_{10}$, $P_{11}$ have one ``free'' edge each. The free edge of $P_7$ is shared by $P_7$ and another tile $P_{12}$. The free edge of $P_8$ is shared by $P_8$ and another tile $P_{12}'$. The situation is similar to the relation between $P_2$, $P_3$, $P_7$, $P_7'$ as decribed in the middle of Figure \ref{deg3fit}. We get the similar conclusion that $P_{12}=P_{12}'$. Therefore the five possible $12$-th tiles are all the same. 

Now we find twelve tiles that are related as in Figure \ref{tile_pattern_pic}. If there are more tiles, then the additional tiles have to be glued to the existing twelve along some vertices or edges. This will either increase the degree of some vertex to be $>3$, or cause some edge to be shared by more than two tiles. Therefore the twelve are all the tiles in the tiling. Moreover, if there are additional identifications among these twelve tiles, then the identification of vertices will increase the degree of some vertex to be $>3$, the identification of edges will cause some edge to be shared by more than two tiles, and the identification of tiles will reduce the total number of tiles to be $<12$. Since all these should not happen, there is no more additional identification. We conclude that Figure \ref{tile_pattern_pic} gives the complete description of the tiling.
\end{proof}

\section{Edge Congruent Tiling}
\label{edgetile1}

We study the distribution of edge lengths in a spherical tiling by congruent pentagons. Since angles are ignored, we introduce the following concept.

\begin{definition}
Two polygons are {\em edge congruent}, if there is a correspondence between the edges, such that the adjacencies of the edges are preserved and the edge lengths are preserved.
\end{definition}

Two spherical (or planar) triangles with great arc (or straight line) edges are congruent if and only if they are edge congruent. However, this is no longer true for quadrilaterals and pentagons. Moreover, our results actually do not even require the edges to be great arcs (or straight). 

\begin{proposition}\label{edge1}
Let $a,b,c,\dotsc$, denote distinct numbers. Then in a spherical tiling by edge congruent pentagons, the edge lengths of the tile must form one of the following five combinations:
\begin{align*}
    a^5 &=\{a,a,a,a,a\}, & a^4b &= \{a,a,a,a,b\}, && \\
    a^3b^2 &= \{a,a,a,b,b\}, & a^3bc &= \{a,a,a,b,c\}, & a^2b^2c &= \{a,a,b,b,c\}.
\end{align*}
\end{proposition}

\begin{proof}
By purely numerical consideration, there are seven possible combinations of five numbers. What we need to show is that the combinations 
\[
a^2bcd=\{a,a,b,c,d\},\quad
abcde=\{a,b,c,d,e\},
\]
are impossible. Up to symmetry, all the possible edge length arrangements of these two combinations are listed in Figure \ref{impossibleedge}.

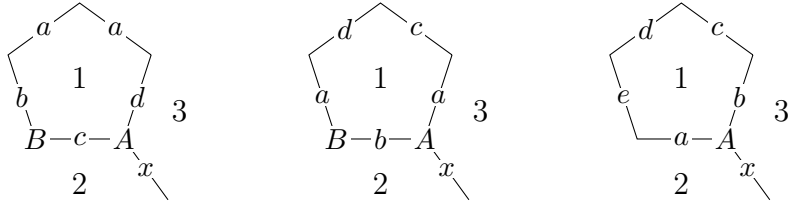
\begin{figure}[htp]
\centering
    \begin{tikzpicture}[scale=1]

    \draw (18:1) -- 	node[fill=white,inner sep=1] {\small $a$}
		(90:1) -- node[fill=white,inner sep=1] {\small $a$}
		(162:1) -- node[fill=white,inner sep=1] {\small $b$}
		(234:1) -- node[fill=white,inner sep=1] {\small $c$}
		(306:1) -- node[fill=white,inner sep=1] {\small $d$} (18:1);
	\draw
		(306:1) -- node[fill=white,inner sep=1] {\small $x$} (306:2);

	\node at (0:0) {$1$};
	\node at (-90:1.4) {$2$};
	\node at (-18:1.4) {$3$};
	
	\node[fill=white,inner sep=1] at (306:1) {\small $A$};
	\node[fill=white,inner sep=1] at (234:1) {\small $B$};
	
	\begin{scope}[xshift=4cm]
	
	\draw (18:1) -- 	node[fill=white,inner sep=1] {\small $c$}
		(90:1) -- node[fill=white,inner sep=1] {\small $d$}
		(162:1) -- node[fill=white,inner sep=1] {\small $a$}
		(234:1) -- node[fill=white,inner sep=1] {\small $b$}
		(306:1) -- node[fill=white,inner sep=1] {\small $a$} (18:1);
	\draw
		(306:1) -- node[fill=white,inner sep=1] {\small $x$} (306:2);

	\node at (0:0) {$1$};
	\node at (-90:1.4) {$2$};
	\node at (-18:1.4) {$3$};
	
	\node[fill=white,inner sep=1] at (306:1) {\small $A$};
	\node[fill=white,inner sep=1] at (234:1) {\small $B$};
	
	\end{scope}

	\begin{scope}[xshift=8cm]

	\draw (18:1) -- 	node[fill=white,inner sep=1] {\small $c$}
		(90:1) -- node[fill=white,inner sep=1] {\small $d$}
		(162:1) -- node[fill=white,inner sep=1] {\small $e$}
		(234:1) -- node[fill=white,inner sep=1] {\small $a$}
		(306:1) -- node[fill=white,inner sep=1] {\small $b$} (18:1);
	\draw (306:1) -- node[fill=white,inner sep=1] {\small $x$} (306:2);

	\node at (0:0) {$1$};
	\node at (-90:1.4) {$2$};
	\node at (-18:1.4) {$3$};
	
	\node[fill=white,inner sep=1] at (306:1) {\small $A$};
	
	\end{scope}

    \end{tikzpicture}
\caption{Impossible edge length arrangements.}
\label{impossibleedge}
\end{figure}

In the first arrangement, by Lemma \ref{deg3_4}, one of the vertices $A$ or $B$ must have degree $3$. By symmetry, we may assume $\deg A=3$ without loss of generality. Let $x$ be the only other edge at $A$ besides $c$ and $d$. Then we have tiles $P_2$ and $P_3$ described in the picture. The edge $x$ is adjacent to an edge of length $c$ in $P_2$. Since $P_2$ is edge congruent to $P_1$, and the edges adjacent to the unique edge of length $c$ in $P_1$ are the edges of lengths $b$ and $d$. Therefore, we have $x=b$ or $x=d$ as far as the edge length is concerned. Similarly, the adjacency of $x$ and $d$ in $P_3$ leads to $x=a$ or $x=c$ as far as the edge length is concerned. Since $a,b,c,d$ are distinct, we get a contradiction. 

We have been very careful with the wording in the argument above. We say ``an edge of length $c$'' instead of ``the edge $c$'' because this allows the possibility that the tile may have several edges of the same length $c$. This may happen when $a,b,c,\dotsc$ are not assumed to be distinct. Since distinct letters indeed denote distinct values for the length in the current proposition, there is actually no confusion about the wording ``the edge $c$''. The wording in the subsequent discussion will be more casual.

In the second arrangement, by Lemma \ref{deg3_4} and symmetry, we may still assume $\deg A=3$. Then we may introduce $x$, $P_2$, $P_3$ as before. By considering $x$ in $P_2$ and $P_3$, we see that $x$ is adjacent to $a$ and $b$. Since there is no such edge in $P_1$, we get a contradiction.

In the third arrangement, we may again assume $\deg A=3$ and introduce $x$, $P_2$, $P_3$. Then $x$ is adjacent to both $a$ and $b$ from the viewpoint of $P_2$ and $P_3$, yet there is no such edge in $P_1$.
\end{proof}

\begin{proposition}\label{edge2}
If a spherical tiling by edge congruent pentagons contains a tile with all vertices having degree $3$, then the edge lengths must be arranged in one of the five ways in Figure \ref{edges}.
\end{proposition}

\begin{figure}[htp]
\centering
    \begin{tikzpicture}[scale=1]

    \draw (18:1) -- 	node[fill=white,inner sep=1] {\small $a$}
		(90:1) -- node[fill=white,inner sep=1] {\small $a$}
		(162:1) -- node[fill=white,inner sep=1] {\small $a$}
		(234:1) -- node[fill=white,inner sep=1] {\small $a$}
		(306:1) -- node[fill=white,inner sep=1] {\small $a$} (18:1);
	
	\begin{scope}[xshift=2.5cm]
	
	\draw (18:1) -- 	node[fill=white,inner sep=1] {\small $a$}
		(90:1) -- node[fill=white,inner sep=1] {\small $a$}
		(162:1) -- node[fill=white,inner sep=1] {\small $a$}
		(234:1) -- node[fill=white,inner sep=1] {\small $b$}
		(306:1) -- node[fill=white,inner sep=1] {\small $a$} (18:1);
	
	\end{scope}
	
	\begin{scope}[xshift=5cm]
	
	\draw (18:1) -- 	node[fill=white,inner sep=1] {\small $b$}
		(90:1) -- node[fill=white,inner sep=1] {\small $b$}
		(162:1) -- node[fill=white,inner sep=1] {\small $a$}
		(234:1) -- node[fill=white,inner sep=1] {\small $a$}
		(306:1) -- node[fill=white,inner sep=1] {\small $a$} (18:1);
	
	\end{scope}

    
    \begin{scope}[xshift=7.5cm]
    
    \draw (18:1) -- 	node[fill=white,inner sep=1] {\small $c$}
		(90:1) -- node[fill=white,inner sep=1] {\small $b$}
		(162:1) -- node[fill=white,inner sep=1] {\small $a$}
		(234:1) -- node[fill=white,inner sep=1] {\small $a$}
		(306:1) -- node[fill=white,inner sep=1] {\small $a$} (18:1);

	\end{scope}
		
	\begin{scope}[xshift=10cm]
	
	\draw (18:1) -- 	node[fill=white,inner sep=1] {\small $b$}
		(90:1) -- node[fill=white,inner sep=1] {\small $a$}
		(162:1) -- node[fill=white,inner sep=1] {\small $a$}
		(234:1) -- node[fill=white,inner sep=1] {\small $c$}
		(306:1) -- node[fill=white,inner sep=1] {\small $b$} (18:1);
	
	\end{scope}

    \end{tikzpicture}
\caption{Edge length arrangements when all vertices have degree $3$.}
\label{edges}
\end{figure}
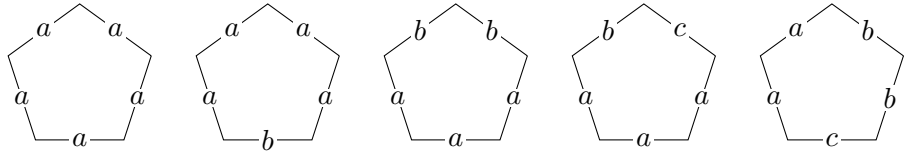

\begin{proof}
By purely numerical consideration, the five possible edge length combinations $a^5$, $a^4b$, $a^3b^2$, $a^3bc$, $a^2b^2c$ in Proposition \ref{edge1} have respectively $1$, $1$, $2$, $2$, $3$ possible arrangements. Besides the five in Figure \ref{edges}, the remaining four are given in Figure \ref{impossibleedge2}. We need to argue that these four are impossible.

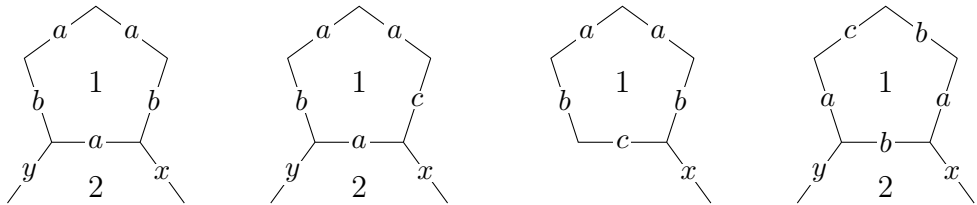
\begin{figure}[htp]
\centering
    \begin{tikzpicture}[scale=1]

	\draw (18:1) -- 	node[fill=white,inner sep=1] {\small $a$}
		(90:1) -- node[fill=white,inner sep=1] {\small $a$}
		(162:1) -- node[fill=white,inner sep=1] {\small $b$}
		(234:1) -- node[fill=white,inner sep=1] {\small $a$}
		(306:1) -- node[fill=white,inner sep=1] {\small $b$} (18:1);
	\draw
		(306:1) -- node[fill=white,inner sep=1] {\small $x$} (306:2)
		(234:1) -- node[fill=white,inner sep=1] {\small $y$} (234:2);
		
	\node at (0:0) {$1$};
	\node at (-90:1.4) {$2$};

	\begin{scope}[xshift=3.5cm]
	
	\draw (18:1) -- 	node[fill=white,inner sep=1] {\small $a$}
		(90:1) -- node[fill=white,inner sep=1] {\small $a$}
		(162:1) -- node[fill=white,inner sep=1] {\small $b$}
		(234:1) -- node[fill=white,inner sep=1] {\small $a$}
		(306:1) -- node[fill=white,inner sep=1] {\small $c$} (18:1);
	\draw
		(306:1) -- node[fill=white,inner sep=1] {\small $x$} (306:2)
		(234:1) -- node[fill=white,inner sep=1] {\small $y$} (234:2);

	\node at (0:0) {$1$};
	\node at (-90:1.4) {$2$};
	
	\end{scope}
		
	\begin{scope}[xshift=7cm]
	
	\draw (18:1) -- 	node[fill=white,inner sep=1] {\small $a$}
		(90:1) -- node[fill=white,inner sep=1] {\small $a$}
		(162:1) -- node[fill=white,inner sep=1] {\small $b$}
		(234:1) -- node[fill=white,inner sep=1] {\small $c$}
		(306:1) -- node[fill=white,inner sep=1] {\small $b$} (18:1);
	\draw
		(306:1) -- node[fill=white,inner sep=1] {\small $x$} (306:2);

	\node at (0:0) {$1$};
	
	\end{scope}

	\begin{scope}[xshift=10.5cm]
	
	\draw (18:1) -- 	node[fill=white,inner sep=1] {\small $b$}
		(90:1) -- node[fill=white,inner sep=1] {\small $c$}
		(162:1) -- node[fill=white,inner sep=1] {\small $a$}
		(234:1) -- node[fill=white,inner sep=1] {\small $b$}
		(306:1) -- node[fill=white,inner sep=1] {\small $a$} (18:1);
	\draw
		(306:1) -- node[fill=white,inner sep=1] {\small $x$} (306:2)
		(234:1) -- node[fill=white,inner sep=1] {\small $y$} (234:2);

	\node at (0:0) {$1$};
	\node at (-90:1.4) {$2$};
	
	\end{scope}

    \end{tikzpicture}
\caption{Impossible edge length arrangements.}
\label{impossibleedge2}
\end{figure}

In the first arrangement, $x$ and $y$ are adjacent to $a$ and $b$. Since the only such edge in $P_1$ is $a$, we have $x=y=a$. Then in $P_2$, we get a row $xay$ of three consecutive edges of the same length $a$. Since there is no such row in $P_1$, we get a contradiction.

In the second arrangement, $x$ is adjacent to $c$, and $y$ is adjacent to $b$. By the adjacency in $P_1$, we have $x=y=a$. Then the edge row $xby=aaa$ appears in $P_2$ but not in $P_1$. We get a contradiction.

In the third arrangement, $x$ is adjacent to $b$ and $c$. Since there is no such edge in $P_1$, we get a contradiction. 

In the fourth arrangement, $x$ and $y$ are adjacent to $a$ and $b$. By the adjacency in $P_1$, we have $x=y=c$. Then $c$ appears twice in $P_2$ but only once in $P_1$. We get a contradiction.
\end{proof}

\section{Classification of Edge Congruent Tiling}
\label{edgetile2}

The discussion in Section \ref{edgetile1} is local, because we ignored the global combinatorial structure given by Propositions \ref{tile_pattern}. In this section, we discuss how the classification in Proposition \ref{edge2} can fit the global structure.

We denote the tiles in Figure \ref{tile_pattern_pic} by $P_1,P_2,\dotsc,P_{12}$. We denote by $E_{ij}$ the edge shared by the tiles $P_i$, $P_j$, and by $V_{ijk}$ the vertex shared by $P_i$, $P_j$, $P_k$. 

We also name an edge by its length and a vertex by the length of the edges at the vertex. For example, in Figure \ref{generalclass}, the edge $E_{12}$ is a $b$-edge, and the edge $E_{13}$ is an $a$-edge. Moreover, the vertex $V_{123}$ is an $abc$-vertex, and the vertex $V_{134}$ is an $a^3$-vertex. The tiling has total of $12$ $a$-edges, $12$ $b$-edges, $6$ $c$-edges, $4$ $a^3$-vertices, $4$ $b^3$-vertices, and $12$ $abc$-vertices. We denote the {\em edgewise vertex combination} of the tiling by $\{4a^3,4b^3,12abc\}$.

For the edge length combination $a^5$ (the first pentagon in Figure \ref{edges}), we basically assign $a$ to all edges in Figure \ref{tile_pattern_pic}. All vertices are $a^3$-vertices, and the edgewise vertex combination is $\{20a^3\}$. This can be realized by the regular dodecahedron tiling.

\begin{proposition}\label{edge_pattern2}
The spherical tilings by $12$ edge congruent pentagons with
edge lengths $a,a,a,a,b$, where $a\ne b$, are given by the tilings in Figure \ref{case4a1b} up to symmetries. The unlabeled edges have length $a$.
\end{proposition}

In proving the main theorem, we will only use the edgewise vertex combination $\{8a^3,12a^2b\}$ derived in the proof.

\begin{figure}[ht]
\centering
\begin{tikzpicture}[scale=0.7]

\foreach \a / \b in {0/0, 7/0, 14/0, 3.5/-7, 10.5/-7}
{
\begin{scope}[shift={(\a ,\b)}]

    \foreach \x in {1,...,5}
      \draw (-54+\x*72:1) -- (18+\x*72:1) -- (18+\x*72:1.75) -- (-18+\x*72:2.4) -- (-54+\x*72:1.75) -- cycle
      (-18+\x*72:2.5) -- (-18+\x*72:3.25) -- (-90+\x*72:3.25) -- (-90+\x*72:2.4);
      
    \node at(0:0){\small $1$};
    \foreach \y in {2,...,6}
      \node at (-90+\y*72:1.4){\small $\y$};
    \foreach \z in {7,...,11}
      \node at (-54+\z*72:2.2){\small $\z$};
	\node at (240:3.5){\small $12$};
	
\end{scope}
}


    \node[fill=white,inner sep=1] at (-90:0.8) {\small $b$};
    \node[fill=white,inner sep=1] at (18:1.35) {\small $b$};
	\node[fill=white,inner sep=1] at (162:1.35) {\small $b$};
	\node[fill=white,inner sep=1] at (90:2.7) {\small $b$};
	\node[fill=white,inner sep=1] at (198:2.8) {\small $b$};
	\node[fill=white,inner sep=1] at (-18:2.8) {\small $b$};


\begin{scope}[xshift=7cm]

	\node[fill=white,inner sep=1] at (-90:0.8) {\small $b$};
    \node[fill=white,inner sep=1] at (18:1.35) {\small $b$};
	\node[fill=white,inner sep=1] at (162:1.35) {\small $b$};
	\node[fill=white,inner sep=1] at (18:2.7) {\small $b$};
	\node[fill=white,inner sep=1] at (126:2.8) {\small $b$};
	\node[fill=white,inner sep=1] at (-90:2.8) {\small $b$};

\end{scope}


\begin{scope}[xshift=14cm]

	\node[fill=white,inner sep=1] at (-90:0.8) {\small $b$};
    \node[fill=white,inner sep=1] at (18:1.35) {\small $b$};
	\node[fill=white,inner sep=1] at (162:1.35) {\small $b$};
	\node[fill=white,inner sep=1] at (-54:2.75) {\small $b$};
	\node[fill=white,inner sep=1] at (54:2.8) {\small $b$};
	\node[fill=white,inner sep=1] at (198:2.8) {\small $b$};
	
\end{scope}


\begin{scope}[shift={(3.5,-7)}]

	\node[fill=white,inner sep=1] at (-90:0.8) {\small $b$};
    \node[fill=white,inner sep=1] at (18:1.35) {\small $b$};
	\node[fill=white,inner sep=1] at (108:2) {\small $b$};
	\node[fill=white,inner sep=1] at (180:2) {\small $b$};
	\node[fill=white,inner sep=1] at (18:2.7) {\small $b$};
	\node[fill=white,inner sep=1] at (-90:2.8) {\small $b$};

\end{scope}


\begin{scope}[shift={(10.5,-7)}]

	\node[fill=white,inner sep=1] at (-90:0.8) {\small $b$};
    \node[fill=white,inner sep=1] at (90:1.35) {\small $b$};
	\node[fill=white,inner sep=1] at (0:2) {\small $b$};
	\node[fill=white,inner sep=1] at (180:2) {\small $b$};
	\node[fill=white,inner sep=1] at (90:2.7) {\small $b$};
	\node[fill=white,inner sep=1] at (-90:2.8) {\small $b$};

\end{scope}

\end{tikzpicture}
\caption{Tilings for the edge length combination $a^4b$.}
\label{case4a1b}
\end{figure}
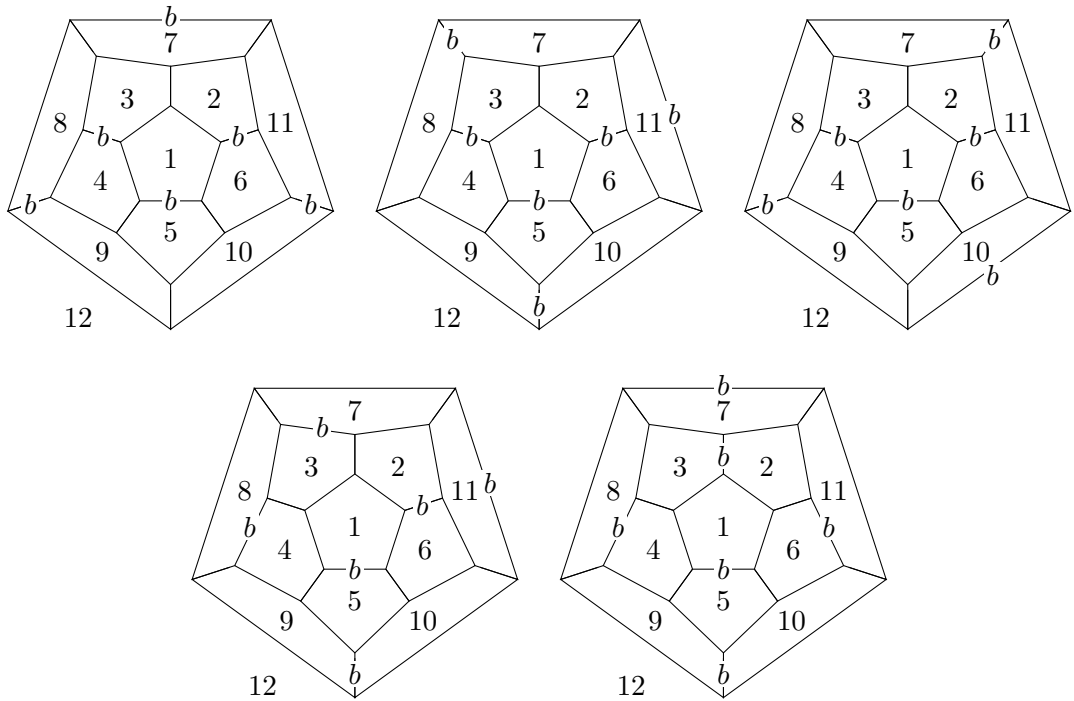

\begin{proof}
We need to assign $6$ $b$-edges to Figure \ref{tile_pattern_pic}, such that each tile has exactly one $b$-edge. In such an assignment, there is no $ab^2$-vertex and no $b^3$-vertex. So the edgewise vertex combination is $\{ma^3,na^2b\}$. On the other hand, $\frac{1}{5}$ of all the $30$ edges are $b$-edges. Therefore there are $N_a=24$ $a$-edges and $N_b=6$ $b$-edges. Then from the edgewise vertex combination, we get $3m+2n=2N_a=48$ and $n=2N_b=12$. The solution is $m=8$ and $n=12$.

Any tile $P$ in Figure \ref{tile_pattern_pic} and the five tiles around $P$ form a tiling $T_P$ of a ``half sphere'' (with wiggled boundary). Moreover, the remaining six tiles also form a tiling of the complementary ``half sphere''. In fact, $P$ has an ``antipodal tile'' $P'$ (the antipodal of $P_1$ is $P_{12}$, for example), and the remaining six tiles form the tiling $T_{P'}$.

We say an assignment of $b$-edges is {\em splitting}, if there is a half sphere tiling $T_P$, such that if a $b$-edge is shared by $P_i$ and $P_j$, then $P_i\in T_P$ if and only if $P_j\in T_P$. This implies that $P_i\in T_{P'}$ if and only if $P_j\in T_{P'}$. We also call the tile $P$ a {\em splitting center}.

A splitting assignment is then the union of two independent assignments of $3$ $b$-edges to $T_P$ and the other $3$ $b$-edges to $T_{P'}$. In this case, it is easy to see that the choice of the $b$-edge for $P$ completely determines the other two $b$-edges in $T_P$. See $T_{P_1}$ in the first three tilings in Figure \ref{case4a1b}. Up to symmetry, we have three possible combinations of the choice of $b$-edges for $P_1$ and $P_{12}$, which give the first three tilings in Figure \ref{case4a1b}. They are not equivalent because they have respectively $6$, $2$ and $4$ splitting centers.

Next we search for non-splitting assignments. Assume there are edges connecting two $b$-edges. Without loss of generality, we may assume $E_{15}=E_{26}=b$ as in the fourth tiling in Figure \ref{case4a1b}. Since $P_1$ and $P_6$ are not splitting centers, we get $E_{34}\ne b$ and $E_{\overline{10}\,\overline{11}}\ne b$. If $E_{7\overline{11}}=b$, then the only possible $b$-edge of $P_3$ is $E_{38}=b$. This further implies that the $b$-edge of $P_4$ is $E_{49}=b$. Then $P_4$ becomes a splitting center. Therefore $E_{7\overline{11}}\ne b$, and the only possible $b$-edge of $P_{11}$ is $E_{\overline{11}\,\overline{12}}=b$. This further successively implies that the $b$-edge of $P_{10}$ is $E_{9\overline{10}}$, the $b$-edge of $P_4$ is $E_{48}$, and $E_{37}=b$.

Finally we look for non-splitting assignments such that no edge connects two $b$-edges. Without loss of generality, we may assume $E_{15}=b$. Since $E_{26}$, $E_{34}$, $E_{49}$, $E_{6\overline{10}}$ are connected to $E_{15}$ by one edge, they are not $b$-edges. This implies that $E_{48}=E_{6\overline{11}}=b$. Then the $8$ edges connected to $E_{48}$ and $E_{6\overline{11}}$ by one edge cannot be $b$-edges. This is enough for us to determine the remaining $3$ $b$-edges and get the fifth tiling in Figure \ref{case4a1b}.
\end{proof}

\begin{proposition}\label{edge_pattern3}
The spherical tiling by $12$ edge congruent pentagons with edge lengths $a,a,a,b,b$, where $a\ne b$, is given by the left of Figure \ref{case3a2b} up to symmetries.
\end{proposition}

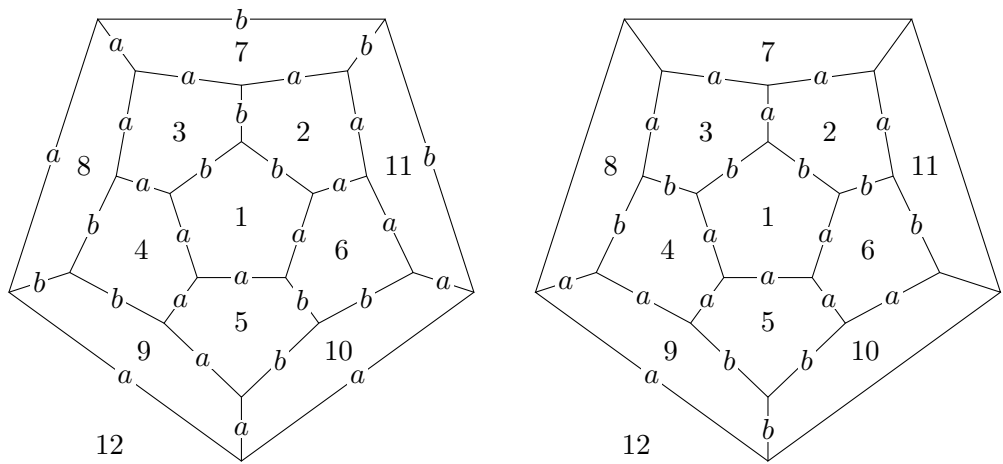
\begin{figure}[htp]
\centering
\begin{tikzpicture}[scale=1]


	\draw (18:1) 
		--node[fill=white,inner sep=1] {\small $b$} (90:1)
		--node[fill=white,inner sep=1] {\small $b$} (162:1)
		--node[fill=white,inner sep=1] {\small $a$} (234:1)
		--node[fill=white,inner sep=1] {\small $a$} (306:1)
		--node[fill=white,inner sep=1] {\small $a$} (18:1);

	\draw (18:1) 
		--node[fill=white,inner sep=1] {\small $a$} (18:1.75)
		--node[fill=white,inner sep=1] {\small $a$} (54:2.4)
		--node[fill=white,inner sep=1] {\small $a$} (90:1.75);
		
	\draw (90:1) 
		--node[fill=white,inner sep=1] {\small $b$} (90:1.75)
		--node[fill=white,inner sep=1] {\small $a$} (126:2.4)
		--node[fill=white,inner sep=1] {\small $a$} (162:1.75);

	\draw (162:1) 
		--node[fill=white,inner sep=1] {\small $a$} (162:1.75)
		--node[fill=white,inner sep=1] {\small $b$} (198:2.4)
		--node[fill=white,inner sep=1] {\small $b$} (234:1.75);

	\draw (234:1) 
		--node[fill=white,inner sep=1] {\small $a$} (234:1.75)
		--node[fill=white,inner sep=1] {\small $a$} (-90:2.4)
		--node[fill=white,inner sep=1] {\small $b$} (-54:1.75);		

	\draw (-54:1) 
		--node[fill=white,inner sep=1] {\small $b$} (-54:1.75)
		--node[fill=white,inner sep=1] {\small $b$} (-18:2.4)
		--node[fill=white,inner sep=1] {\small $a$} (18:1.75);	
				
	\draw (54:2.4) 
		--node[fill=white,inner sep=1] {\small $b$} (54:3.25)
		--node[fill=white,inner sep=1] {\small $b$} (126:3.25);

	\draw (126:2.4) 
		--node[fill=white,inner sep=1] {\small $a$} (126:3.25)
		--node[fill=white,inner sep=1] {\small $a$} (198:3.25);

	\draw (198:2.4) 
		--node[fill=white,inner sep=1] {\small $b$} (198:3.25)
		--node[fill=white,inner sep=1] {\small $a$} (-90:3.25);

	\draw (-90:2.4) 
		--node[fill=white,inner sep=1] {\small $a$} (-90:3.25)
		--node[fill=white,inner sep=1] {\small $a$} (-18:3.25);

	\draw (-18:2.4) 
		--node[fill=white,inner sep=1] {\small $a$} (-18:3.25)
		--node[fill=white,inner sep=1] {\small $b$} (54:3.25);
						
    \node at (0:0) {\small $1$};
    \foreach \y in {2,...,6}
      \node at (-90+\y*72:1.4) {\small $\y$};
    \foreach \z in {7,...,11}
      \node at (-54+\z*72:2.2) {\small $\z$};
      
      \node at (240:3.5){\small $12$};

      
    \begin{scope}[xshift=7cm]

	\draw (18:1) 
		--node[fill=white,inner sep=1] {\small $b$} (90:1)
		--node[fill=white,inner sep=1] {\small $b$} (162:1)
		--node[fill=white,inner sep=1] {\small $a$} (234:1)
		--node[fill=white,inner sep=1] {\small $a$} (306:1)
		--node[fill=white,inner sep=1] {\small $a$} (18:1);

	\draw (18:1) 
		--node[fill=white,inner sep=1] {\small $b$} (18:1.75)
		--node[fill=white,inner sep=1] {\small $a$} (54:2.4)
		--node[fill=white,inner sep=1] {\small $a$} (90:1.75);
		
	\draw (90:1) 
		--node[fill=white,inner sep=1] {\small $a$} (90:1.75)
		--node[fill=white,inner sep=1] {\small $a$} (126:2.4)
		--node[fill=white,inner sep=1] {\small $a$} (162:1.75);

	\draw (162:1) 
		--node[fill=white,inner sep=1] {\small $b$} (162:1.75)
		--node[fill=white,inner sep=1] {\small $b$} (198:2.4)
		--node[fill=white,inner sep=1] {\small $a$} (234:1.75);

	\draw (234:1) 
		--node[fill=white,inner sep=1] {\small $a$} (234:1.75)
		--node[fill=white,inner sep=1] {\small $b$} (-90:2.4)
		--node[fill=white,inner sep=1] {\small $b$} (-54:1.75);		
	\draw (-54:1) 
		--node[fill=white,inner sep=1] {\small $a$} (-54:1.75)
		--node[fill=white,inner sep=1] {\small $a$} (-18:2.4)
		--node[fill=white,inner sep=1] {\small $b$} (18:1.75);	
				
	\draw (54:2.4) 
		-- (54:3.25)
		-- (126:3.25);

	\draw (126:2.4) 
		-- (126:3.25)
		-- (198:3.25);

	\draw (198:2.4) 
		--node[fill=white,inner sep=1] {\small $a$} (198:3.25)
		--node[fill=white,inner sep=1] {\small $a$} (-90:3.25);

	\draw (-90:2.4) 
		--node[fill=white,inner sep=1] {\small $b$} (-90:3.25)
		-- (-18:3.25);

	\draw (-18:2.4) 
		-- (-18:3.25)
		-- (54:3.25);
						
    \node at (0:0) {\small $1$};
    \foreach \y in {2,...,6}
      \node at (-90+\y*72:1.4) {\small $\y$};
    \foreach \z in {7,...,11}
      \node at (-54+\z*72:2.2) {\small $\z$};
      
      \node at (240:3.5){\small $12$};
      
      \end{scope}

    \end{tikzpicture}
\caption{Tilings for the edge length combination $a^3b^2$.}
\label{case3a2b}
\end{figure}

\begin{proof}
By Proposition \ref{edge2}, we may start with $P_1$, with the edges arranged as the third in Figure \ref{edges}. See Figure \ref{case3a2b}. The length of $E_{23}$ is either $a$ or $b$.

The case $E_{23}=b$ is described on the left of Figure \ref{case3a2b}. From $E_{12}=E_{23}=b$, we see that the other three edges of $P_2$ have length $a$. By the same reason, we get all the edge lengths of $P_3$. Then among $E_{78}$ and $E_{7\overline{11}}$, one is $a$ and the other is $b$. By symmetry, we may assume $E_{78}=a$ and $E_{7\overline{11}}=b$. Then $E_{7\overline{12}}=b$. From $E_{2\overline{11}}=a$, $E_{7\overline{11}}=b$, we get all the edges of $P_{11}$. From $E_{16}=E_{26}=E_{6\overline{11}}=a$, we get all the edges of $P_6$. From $E_{6\overline{10}}=b$, $E_{\overline{10}\,\overline{11}}=a$, we get all the edges of $P_{10}$. From $E_{56}=E_{5\overline{10}}=b$, we get all the edges of $P_5$. From $E_{14}=E_{34}=E_{45}=a$, we get all the edges of $P_4$. From $E_{49}=b,E_{59}=a$, we get all the edges of $P_9$. From $E_{38}=a$, $E_{48}=b$, we get all the edges of $P_8$. This concludes all the edge lengths.

The case $E_{23}=a$ is the right of Figure \ref{case3a2b}. We can immediately get all the edges of $P_2$, $P_3$. Afterwards, we get all the edges of $P_4$, $P_6$. Now we know three edges of $P_5$ to be $a$, so that the other two edges are $b$. Then we get all the edges of $P_9$, and find that $P_8$ has one $b$-edge adjacent to two $a$-edges, a contradiction. 
\end{proof}

\begin{proposition}\label{edge_pattern5}
The spherical tiling by $12$ edge congruent pentagons cannot have edge lengths $a,a,a,b,c$, for distinct $a,b,c$.
\end{proposition}

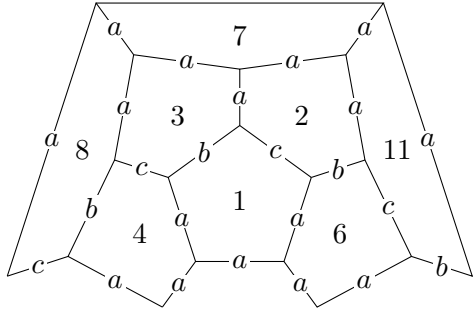
\begin{figure}[htp]
\centering
\begin{tikzpicture}[scale=1]


	\draw (18:1)
		--node[fill=white,inner sep=1] {\small $c$} (90:1)
		--node[fill=white,inner sep=1] {\small $b$} (162:1)
		--node[fill=white,inner sep=1] {\small $a$} (234:1)
		--node[fill=white,inner sep=1] {\small $a$} (306:1)
		--node[fill=white,inner sep=1] {\small $a$} (18:1);

	\draw (18:1)
		--node[fill=white,inner sep=1] {\small $b$} (18:1.75)
		--node[fill=white,inner sep=1] {\small $a$} (54:2.4)
		--node[fill=white,inner sep=1] {\small $a$} (90:1.75);
		
	\draw (90:1)
		--node[fill=white,inner sep=1] {\small $a$} (90:1.75)
		--node[fill=white,inner sep=1] {\small $a$} (126:2.4)
		--node[fill=white,inner sep=1] {\small $a$} (162:1.75);

	\draw (162:1) 
		--node[fill=white,inner sep=1] {\small $c$} (162:1.75)
		--node[fill=white,inner sep=1] {\small $b$} (198:2.4)
		--node[fill=white,inner sep=1] {\small $a$} (234:1.75);

	\draw (234:1)
		--node[fill=white,inner sep=1] {\small $a$} (234:1.75);	

	\draw (-54:1)
		--node[fill=white,inner sep=1] {\small $a$} (-54:1.75)
		--node[fill=white,inner sep=1] {\small $a$} (-18:2.4)
		--node[fill=white,inner sep=1] {\small $c$} (18:1.75);	 
				
	\draw (54:2.4)
		--node[fill=white,inner sep=1] {\small $a$} (54:3.25)
		-- (126:3.25);

	\draw (126:2.4)
		--node[fill=white,inner sep=1] {\small $a$} (126:3.25)
		--node[fill=white,inner sep=1] {\small $a$} (198:3.25);

	\draw (198:2.4)
		--node[fill=white,inner sep=1] {\small $c$} (198:3.25);

	\draw (-18:2.4)
		--node[fill=white,inner sep=1] {\small $b$} (-18:3.25)
		--node[fill=white,inner sep=1] {\small $a$} (54:3.25);
						
    \node at (0:0) {\small $1$};
    \foreach \y in {2,3,4,6}
      \node at (-90+\y*72:1.4) {\small $\y$};
    \foreach \z in {7,8,11}
      \node at (-54+\z*72:2.2) {\small $\z$};

    \end{tikzpicture}
\caption{Tilings for the edge length combination $a^3bc$.}
\label{case2abcd}
\end{figure}

\begin{proof}
By Proposition \ref{edge2}, we may start with $P_1$, with the edges arranged as the fourth in Figure \ref{edges}. See Figure \ref{case2abcd}. Since $E_{23}$ is adjacent to $b$ and $c$, we get $E_{23}=a$. Then we successively get all the edges of $P_2$, $P_3$, $P_4$, $P_6$, $P_8$, $P_{11}$. Now we have four $a$-edges in $P_7$, a contradiction. 
\end{proof}

\begin{proposition}\label{edge_pattern4}
The spherical tiling by $12$ edge congruent pentagons with edge lengths $a,a,b,b,c$, where $a,b,c$ are distinct, is given by the left of Figure \ref{case2a2bc} up to symmetries.
\end{proposition}

\begin{figure}[htp]
\centering
\begin{tikzpicture}[scale=1]


	\draw (18:1)
		--node[fill=white,inner sep=1] {\small $b$} (90:1)
		--node[fill=white,inner sep=1] {\small $a$} (162:1)
		--node[fill=white,inner sep=1] {\small $a$} (234:1)
		--node[fill=white,inner sep=1] {\small $c$} (306:1)
		--node[fill=white,inner sep=1] {\small $b$} (18:1);

	\draw (18:1)
		--node[fill=white,inner sep=1] {\small $b$} (18:1.75)
		--node[fill=white,inner sep=1] {\small $a$} (54:2.4)
		--node[fill=white,inner sep=1] {\small $a$} (90:1.75);
		
	\draw (90:1)
		--node[fill=white,inner sep=1] {\small $c$} (90:1.75)
		--node[fill=white,inner sep=1] {\small $b$} (126:2.4)
		--node[fill=white,inner sep=1] {\small $b$} (162:1.75);

	\draw (162:1) 
		--node[fill=white,inner sep=1] {\small $a$} (162:1.75)
		--node[fill=white,inner sep=1] {\small $c$} (198:2.4)
		--node[fill=white,inner sep=1] {\small $b$} (234:1.75);

	\draw (234:1)
		--node[fill=white,inner sep=1] {\small $b$} (234:1.75)
		--node[fill=white,inner sep=1] {\small $b$} (-90:2.4)
		--node[fill=white,inner sep=1] {\small $a$} (-54:1.75);		

	\draw (-54:1)
		--node[fill=white,inner sep=1] {\small $a$} (-54:1.75)
		--node[fill=white,inner sep=1] {\small $a$} (-18:2.4)
		--node[fill=white,inner sep=1] {\small $c$} (18:1.75);	 
				
	\draw (54:2.4)
		--node[fill=white,inner sep=1] {\small $a$} (54:3.25)
		--node[fill=white,inner sep=1] {\small $c$} (126:3.25);

	\draw (126:2.4)
		--node[fill=white,inner sep=1] {\small $b$} (126:3.25)
		--node[fill=white,inner sep=1] {\small $a$} (198:3.25);

	\draw (198:2.4)
		--node[fill=white,inner sep=1] {\small $a$} (198:3.25)
		--node[fill=white,inner sep=1] {\small $a$} (-90:3.25);

	\draw (-90:2.4)
		--node[fill=white,inner sep=1] {\small $c$} (-90:3.25)
		--node[fill=white,inner sep=1] {\small $b$} (-18:3.25);

	\draw (-18:2.4)
		--node[fill=white,inner sep=1] {\small $b$} (-18:3.25)
		--node[fill=white,inner sep=1] {\small $b$} (54:3.25);
						
    \node at (0:0) {\small $1$};
    \foreach \y in {2,...,6}
      \node at (-90+\y*72:1.4) {\small $\y$};
    \foreach \z in {7,...,11}
      \node at (-54+\z*72:2.2) {\small $\z$};

      \node at (240:3.5){\small $12$};


\begin{scope}[xshift=7cm]

	\draw (18:1)
		--node[fill=white,inner sep=1] {\small $b$} (90:1)
		--node[fill=white,inner sep=1] {\small $a$} (162:1)
		--node[fill=white,inner sep=1] {\small $a$} (234:1)
		--node[fill=white,inner sep=1] {\small $c$} (306:1)
		--node[fill=white,inner sep=1] {\small $b$} (18:1);

	\draw (18:1)
		--node[fill=white,inner sep=1] {\small $c$} (18:1.75)
		--node[fill=white,inner sep=1] {\small $a$} (54:2.4)
		--node[fill=white,inner sep=1] {\small $a$} (90:1.75);
		
	\draw (90:1)
		--node[fill=white,inner sep=1] {\small $b$} (90:1.75)
		--node[fill=white,inner sep=1] {\small $b$} (126:2.4)
		--node[fill=white,inner sep=1] {\small $c$} (162:1.75);

	\draw (162:1) 
		--node[fill=white,inner sep=1] {\small $a$} (162:1.75)
		-- (198:2.4)
		-- (234:1.75);
		
	\draw (234:1)
		--node[fill=white,inner sep=1] {\small $a$} (234:1.75)
		--node[fill=white,inner sep=1] {\small $a$} (-90:2.4)
		--node[fill=white,inner sep=1] {\small $b$} (-54:1.75);		

	\draw (-54:1)
		--node[fill=white,inner sep=1] {\small $b$} (-54:1.75)
		--node[fill=white,inner sep=1] {\small $a$} (-18:2.4)
		--node[fill=white,inner sep=1] {\small $a$} (18:1.75);	 
				
	\draw (54:2.4)
		-- (54:3.25)
		--  (126:3.25);

	\draw (126:2.4)
		-- (126:3.25)
		-- (198:3.25);

	\draw (198:2.4)
		-- (198:3.25)
		-- (-90:3.25);

	\draw (-90:2.4)
		-- (-90:3.25)
		-- (-18:3.25);

	\draw (-18:2.4)
		-- (-18:3.25)
		-- (54:3.25);
						
    \node at (0:0) {\small $1$};
    \foreach \y in {2,...,6}
      \node at (-90+\y*72:1.4) {\small $\y$};

    \node at (-54:2.2) {\small $10$};
      
      \end{scope}

    \end{tikzpicture}
\caption{Tilings for the edge length combination $a^2b^2c$.}
\label{case2a2bc}
\end{figure}
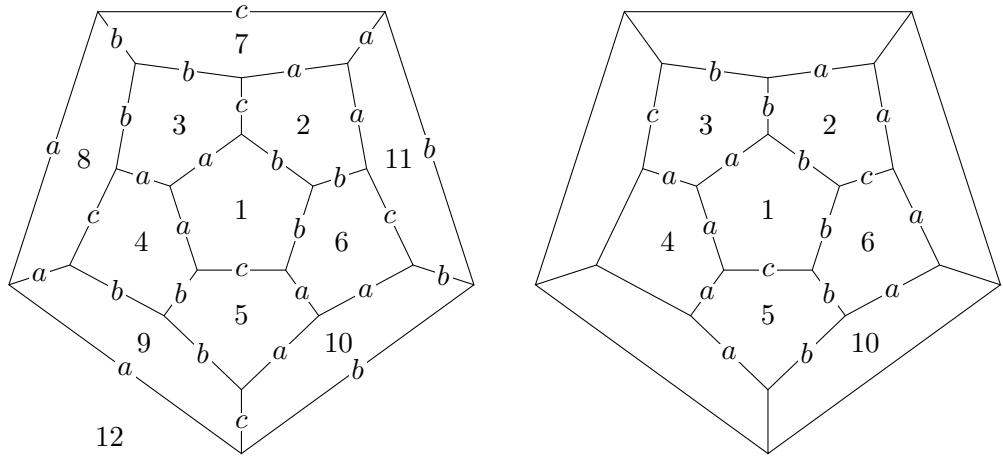

\begin{proof}
By Proposition \ref{edge2}, we may start with $P_1$, with the edges arranged as the fifth in Figure \ref{edges}. See Figure \ref{case2a2bc}. Since $E_{45}$ is adjacent to $c$, its length is either $a$ or $b$.

The case $E_{45}=b$ is the left of Figure \ref{case2a2bc}. From $E_{14}=a$, $E_{15}=c$, $E_{45}=b$, we get all the edges of $P_4$, $P_5$. Then we further get all the edges of $P_6$. Since we cannot have three $a$-edges or three $b$-edges in a tile, we must have $E_{23}=c$ and then get all the edges of $P_2$, $P_3$. By the same reason, we get $E_{9\overline{10}}=c$ and all the edges of $P_9$, $P_{10}$. Then we may get all the edges of $P_7$, $P_8$, $P_{11}$. This concludes all the edge lengths.

The case $E_{45}=a$ is the right of Figure \ref{case2a2bc}. We immediately get all the edges of $P_5$. Then in $P_6$, either $E_{26}$ or $E_{6\overline{10}}$ is $c$. The two assumptions are actually equivalent by symmetry, so that we may assume $E_{26}=c$ without loss of generality. Then we get all the edges of $P_2$, $P_6$. After further getting all the edges of $P_3$, we find three $a$-edges in $P_4$, a contradiction. 
\end{proof}

\section{Angle Congruent Tiling}
\label{angletile1}

Now we turn to the distribution of angles in a spherical tiling by congruent pentagons. Since edge lengths are ignored, we introduce the following concept.

\begin{definition}
Two polygons are {\em angle congruent}, if there is a correspondence between the edges, such that the adjacencies of edges are preserved and the angles between the adjacent edges are preserved.
\end{definition}

Two spherical triangles with great arc edges are congruent if and only if they are angle congruent. Two planar triangles with straight line edges are similar if and only if they are angle congruent. However, this is no longer true for quadrilaterals and pentagons. 

In the regular dodecahedron tiling, all the angles are 
\[
\alpha=\frac{2\pi}{3}.
\]
In what follows, we will always reserve $\alpha$ for this special angle, and let $\beta,\gamma,\dotsc$ be angles not equal to $\alpha$. In this section, we may occasionally allow some from $\beta,\gamma,\dotsc$ to be equal. In the later sections, we will assume $\beta,\gamma,\dotsc$ to be distinct.

\begin{lemma}\label{anglesum}
In a spherical tiling by $12$ angle congruent pentagons, the sum of five angles in a tile is $\frac{10 \pi}{3}=5\alpha$.
\end{lemma}

In case the edges are great arcs, the lemma is a consequence of the formula for the area of a spherical pentagon. However, the lemma holds even if the edges are not great arcs. 

\begin{proof}
The sum of angles at each vertex is $2\pi$. By Lemma \ref{vertex3}, there are $20$ vertices. Therefore the total sum of all the angles in the tiling is $40\pi$. On the other hand, the sum $\Sigma$ of five angles in a tile is the same for all tiles because of the angle congruence. Since there are $12$ tiles, we get $40\pi=12\Sigma$. This implies $\Sigma=\frac{10 \pi}{3}$.
\end{proof}

In Section \ref{edgetile2}, we named a vertex by the lengths of edges at the vertex. Similarly, we can also name a vertex by the angles at the vertex. For example, in Figure \ref{generalclass}, $V_{123}$ is a $\beta\gamma\delta$-vertex, and $V_{126}$ is an $\alpha^3$-vertex. The tiling has $8$ $\alpha^3$-vertices and $12$ $\beta\gamma\delta$-vertices, and has the {\em anglewise vertex combination} $\{8\alpha^3,12\beta\gamma\delta\}$. We also say that $\alpha^3$ and $\beta\gamma\delta$ are vertices in the tiling, and the other combinations are not vertices.

We say a vertex is of {\em $\alpha\beta\gamma$-type}, if the vertex is an  $\alpha\beta'\gamma'$-vertex, where $\beta'$ and $\gamma'$ may or may not be $\beta$ and $\gamma$. Thus an $\alpha\beta\gamma$-type vertex can be an $\alpha\beta\gamma$-vertex, or an $\alpha\beta\delta$-vertex, or an $\alpha\gamma\delta$-vertex, etc. From purely symbolical viewpoint (recall that the symbol $\alpha$ is special), a degree $3$ vertex must be one of the following types:
\[
\alpha^3,\alpha^2\beta,\alpha\beta^2,\alpha\beta\gamma,\beta^3,\beta^2\gamma,\beta\gamma\delta.
\]

\begin{lemma}\label{atype}
A vertex of degree $3$ must be one of the following types:
\begin{enumerate}
\item $\alpha^3$.
\item $\alpha\beta\gamma$ with $\beta\ne\gamma$.
\item $\beta^2\gamma$ with $\beta\ne\gamma$.
\item $\beta\gamma\delta$ with $\beta,\gamma,\delta$ distinct.
\end{enumerate}
\end{lemma}

\begin{proof}
We need to show that, if $\alpha\ne\beta$, then $\alpha^2\beta$, $\alpha\beta^2$, $\beta^3$ cannot be vertices. If $\alpha^2\beta$ is a vertex, then $2\alpha+\beta=2\pi=3\alpha$, which implies $\beta=\alpha$, a contradiction. If $\alpha\beta^2$ is a vertex, then $\alpha+2\beta=2\pi=3\alpha$, which again implies $\beta=\alpha$. If $\beta^3$ is a vertex, then $3\beta=2\pi=3\alpha$, which still implies $\beta=\alpha$.
\end{proof}

\begin{lemma}\label{aexclude}
Suppose $\beta,\gamma,\delta,\epsilon$ are distinct angles.
\begin{enumerate}
\item If $\beta^2\gamma$ is a vertex, then $\beta^2\delta$, $\beta\gamma^2$, $\beta\gamma\delta$ are not vertices.
\item If $\beta\gamma\delta$ is a vertex, then $\beta^2\gamma$, $\beta\gamma\epsilon$ are not vertices.
\end{enumerate}
\end{lemma}

The lemma actually allows some angles to be $\alpha$. Moreover, we get more exclusions by symmetry. For example, if $\beta\gamma\delta$ is a vertex, then the following cannot be vertices:
\[
\beta^2\gamma,\beta^2\delta,\beta\gamma^2,\beta\delta^2,\gamma^2\delta,\gamma\delta^2,\beta\gamma\epsilon,\beta\delta\epsilon,\gamma\delta\epsilon.
\]

\begin{proof}
If both $\beta^2\gamma$ and $\beta^2\delta$ are vertices, then $2\beta+\gamma=3\alpha$ and $2\beta+\delta=3\alpha$. This implies $\gamma=\delta$, a contradiction. 

If both $\beta^2\gamma$ and $\beta\gamma^2$ are vertices, then $2\beta+\gamma=3\alpha$ and $\beta+2\gamma=3\alpha$. This implies $\beta=\gamma$, a contradiction. 

If both $\beta^2\gamma$ and $\beta\gamma\delta$ are vertices, then $2\beta+\gamma=3\alpha$ and $\beta+\gamma+\delta=3\alpha$. This implies $\beta=\delta$, again a contradiction. 

The proof of the second statement is similar.
\end{proof}

\begin{proposition}\label{angle}
Let $\alpha=\frac{2}{3}\pi$ and let $\alpha,\beta,\gamma,\dotsc$ denote distinct angles. Then a spherical tiling by $12$ angle congruent pentagons must have one of the following (unordered) angle combinations:
\begin{enumerate}
\item $\alpha^5$: The angles in the pentagon are $\alpha,\alpha,\alpha,\alpha,\alpha$. The anglewise vertex combination is $\{20\alpha^3\}$.
\item $\alpha^3\beta\gamma$: The angles in the pentagon are $\alpha,\alpha,\alpha,\beta,\gamma$, satisfying $\beta+\gamma=2\alpha$. The anglewise vertex combination is $\{8\alpha^3,12\alpha\beta\gamma\}$.
\item $\alpha^2\beta^2\gamma$: The angles in the pentagon are $\alpha,\alpha,\beta,\beta,\gamma$, satisfying $2\beta+\gamma=3\alpha$. The anglewise vertex combination is $\{8\alpha^3,12\beta^2\gamma\}$.
\item $\alpha^2\beta\gamma\delta$: The angles in the pentagon are $\alpha,\alpha,\beta,\gamma,\delta$, satisfying $\beta+\gamma+\delta=3\alpha$. The anglewise vertex combination is $\{8\alpha^3,12\beta\gamma\delta\}$.
\item $\alpha\beta\gamma\delta\epsilon$: The angles in the pentagon are $\alpha,\beta,\gamma,\delta,\epsilon$, satisfying $\gamma=3\alpha-2\beta$, $\delta=2\alpha-\beta$, $\epsilon=2\beta-\alpha$. There are three possible anglewise vertex combinations: $\{4\alpha\beta\delta,8\alpha\gamma\epsilon,4\beta^2\gamma,4\delta^2\epsilon\}$, $\{\alpha^3,2\alpha\beta\delta,7\alpha\gamma\epsilon,5\beta^2\gamma,5\delta^2\epsilon\}$,
\newline
$\{2\alpha^3,6\alpha\gamma\epsilon,6\beta^2\gamma,6\delta^2\epsilon\}$.
\end{enumerate}
\end{proposition}

The fourth case degenerates into the second (when $\delta=\alpha$) and the third (when $\delta=\beta$) cases, which further degenerate into the first case. The tiling $T_5$ in Figure \ref{completeclassify} and Figure \ref{generalclass} is the fourth case and its degenerate cases. The tilings $T_1$, $T_2$, $T_3$, $T_4$ in Figure \ref{completeclassify} belong to the third anglewise vertex combination in the fifth case.

\begin{proof}
We divide the proof by considering how many times the angle $\alpha$ appears in the pentagon.

\medskip

{\bf Case 1}: If the angles in the pentagon are $\alpha,\alpha,\alpha,\alpha,\beta$, then Lemma \ref{anglesum} tells us $4\alpha+\beta=5\alpha$, so that the angles are really $\alpha,\alpha,\alpha,\alpha,\alpha$. This is the first case in the proposition.

\medskip

{\bf Case 2}: The angles in the pentagon are $\alpha,\alpha,\alpha,\beta,\gamma$, with $\beta,\gamma\ne\alpha$. Note that $\beta,\gamma$ are not yet assumed to be distinct.

Lemma \ref{anglesum} becomes $\beta+\gamma=2\alpha$. If $\beta^2\gamma$ is a vertex, then $2\beta+\gamma=2\pi=3\alpha$, and we get $\beta=\gamma=\alpha$, a contradiction. By the similar reason, $\beta\gamma^2$ is not a vertex. By Lemma \ref{atype}, the only possible vertices are $\alpha^3$ and $\alpha\beta\gamma$, and $\beta,\gamma$ must be distinct. 

Let $\{m\alpha^3,n\alpha\beta\gamma\}$ be the anglewise vertex combination. Since $12$ tiles with angles $\alpha,\alpha,\alpha,\beta,\gamma$ in each tile have total of $36$ angle $\alpha$, $12$ angle $\beta$ and $12$ angle $\gamma$, we get $3m+n=36$ and $n=12$. Therefore $m=8$ and $n=12$. This is the second case of the proposition.

\medskip

{\bf Case 3}: The angles in the pentagon are $\alpha,\alpha,\beta,\gamma,\delta$, with $\beta,\gamma,\delta\ne\alpha$. Note that $\beta,\gamma,\delta$ are not yet assumed to be distinct.

Lemma \ref{anglesum} becomes $\beta+\gamma+\delta= 3\alpha$. If $\alpha\beta\gamma$ is a vertex, then $\alpha+\beta+\gamma=3\alpha$, and we get $\delta=\alpha$, a contradiction. The argument actually shows that there is no $\alpha\beta\gamma$-type vertex.

If $\beta^2\gamma$ is a vertex, then $2\beta+\gamma=3\alpha$ and $\beta=3\alpha-\beta-\gamma=\delta$. Thus the $\beta^2\gamma$-vertex is really a $\beta\gamma\delta$-vertex. The argument actually applies to any $\beta^2\gamma$-type vertex. By Lemma \ref{atype}, we conclude that any vertex is either an $\alpha^3$-vertex or a $\beta\gamma\delta$-vertex. 

Let $\{m\alpha^3,n\beta\gamma\delta\}$ be the anglewise vertex combination. Without loss of generality, we may further assume that $\gamma$ is different from $\alpha,\beta,\delta$. Since $12$ tiles with angles $\alpha,\alpha,\beta,\gamma,\delta$ have total of $24$ angle $\alpha$ and $12$ angle $\gamma$, we get $3m=24$, $n=12$. Therefore $m=8$ and $n=12$. Depending on whether $\beta=\delta$, we get the third and the fourth cases of the proposition.

\medskip

{\bf Case 4}: The angles in the pentagon are $\alpha,\beta,\gamma,\delta,\epsilon$, with $\beta,\gamma,\delta,\epsilon\ne\alpha$. Again $\beta,\gamma,\delta,\epsilon$ are not yet assumed to be distinct.

Lemma \ref{anglesum} becomes $\beta+\gamma+\delta+\epsilon=4\alpha$. We further divide the case by considering whether some among $\beta,\gamma,\delta,\epsilon$ are equal.

\medskip

{\bf Case 4.1}: $\beta=\gamma=\delta=\epsilon$. 

Lemma \ref{anglesum} becomes $4\beta=4\alpha$, and we get $\beta=\alpha$, a contradiction.

\medskip

{\bf Case 4.2}: $\beta=\delta=\epsilon\ne\gamma$. 

Lemma \ref{anglesum} becomes $3\beta+\gamma=4\alpha$. There is no $\alpha\beta\gamma$-vertex because this implies $\alpha+\beta+\gamma=3\alpha$, and we get $\beta=\gamma=\alpha$. Similar argument shows that $\beta^2\gamma$, $\beta\gamma^2$ are not vertices. By Lemma \ref{atype}, the only vertex is $\alpha^3$, which is impossible.

\medskip

{\bf Case 4.3}: $\beta=\delta\ne \gamma=\epsilon$.

Lemma \ref{anglesum} becomes $\beta+\gamma=2\alpha$. An argument similar to Case 4.2 shows that $\beta^2\gamma$, $\beta\gamma^2$ are not vertices. By Lemma \ref{atype}, the only vertices are $\alpha^3$ and $\alpha\beta\gamma$, so that the anglewise vertex combination is $\{m\alpha^3,n\alpha\beta\gamma\}$. Since $12$ tiles with angles $\alpha,\beta,\beta,\gamma,\gamma$ have total of $12$ angle $\alpha$ and $24$ angle $\beta$, we get $3m+n=12$ and $n=24$. The system has no non-negative solution, a contradiction.

\medskip

{\bf Case 4.4}: $\beta=\epsilon$, and $\beta,\gamma,\delta$ are distinct. 

Similar to Case 4.2, $\alpha\gamma\delta$, $\beta^2\gamma$, $\beta^2\delta$, $\beta\gamma\delta$ are not vertices. By the symmetry of $\gamma$ and $\delta$ in the current case and the exclusion between $\alpha\beta\gamma$ and $\alpha\beta\delta$ (see Lemma \ref{aexclude}), we may assume that the only possible $\alpha\beta\gamma$-type vertex is $\alpha\beta\gamma$. Then by Lemma \ref{atype}, the anglewise vertex combination is
\[
\{m_1\alpha^3,m_2\alpha\beta\gamma,n_1\beta\gamma^2,n_2\beta\delta^2,k_1\gamma^2\delta,k_2\gamma\delta^2\}.
\] 
Since $12$ tiles with angles $\alpha,\beta,\beta,\gamma,\delta$ have total of $12$ angle $\alpha$, $24$ angle $\beta$, $12$ angle $\gamma$ and $12$ angle $\delta$, we get
\begin{align*}
3m_1+m_2 &= 12, \\
m_2+n_1+n_2 &= 24, \\
m_2+2n_1+2k_1+k_2 &= 12, \\
2n_2+k_1+2k_2 &= 12.
\end{align*}
Adding the third and the fourth equations together, we get $m_2+2(n_1+n_2)+3(k_1+k_2)=24$. Compared with the second equation, we get $n_1=n_2=k_1=k_2=0$. Then the second equation becomes $m_2=24$, contradicting to the first equation.

\medskip

{\bf Case 4.5}: $\beta,\gamma,\delta,\epsilon$ are distinct. 

Similar to Case 4.2, there is no $\beta\gamma\delta$-type vertex. By Lemma \ref{atype}, the only possible types are $\alpha^3$, $\alpha\beta\gamma$, $\beta^2\gamma$. We further divide into two subcases.

\medskip

{\bf Case 4.5.1}: $\beta$ appears in two vertices not involving $\alpha$ and with different angle combinations. 

Without loss of generality, we may assume that the two vertices are $\beta^i\gamma^{3-i}$ and $\beta^j\delta^{3-j}$, $0<i,j<3$. By Lemma \ref{aexclude}, we have $i\ne j$. Therefore up to symmetry, we may assume $\beta^2\gamma$, $\beta\delta^2$ are vertices. Solving $2\beta+\gamma=\beta+2\delta=3\alpha$ together with $\beta+\gamma+\delta+\epsilon=4\alpha$ from Lemma \ref{anglesum}, we get
\[
3\beta-2\epsilon=\alpha,\quad
3\gamma+4\epsilon=7\alpha,\quad
3\delta+\epsilon=4\alpha.
\]

Now the angle $\epsilon$ must appear at some vertex. Since there is no $\beta\gamma\delta$-type vertex, by Lemmas \ref{atype} and \ref{aexclude}, one of $\alpha\beta\epsilon$, $\alpha\gamma\epsilon$, $\alpha\delta\epsilon$, $\gamma^2\epsilon$, $\delta^2\epsilon$, $\gamma\epsilon^2$, $\delta\epsilon^2$ must appear as the vertex involving $\epsilon$. If $\alpha\beta\epsilon$ is a vertex, then $\delta+\epsilon=2\alpha$ and $3\beta-2\epsilon=\alpha$ imply $\beta=\epsilon$, a contradiction. All other possibilities lead to similar contradictions.

\medskip

{\bf Case 4.5.2}: The opposite of Case 4.5.1 and its symmetric permutations. 

This means that, without loss of generality, we may assume that the only possible vertices not involving $\alpha$ are $\beta^2\gamma$ and $\delta^2\epsilon$. Then the anglewise vertex combination is $\alpha^3$, $\beta^2\gamma$, $\delta^2\epsilon$ together with some $\alpha\beta\gamma$-type vertices. By Lemma \ref{aexclude}, we find only two possibilities: 
\[
\{m\alpha^3,n_1\alpha\beta\delta,n_2\alpha\gamma\epsilon,k_1\beta^2\gamma,k_2\delta^2\epsilon\},\quad 
\{m\alpha^3,n_1\alpha\beta\epsilon,n_2\alpha\gamma\delta,k_1\beta^2\gamma,k_2\delta^2\epsilon\}.
\]

In the first combination, counting the total number of each angle gives
\[
3m+n_1+n_2 
=n_1+2k_1 
=n_2+k_1
=n_1+2k_2
=n_2+k_2
= 12.
\]
We get $k_1=k_2=m+4$, $n_1=4-2m$, $n_2=8-m$. In particular, we have $n_2,k_1,k_2>0$, so that $\alpha\gamma\epsilon$, $\beta^2\gamma$, $\delta^2\epsilon$ are vertices, and we get 
\[
\gamma+\epsilon=2\alpha,\quad
2\beta+\gamma=3\alpha,\quad
2\delta+\epsilon=3\alpha.
\]
The solution is
\[
\gamma=3\alpha-2\beta,\quad
\delta=2\alpha-\beta,\quad
\epsilon=2\beta-\alpha.
\]
On the other hand, to maintain $n_1\ge 0$, the possible values for $m$ are $0,1,2$. Then we get the corresponding anglewise vertex combinations, which  are the three combinations in the fifth case of the proposition.

The second combination gives
\[
3m+n_1+n_2
=n_1+2k_1
=n_2+k_1
=n_2+2k_2
=n_1+k_2
=12.
\]
The system has no non-negative solution.

\medskip

{\bf Case 5}:  The angles in the pentagon are $\beta,\gamma,\delta,\epsilon,\zeta$, all different from $\alpha$. But the five angles are not yet assumed to be distinct.

Lemma \ref{anglesum} becomes $\beta+\gamma+\delta+\epsilon+\zeta= 5\alpha$. We further divide the case by considering whether some angles are equal.

\medskip

{\bf Case 5.1}: $\beta=\gamma=\delta=\epsilon=\zeta$. 

Lemma \ref{anglesum} implies $\beta=\alpha$, a contradiction.

\medskip

{\bf Case 5.2}: $\beta=\delta=\epsilon=\zeta\ne\gamma$.

The vertices must be $\beta^2\gamma$ or $\beta\gamma^2$, and Lemma \ref{anglesum} becomes $4\beta+\gamma=5\alpha$. If $\beta^2\gamma$ is a vertex, then $2\beta+\gamma=3\alpha$, and we get $\beta=\gamma$, a contradiction. If $\beta\gamma^2$ is a vertex, then we get the similar contradiction.

\medskip

{\bf Case 5.3}: $\beta=\delta=\epsilon\ne\gamma=\zeta$.

The vertices must be $\beta^2\gamma$ or $\beta\gamma^2$. We get contradiction similar to Case 5.2.

\medskip

{\bf Case 5.4}: $\beta=\epsilon=\zeta$, and $\beta,\gamma,\delta$ are distinct. 

Lemma \ref{anglesum} becomes $3\beta+\gamma+\delta=5\alpha$. This implies that $\beta\gamma\delta$ is not a vertex, so that all vertices are $\beta^2\gamma$-type. Since there are three distinct angles, one angle must appear in two vertices with different angle combinations. By Lemma \ref{aexclude} and the symmetry between $\gamma$ and $\delta$, we only need to consider both $\beta^2\gamma$, $\beta\delta^2$ appear, or both $\beta^2\gamma$, $\gamma^2\delta$ appear, or both $\beta\gamma^2$, $\gamma\delta^2$ appear.

If $\beta^2\gamma$, $\beta\delta^2$ are vertices, then $2\beta+\gamma=3\alpha$ and $\beta+2\delta=3\alpha$. Combined with $3\beta+\gamma+\delta=5\alpha$, we get all angles equal, a contradiction. The other pairs of vertices lead to similar contradictions. 

\medskip

{\bf Case 5.5}: $\beta=\epsilon$, $\gamma=\zeta$, and $\beta,\gamma,\delta$ are distinct.

Lemma \ref{anglesum} becomes $2\beta+2\gamma+\delta=5\alpha$. This implies that $\beta^2\delta$, $\gamma^2\delta$, $\beta\gamma\delta$ are not vertices. Since there are three distinct angles and all vertices are $\beta^2\gamma$-type, one angle must appear in two vertices with different angle combinations. By Lemma \ref{aexclude}, the symmetry between $\beta$ and $\gamma$, and the fact that $\beta^2\delta$ and $\gamma^2\delta$ are not vertices, we may assume that $\beta^2\gamma$, $\beta\delta^2$ are vertices. This implies $2\beta+\gamma=3\alpha$ and $\beta+2\delta=3\alpha$. Combined with $2\beta+2\gamma+\delta=5\alpha$, we get all angles equal, a contradiction.

\medskip

{\bf Case 5.6}: $\beta=\zeta$, and $\beta,\gamma,\delta,\epsilon$ are distinct. 

Lemma \ref{anglesum} becomes $2\beta+\gamma+\delta+\epsilon=5\alpha$. This implies that $\gamma\delta\epsilon$ is not a vertex. 

If there is a vertex of $\beta\gamma\delta$-type, then up to symmetry, we may assume $\beta\gamma\delta$ is a vertex. Then we get $\beta+\gamma+\delta=3\alpha$ and $
\beta+\epsilon=5\alpha-(\beta+\gamma+\delta)=2\alpha$. This implies that $\beta,\gamma,\delta$ cannot form $\beta^2\gamma$-type vertices, and $\beta^2\epsilon$, $\beta\epsilon^2$ are not vertices. So the only possible vertices of $\beta^2\gamma$-type are $\gamma^2\epsilon$, $\gamma\epsilon^2$, $\delta^2\epsilon$, $\delta\epsilon^2$. Moreover, by Lemma \ref{aexclude}, the appearance of $\gamma^2\epsilon$ excludes $\gamma\epsilon^2$ and $\delta^2\epsilon$. Then by the symmetry between $\gamma$ and $\delta$ in our current situation, we only need to consider the anglewise vertex combination $\{n\beta\gamma\delta,k_1\gamma^2\epsilon,k_2\delta\epsilon^2\}$. Counting the total number of each angle gives
\[
n=24,\quad
n+2k_1
=n+k_2
=k_1+2k_2
=12.
\]
The system has no non-negative solution.

So all vertices are of $\beta^2\gamma$-type, and the anglewise vertex combination is
\[
\{m_1\beta^2\gamma,m_2\beta^2\delta,m_3\beta^2\epsilon,n_1\beta\gamma^2,n_2\beta\delta^2,n_3\beta\epsilon^2,k_1\gamma^2\delta,k_2\gamma^2\epsilon,k_3\gamma\delta^2,k_4\gamma\epsilon^2,k_5\delta^2\epsilon,k_6\delta\epsilon^2\}.
\]
Counting the total number of the angle $\beta$ gives $2(m_1+m_2+m_3)+n_1+n_2+n_3=24$.

If $m_1=m_2=m_3=0$ and $n_1>0$, then by Lemma \ref{aexclude}, we get $n_2=n_3=0$ and $n_1=24$. Therefore the total number of the angle $\gamma$ is $\ge 2n_1=48$, a contradiction. Similar argument rules out $n_2>0$ and $n_3>0$. Therefore $m_1,m_2,m_3$ cannot be all $0$, and we may assume $m_1>0$ without loss of generality. By Lemma \ref{aexclude}, this implies $m_2=m_3=n_1=k_3=k_4=0$, and we get $2m_1+n_2+n_3=24$. 

If $n_2=n_3=0$, then we get $m_1=12$. Thus the $12$ $\beta^2\gamma$-vertices already contain all $12$ appearances of the angle $\gamma$. Therefore $\gamma$ does not appear in any other vertex, and we are left with the anglewise vertex combination $\{12\beta^2\gamma,k_5\delta^2\epsilon, k_6\delta\epsilon^2\}$. By counting the total numbers of the angles $\delta$ and $\epsilon$, we get $2k_5+k_6=k_5+2k_6=12$. This implies $k_5=k_6=4>0$. Therefore both $\delta^2\epsilon$ and $\delta\epsilon^2$ are vertices, a contradiction to Lemma \ref{aexclude}. 

So we may assume $n_2>0$. By Lemma \ref{aexclude}, we get $n_3=k_5=0$, so that the anglewise vertex combination becomes $\{m_1\beta^2\gamma,n_2\beta\delta^2,k_1\gamma^2\delta,k_2\gamma^2\epsilon,k_6\delta\epsilon^2\}$. By counting the total number of each angle, we get
\[
2m_1+n_2=24, \quad
m_1+2k_1+2k_2
=2n_2+k_1+k_6
=k_2+2k_6
=12.
\]
By Lemma \ref{aexclude}, we must have either $k_1=0$ or $k_2=0$. If $k_1=0$, the system gives $k_2=\frac{4}{5}$. If $k_2=0$, the system gives $k_1=\frac{2}{3}$. Since the numbers are not integers, we get a contradiction.

\medskip

{\bf Case 5.7}: $\beta,\gamma,\delta,\epsilon,\zeta$ are distinct.  

\medskip

{\bf Case 5.7.1}: There are $\beta\gamma\delta$-type vertices. 

We may assume $\beta\gamma\delta$ is a vertex. This implies $\beta+\gamma+\delta=3\alpha$, so that $\beta,\gamma,\delta$ cannot form $\beta^2\gamma$-type vertices. By Lemma \ref{anglesum}, we have $\epsilon+\zeta=5\alpha-(\beta+\gamma+\delta)=2\alpha$, so that $\epsilon,\zeta$ cannot appear at the same vertex. By Lemma \ref{aexclude}, the only $\beta\gamma\delta$-type vertex is $\beta\gamma\delta$, and the only way for $\epsilon,\zeta$ to appear is to get combined with angles from $\beta,\gamma,\delta$.

If $\epsilon,\zeta$ can be combined with the same angle from $\beta,\gamma,\delta$, then up to symmetry, we may assume $\beta\gamma\delta$, $\beta^2\epsilon$, $\beta\zeta^2$ are vertices. We get $\beta+\gamma+\delta=2\beta+\epsilon=\beta+2\zeta=3\alpha$. Together with $\epsilon+\zeta=2\alpha$, we get $\delta=\alpha$, a contradiction.

So $\epsilon,\zeta$ must be combined with different angles from $\beta,\gamma,\delta$. By Lemma \ref{aexclude} and up to symmetry, there are three possible combinations. 

In the first combination, $\beta\gamma\delta$, $\beta^2\epsilon$, $\gamma^2\zeta$ are vertices. We get $\beta+\gamma+\delta=2\beta+\epsilon=2\gamma+\zeta=3\alpha$. Together with $\epsilon+\zeta=2\alpha$, we get $\delta=\alpha$, a contradiction.

In the second combination, $\beta\gamma\delta$, $\beta\epsilon^2$, $\gamma\zeta^2$ are vertices. Similar to the first combination, we also get the contradiction $\delta=\alpha$.

In the third combination, $\beta\gamma\delta$, $\beta^2\epsilon$, $\gamma\zeta^2$ are vertices. We get $\beta+\gamma+\delta=2\beta+\epsilon=\gamma+2\zeta=3\alpha$. Together with $\epsilon+\zeta=2\alpha$, we do not get a contradiction as before. However, if $\beta\gamma\delta$, $\beta^2\epsilon$, $\gamma\zeta^2$ are the only vertices, then the total number of $\beta$ will be strictly bigger than the total number of $\delta$. Since both total numbers are actually $12$, we conclude that these cannot be all the vertices. By Lemma \ref{aexclude}, the only other possible vertices are  $\delta\epsilon^2$, $\delta^2\zeta$. If $\delta\epsilon^2$ is a vertex, then $\delta+2\epsilon=3\alpha$. Adding this to the equations above, we get all angles equal, a contradiction. If $\delta^2\zeta$ is a vertex, then we get a similar contradiction. 

\medskip

{\bf Case 5.7.2}: There is no $\beta\gamma\delta$-type vertex. 

We may assume $\beta^2\gamma$ is a vertex. Since there is no $\beta\gamma\delta$-type vertex, by Lemma \ref{aexclude}, the only other vertex involving $\beta$ is $\beta\theta^2$ for a unique $\theta$, and the only other vertex involving $\gamma$ is $\gamma^2\rho$ for a unique $\rho$. Since each of three distinct angles $\delta,\epsilon,\zeta$ must appear at some vertex, there must be some pairing among $\delta,\epsilon,\zeta$. Without loss of generality, therefore, we may assume both $\beta^2\gamma$, $\delta^2\epsilon$ are vertices. Moreover, up to symmetry, we may further assume that $\zeta$ is paired with $\beta$ or $\gamma$. Given the existence of $\beta^2\gamma$, the only possibilities are $\beta\zeta^2$ and $\gamma^2\zeta$. So either $\beta^2\gamma$, $\delta^2\epsilon$, $\beta\zeta^2$ are vertices, or $\beta^2\gamma$, $\delta^2\epsilon$, $\gamma^2\zeta$ are vertices.

If $\beta^2\gamma$, $\delta^2\epsilon$, $\beta\zeta^2$ are vertices, then $2\beta+\gamma=2\delta+\epsilon=\beta+2\zeta=3\alpha$. Combined with $\beta+\gamma+\delta+\epsilon+\zeta=5\alpha$ from Lemma \ref{anglesum}, we get the following pairwise relations
\begin{align*}
3\beta+2\delta&=5\alpha,&
3\beta-\epsilon&=2\alpha,&
3\gamma-4\delta&=-\alpha,&
3\gamma+2\epsilon&=5\alpha,\\
\gamma-4\zeta&=-3\alpha,&
-\delta+3\zeta&=2\alpha,&
\epsilon+6\zeta&=7\alpha.
\end{align*}
The relations exclude all other vertices, which must be of $\beta^2\gamma$-type. Therefore the anglewise vertex combination is $\{m\beta^2\gamma,n\delta^2\epsilon,k\beta\zeta^2\}$. This implies that the total number of $\beta$ is strictly bigger than the total number of $\gamma$. Since both numbers are $12$, we get a contradiction.

If $\beta^2\gamma$, $\delta^2\epsilon$, $\gamma^2\zeta$ are vertices, then $2\beta+\gamma=2\delta+\epsilon=2\gamma+\zeta=3\alpha$. Similar to the previous case, the equations together with $\beta+\gamma+\delta+\epsilon+\zeta=5\alpha$ from Lemma \ref{anglesum} exclude all other vertices. Therefore the anglewise vertex combination is $\{m\beta^2\gamma, n\delta^2\epsilon, k\gamma^2\zeta\}$. Then the total number of $\delta$ is more than the total number of $\epsilon$, a contradiction.
\end{proof}

\section{Classification of Angle Congruent Tiling}
\label{angletile2}

Proposition \ref{angle} classifies the numerics of spherical tilings by $12$ angle congruent pentagons. In this section, we study how the numerical information can fit the combinatorial structure in Figure \ref{tile_pattern_pic}, similar to what we have done for edges in Section \ref{edgetile2}.

In Section \ref{edgetile2}, we introduced the notations $P_i$, $E_{ij}$, $V_{ijk}$ for the tiles, the edges and the vertices in Figure \ref{tile_pattern_pic}. In this section, we denote by $A_{i,jk}$ the angle of the tile $P_i$ at the vertex $V_{ijk}$. We will even write $V_{ijk}=\alpha\beta\gamma$ if we know $V_{ijk}$ is an $\alpha\beta\gamma$-vertex.

Inside a specified tile, we can also name a vertex by its angle and an edge by the two angles at the ends. For example, in Figure \ref{generalclass}, $A_{1,23}$ is a $\delta$-angle, $A_{1,26}$ and $A_{1,34}$ are $\alpha$-angles. Moreover, $E_{12}$ and $E_{13}$ are $\alpha\delta$-edges of $P_1$.

In this section, we let $\alpha=\frac{2}{3}\pi$ and assume $\alpha, \beta,\gamma,\dotsc$ are always distinct.

There is not much to classify for the angle combination $\alpha^5$. Basically, we assign $\alpha$ to all angles in Figure \ref{tile_pattern_pic}. This can be realized by the regular dodecahedron tiling.

There are many angle congruent tilings for the angle combination $\alpha^3\beta\gamma$. Fortunately, the classification is not needed in the proof of the main theorem. We simply present some examples in Figure \ref{angleexamples}. Note that for any thick line in any tiling, we can exchange the angles $\beta$ and $\gamma$ along the line and still get an angle congruent tiling. 

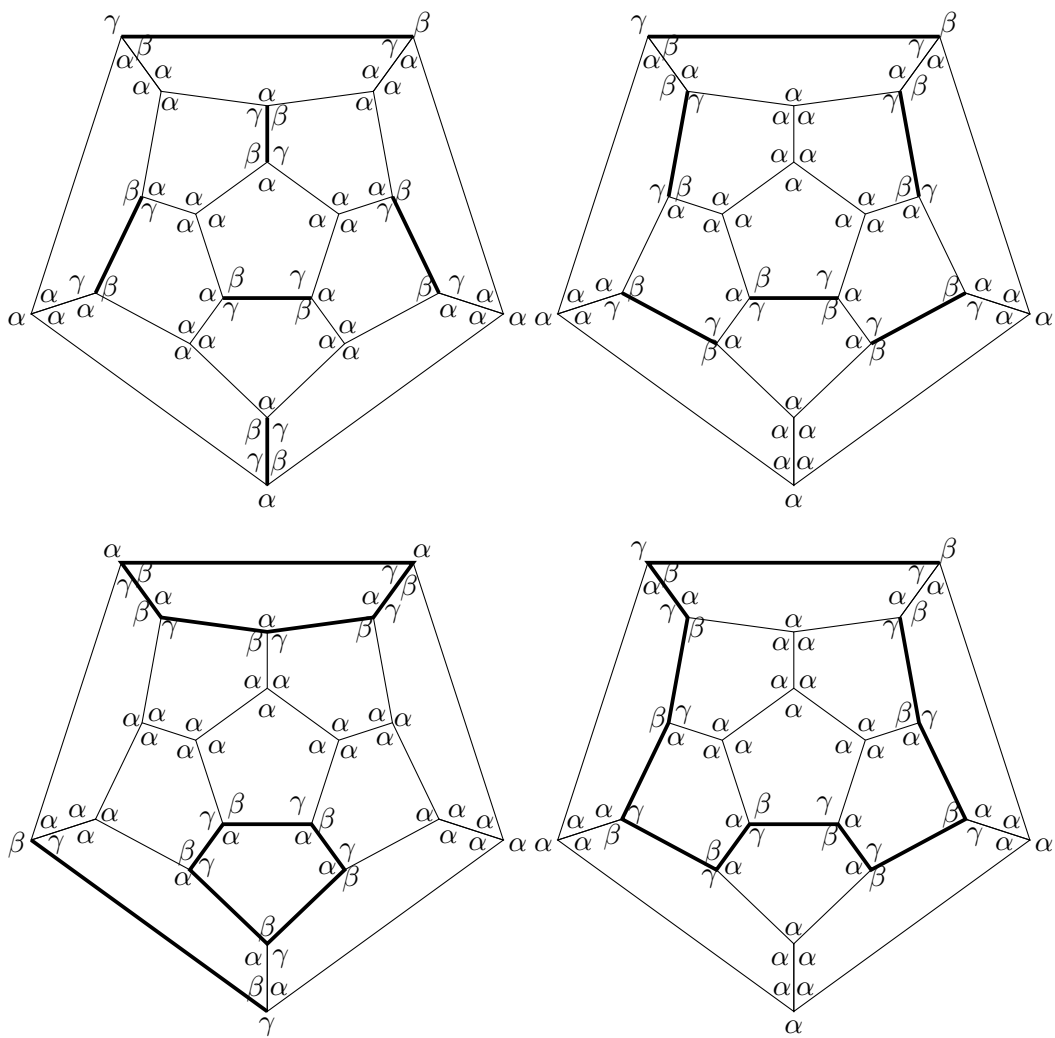
\begin{figure}[htp]
\centering
    \begin{tikzpicture}[scale=1]


\foreach \y in {0,7}
\foreach \z in {0,-7}
{
\begin{scope}[shift={(\y,\z)}]
	\foreach \x in {1,...,5}
    \draw 
    		(-54+72*\x:1) -- (18+72*\x:1)
    		(-54+72*\x:1) -- (-54+72*\x:1.75) -- (-18+72*\x:2.4) -- (18+72*\x:1.75)
		(-18+72*\x:2.4) -- (-18+72*\x:3.3) -- (54+72*\x:3.3) -- (54+72*\x:2.4);
    
\end{scope}
}


    \draw[line width=0.05cm]
		(234:1) -- (-54:1)
		(90:1) -- (90:1.75)
		(54:3.3) -- (126:3.3)
		(18:1.75) -- (-18:2.4)
		(162:1.75) -- (198:2.4)
		(-90:2.4) -- (-90:3.3);

    \node at (18:0.7) {\small $\alpha$};
    \node at (8:1.1) {\small $\alpha$};
    \node at (28:1.1) {\small $\alpha$};
    
    \node at (90:0.7) {\small $\alpha$};
    \node at (80:1.1) {\small $\gamma$};
    \node at (100:1.1) {\small $\beta$};
    
    \node at (162:0.7) {\small $\alpha$};
    \node at (152:1.1) {\small $\alpha$};
    \node at (172:1.1) {\small $\alpha$};
    
    \node at (234:0.7) {\small $\beta$};
    \node at (224:1.1) {\small $\alpha$};
    \node at (244:1.1) {\small $\gamma$};
    
    \node at (-54:0.7) {\small $\gamma$};
    \node at (-44:1.1) {\small $\alpha$};
    \node at (-64:1.1) {\small $\beta$};

    \node at (18:1.9) {\small $\beta$};
    \node at (12:1.6) {\small $\gamma$};
    \node at (24:1.6) {\small $\alpha$};
    
    \node at (90:1.9) {\small $\alpha$};
    \node at (84:1.6) {\small $\beta$};
    \node at (96:1.6) {\small $\gamma$};
    
    \node at (162:1.9) {\small $\beta$};
    \node at (156:1.6) {\small $\alpha$};
    \node at (168:1.6) {\small $\gamma$};
    
    \node at (234:1.9) {\small $\alpha$};
    \node at (228:1.6) {\small $\alpha$};
    \node at (240:1.6) {\small $\alpha$};
    
    \node at (-54:1.9) {\small $\alpha$};
    \node at (-48:1.6) {\small $\alpha$};
    \node at (-60:1.6) {\small $\alpha$};

    \node at (54:2.2) {\small $\alpha$};
    \node at (50:2.6) {\small $\alpha$};
    \node at (58:2.6) {\small $\alpha$};
    
    \node at (126:2.2) {\small $\alpha$};
    \node at (122:2.6) {\small $\alpha$};
    \node at (130:2.6) {\small $\alpha$};
    
    \node at (198:2.2) {\small $\beta$};
    \node at (194:2.6) {\small $\gamma$};
    \node at (202:2.6) {\small $\alpha$};

    \node at (-90:2.2) {\small $\alpha$};
    \node at (-94:2.6) {\small $\beta$};
    \node at (-86:2.6) {\small $\gamma$};
    
    \node at (-18:2.2) {\small $\beta$};
    \node at (-22:2.6) {\small $\alpha$};
    \node at (-14:2.6) {\small $\gamma$};

    \node at (54:3.5) {\small $\beta$};
    \node at (51:3) {\small $\alpha$};
    \node at (57:3) {\small $\gamma$};

    \node at (126:3.5) {\small $\gamma$};
    \node at (123:3) {\small $\beta$};
    \node at (129:3) {\small $\alpha$};

    \node at (198:3.5) {\small $\alpha$};
    \node at (195:3) {\small $\alpha$};
    \node at (201:3) {\small $\alpha$};

    \node at (-90:3.5) {\small $\alpha$};
    \node at (-93:3) {\small $\gamma$};
    \node at (-87:3) {\small $\beta$};

    \node at (-18:3.5) {\small $\alpha$};
    \node at (-15:3) {\small $\alpha$};
    \node at (-21:3) {\small $\alpha$};


\begin{scope}[xshift=7cm]

\draw[line width=0.05cm]
		(234:1) -- (-54:1)
		(18:1.75) -- (54:2.4)
		(162:1.75) -- (126:2.4)
		(54:3.3) -- (126:3.3)
		(234:1.75) -- (198:2.4)
		(-54:1.75) -- (-18:2.4);

    \node at (18:0.7) {\small $\alpha$};
    \node at (8:1.1) {\small $\alpha$};
    \node at (28:1.1) {\small $\alpha$};
    
    \node at (90:0.7) {\small $\alpha$};
    \node at (80:1.1) {\small $\alpha$};
    \node at (100:1.1) {\small $\alpha$};
    
    \node at (162:0.7) {\small $\alpha$};
    \node at (152:1.1) {\small $\alpha$};
    \node at (172:1.1) {\small $\alpha$};
    
    \node at (234:0.7) {\small $\beta$};
    \node at (224:1.1) {\small $\alpha$};
    \node at (244:1.1) {\small $\gamma$};
    
    \node at (-54:0.7) {\small $\gamma$};
    \node at (-44:1.1) {\small $\alpha$};
    \node at (-64:1.1) {\small $\beta$};

    \node at (18:1.9) {\small $\gamma$};
    \node at (12:1.6) {\small $\alpha$};
    \node at (24:1.6) {\small $\beta$};
    
    \node at (90:1.9) {\small $\alpha$};
    \node at (84:1.6) {\small $\alpha$};
    \node at (96:1.6) {\small $\alpha$};
    
    \node at (162:1.9) {\small $\gamma$};
    \node at (156:1.6) {\small $\beta$};
    \node at (168:1.6) {\small $\alpha$};
    
    \node at (234:1.9) {\small $\beta$};
    \node at (228:1.6) {\small $\gamma$};
    \node at (240:1.6) {\small $\alpha$};
    
    \node at (-54:1.9) {\small $\beta$};
    \node at (-48:1.6) {\small $\gamma$};
    \node at (-60:1.6) {\small $\alpha$};

    \node at (54:2.2) {\small $\gamma$};
    \node at (50:2.6) {\small $\beta$};
    \node at (58:2.6) {\small $\alpha$};
    
    \node at (126:2.2) {\small $\gamma$};
    \node at (122:2.6) {\small $\alpha$};
    \node at (130:2.6) {\small $\beta$};
    
    \node at (198:2.2) {\small $\beta$};
    \node at (194:2.6) {\small $\alpha$};
    \node at (202:2.6) {\small $\gamma$};

    \node at (-90:2.2) {\small $\alpha$};
    \node at (-94:2.6) {\small $\alpha$};
    \node at (-86:2.6) {\small $\alpha$};
    
    \node at (-18:2.2) {\small $\beta$};
    \node at (-22:2.6) {\small $\gamma$};
    \node at (-14:2.6) {\small $\alpha$};

    \node at (54:3.5) {\small $\beta$};
    \node at (51:3) {\small $\alpha$};
    \node at (57:3) {\small $\gamma$};

    \node at (126:3.5) {\small $\gamma$};
    \node at (123:3) {\small $\beta$};
    \node at (129:3) {\small $\alpha$};

    \node at (198:3.5) {\small $\alpha$};
    \node at (195:3) {\small $\alpha$};
    \node at (201:3) {\small $\alpha$};

    \node at (-90:3.5) {\small $\alpha$};
    \node at (-93:3) {\small $\alpha$};
    \node at (-87:3) {\small $\alpha$};

    \node at (-18:3.5) {\small $\alpha$};
    \node at (-15:3) {\small $\alpha$};
    \node at (-21:3) {\small $\alpha$};
    
    \end{scope}


    \begin{scope}[yshift=-7cm]

\draw[line width=0.05cm]
		(54:3.3) -- (54:2.4) -- (90:1.75) -- (126:2.4) -- (126:3.3) -- cycle
		(234:1.75) -- (234:1) -- (-54:1) -- (-54:1.75) -- (-90:2.4) -- cycle
		(-90:3.3) -- (198:3.3);

    \node at (18:0.7) {\small $\alpha$};
    \node at (8:1.1) {\small $\alpha$};
    \node at (28:1.1) {\small $\alpha$};
    
    \node at (90:0.7) {\small $\alpha$};
    \node at (80:1.1) {\small $\alpha$};
    \node at (100:1.1) {\small $\alpha$};
    
    \node at (162:0.7) {\small $\alpha$};
    \node at (152:1.1) {\small $\alpha$};
    \node at (172:1.1) {\small $\alpha$};
    
    \node at (234:0.7) {\small $\beta$};
    \node at (224:1.1) {\small $\gamma$};
    \node at (244:1.1) {\small $\alpha$};
    
    \node at (-54:0.7) {\small $\gamma$};
    \node at (-44:1.1) {\small $\beta$};
    \node at (-64:1.1) {\small $\alpha$};

    \node at (18:1.9) {\small $\alpha$};
    \node at (12:1.6) {\small $\alpha$};
    \node at (24:1.6) {\small $\alpha$};
    
    \node at (90:1.9) {\small $\alpha$};
    \node at (84:1.6) {\small $\gamma$};
    \node at (96:1.6) {\small $\beta$};
    
    \node at (162:1.9) {\small $\alpha$};
    \node at (156:1.6) {\small $\alpha$};
    \node at (168:1.6) {\small $\alpha$};
    
    \node at (234:1.9) {\small $\alpha$};
    \node at (228:1.6) {\small $\beta$};
    \node at (240:1.6) {\small $\gamma$};
    
    \node at (-54:1.9) {\small $\beta$};
    \node at (-48:1.6) {\small $\gamma$};
    \node at (-60:1.6) {\small $\alpha$};

    \node at (54:2.2) {\small $\beta$};
    \node at (50:2.6) {\small $\gamma$};
    \node at (58:2.6) {\small $\alpha$};
    
    \node at (126:2.2) {\small $\gamma$};
    \node at (122:2.6) {\small $\alpha$};
    \node at (130:2.6) {\small $\beta$};
    
    \node at (198:2.2) {\small $\alpha$};
    \node at (194:2.6) {\small $\alpha$};
    \node at (202:2.6) {\small $\alpha$};

    \node at (-90:2.2) {\small $\beta$};
    \node at (-94:2.6) {\small $\alpha$};
    \node at (-86:2.6) {\small $\gamma$};
    
    \node at (-18:2.2) {\small $\alpha$};
    \node at (-22:2.6) {\small $\alpha$};
    \node at (-14:2.6) {\small $\alpha$};

    \node at (54:3.5) {\small $\alpha$};
    \node at (51:3) {\small $\beta$};
    \node at (57:3) {\small $\gamma$};

    \node at (126:3.5) {\small $\alpha$};
    \node at (123:3) {\small $\beta$};
    \node at (129:3) {\small $\gamma$};

    \node at (198:3.5) {\small $\beta$};
    \node at (195:3) {\small $\alpha$};
    \node at (201:3) {\small $\gamma$};

    \node at (-90:3.5) {\small $\gamma$};
    \node at (-93:3) {\small $\beta$};
    \node at (-87:3) {\small $\alpha$};

    \node at (-18:3.5) {\small $\alpha$};
    \node at (-15:3) {\small $\alpha$};
    \node at (-21:3) {\small $\alpha$};
    
    \end{scope}


\begin{scope}[shift={(7cm,-7cm)}]

\draw[line width=0.05cm]
		(54:2.4) -- (18:1.75) -- (-18:2.4) -- (-54:1.75) -- (-54:1) -- (234:1) -- (234:1.75) -- (198:2.4) -- (162:1.75) -- (126:2.4) -- (126:3.3) -- (54:3.3);

    \node at (18:0.7) {\small $\alpha$};
    \node at (8:1.1) {\small $\alpha$};
    \node at (28:1.1) {\small $\alpha$};
    
    \node at (90:0.7) {\small $\alpha$};
    \node at (80:1.1) {\small $\alpha$};
    \node at (100:1.1) {\small $\alpha$};
    
    \node at (162:0.7) {\small $\alpha$};
    \node at (152:1.1) {\small $\alpha$};
    \node at (172:1.1) {\small $\alpha$};
    
    \node at (234:0.7) {\small $\beta$};
    \node at (224:1.1) {\small $\alpha$};
    \node at (244:1.1) {\small $\gamma$};
    
    \node at (-54:0.7) {\small $\gamma$};
    \node at (-44:1.1) {\small $\alpha$};
    \node at (-64:1.1) {\small $\beta$};

    \node at (18:1.9) {\small $\gamma$};
    \node at (12:1.6) {\small $\alpha$};
    \node at (24:1.6) {\small $\beta$};
    
    \node at (90:1.9) {\small $\alpha$};
    \node at (84:1.6) {\small $\alpha$};
    \node at (96:1.6) {\small $\alpha$};
    
    \node at (162:1.9) {\small $\beta$};
    \node at (156:1.6) {\small $\gamma$};
    \node at (168:1.6) {\small $\alpha$};
    
    \node at (234:1.9) {\small $\gamma$};
    \node at (228:1.6) {\small $\beta$};
    \node at (240:1.6) {\small $\alpha$};
    
    \node at (-54:1.9) {\small $\beta$};
    \node at (-48:1.6) {\small $\gamma$};
    \node at (-60:1.6) {\small $\alpha$};

    \node at (54:2.2) {\small $\gamma$};
    \node at (50:2.6) {\small $\beta$};
    \node at (58:2.6) {\small $\alpha$};
    
    \node at (126:2.2) {\small $\beta$};
    \node at (122:2.6) {\small $\alpha$};
    \node at (130:2.6) {\small $\gamma$};
    
    \node at (198:2.2) {\small $\gamma$};
    \node at (194:2.6) {\small $\alpha$};
    \node at (202:2.6) {\small $\beta$};

    \node at (-90:2.2) {\small $\alpha$};
    \node at (-94:2.6) {\small $\alpha$};
    \node at (-86:2.6) {\small $\alpha$};
    
    \node at (-18:2.2) {\small $\beta$};
    \node at (-22:2.6) {\small $\gamma$};
    \node at (-14:2.6) {\small $\alpha$};

    \node at (54:3.5) {\small $\beta$};
    \node at (51:3) {\small $\alpha$};
    \node at (57:3) {\small $\gamma$};

    \node at (126:3.5) {\small $\gamma$};
    \node at (123:3) {\small $\beta$};
    \node at (129:3) {\small $\alpha$};

    \node at (198:3.5) {\small $\alpha$};
    \node at (195:3) {\small $\alpha$};
    \node at (201:3) {\small $\alpha$};

    \node at (-90:3.5) {\small $\alpha$};
    \node at (-93:3) {\small $\alpha$};
    \node at (-87:3) {\small $\alpha$};

    \node at (-18:3.5) {\small $\alpha$};
    \node at (-15:3) {\small $\alpha$};
    \node at (-21:3) {\small $\alpha$};
    
    \end{scope}

    \end{tikzpicture}
\caption{Some tilings for the angle combination $\alpha^3\beta\gamma$.}
\label{angleexamples}
\end{figure}

\begin{proposition}\label{angle_pattern3}
The spherical tilings by $12$ angle congruent pentagons with angles $\alpha,\alpha,\beta,\beta,\gamma$, where $\alpha=\frac{2\pi}{3}$ and $\alpha,\beta,\gamma$ are distinct, are given by Figure \ref{2a2b1cB} up to symmetry.
Note that in the tiling on the right, we can exchange the four angles $\beta$ and $\gamma$ around any thick line and still get a tiling. 
\end{proposition}

As an example of the exchange, on the right of Figure \ref{2a2b1cB}, we can exchange $\beta$ and $\gamma$ around the edge $E_{15}$ to get $A_{1,45}=A_{5,16}=\gamma$ and $A_{1,56}=A_{5,14}=\beta$. But we should keep $A_{4,15}=A_{6,15}=\beta$. The result is also an angle congruent tiling.

\begin{proof}
We are in the third case of Proposition \ref{angle}. The anglewise vertex combination is $\{8\alpha^3,12\beta^2\gamma\}$, which implies the following AVC (for anglewise vertex combination) condition.

\medskip

{\bf AVC}(for $\alpha^2\beta^2\gamma$): Any vertex is either $\alpha^3$ or $\beta^2\gamma$. 

\medskip

Up to symmetry, there are four possible ways of arranging the angles in the pentagon, described in Figure \ref{2a2b1c}. We start by assuming the angles of the tile $P_1$ in Figure \ref{tile_pattern_pic} are arranged as in Figure \ref{2a2b1c}.

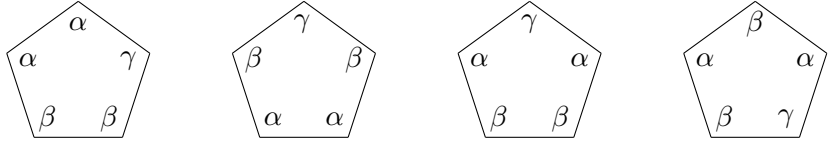
\begin{figure}[htp]
\centering
    \begin{tikzpicture}[scale=1]
    
   
    \foreach \x in {1,...,5}
    \draw (-54+72*\x:1) -- (18+72*\x:1);
	
    \node at (18:0.7) {\small$\gamma$};
    \node at (90:0.7) {\small$\alpha$};
    \node at (162:0.7) {\small$\alpha$};
    \node at (234:0.7) {\small$\beta$};
    \node at (-54:0.7) {\small$\beta$};


    \begin{scope}[xshift = 3cm]
    
    \foreach \x in {1,...,5}
    \draw (-54+72*\x:1) -- (18+72*\x:1);
	
	\node at (18:0.7) {\small$\beta$};
    \node at (92:0.7) {\small$\gamma$};
    \node at (162:0.7) {\small$\beta$};
    \node at (234:0.7) {\small$\alpha$};
    \node at (-54:0.7) {\small$\alpha$};
    
    \end{scope}


    \begin{scope}[xshift = 6cm]
    
    \foreach \x in {1,...,5}
    \draw (-54+72*\x:1) -- (18+72*\x:1);
	
	\node at (18:0.7) {\small$\alpha$};
    \node at (90:0.7) {\small$\gamma$};
    \node at (162:0.7) {\small$\alpha$};
    \node at (234:0.7) {\small$\beta$};
    \node at (-54:0.7) {\small$\beta$};
    
    \end{scope}

    
    \begin{scope}[xshift = 9cm]
    
    \foreach \x in {1,...,5}
    \draw (-54+72*\x:1) -- (18+72*\x:1);
	
	\node at (18:0.7) {\small$\alpha$};
    \node at (90:0.7) {\small$\beta$};
    \node at (162:0.7) {\small$\alpha$};
    \node at (234:0.7) {\small$\beta$};
    \node at (-54:0.7) {\small$\gamma$};
    
    \end{scope}

    \end{tikzpicture}
\caption{Four possible arrangements for the angle combination $\alpha^2\beta^2\gamma$.}
\label{2a2b1c}
\end{figure}

\medskip

{\bf Arrangement 1}

The case is the left of Figure \ref{2a2b1cA}. By $A_{1,26}=\gamma$ and the AVC condition, we get $V_{126}=\beta^2\gamma$, so that $A_{2,16}=A_{6,12}=\beta$. By $A_{1,23}=A_{1,34}=\alpha$ and the AVC condition, we get $V_{123}=V_{134}=\alpha^3$, so that $A_{2,13}=A_{3,12}=A_{3,14}=A_{4,13}=\alpha$. By $A_{2,13}=\alpha$, $A_{2,16}=\beta$, we get all the angles of $P_2$. Then by $A_{2,37}=\alpha$ and the AVC condition, we get $V_{237}=\alpha^3$, so that there are three $\alpha$ in $P_3$, a contradiction.

\begin{figure}[htp]
\centering
    \begin{tikzpicture}[scale=1]

	\foreach \x in {1,...,5}
    \draw (-54+72*\x:1) -- (18+72*\x:1);
    
    \draw 
    		(90:1) -- (90:1.75) -- (54:2.5) -- (18:1.75) -- (18:1)
    		(90:1.75) -- (126:2.5) -- (162:1.75) -- (162:1)
		(234:1) -- (234:1.75) -- (198:2.5) -- (162:1.75)
		(-54:1) -- (-54:1.75) -- (-18:2.5) -- (18:1.75)
		(54:2.5) -- (54:3.3) -- (126:3.3) -- (126:2.5);

    \node at (0:0) {\small $1$};
    \node at (54:1.5) {\small $2$};
    \node at (126:1.5) {\small $3$};
    \node at (198:1.5) {\small $4$};
    \node at (-18:1.5) {\small $6$};
    \node at (90:2.4) {\small $7$};
    
    \node at (18:0.7) {\small $\gamma$};
    \node at (8:1.1) {\small $\beta$};
    \node at (28:1.1) {\small $\beta$};
    
    \node at (90:0.7) {\small $\alpha$};
    \node at (80:1.1) {\small $\alpha$};
    \node at (100:1.1) {\small $\alpha$};
    
    \node at (162:0.7) {\small $\alpha$};
    \node at (152:1.1) {\small $\alpha$};
    \node at (172:1.1) {\small $\alpha$};
    
    \node at (234:0.7) {\small $\beta$};
    
    \node at (-54:0.7) {\small $\beta$};
    
    \node at (90:1.9) {\small $\alpha$};
    \node at (84:1.6) {\small $\alpha$};
    \node at (96:1.6) {\small $\alpha$};
    
    \node at (54:2.2) {\small $\gamma$};
    
    \node at (28:1.6) {\small $\beta$};

	\begin{scope}[xshift=6cm]
	
	\foreach \x in {1,...,5}
    \draw (-54+72*\x:1) -- (18+72*\x:1);
    
    \draw 
    		(90:1) -- (90:1.75) -- (54:2.5) -- (18:1.75) -- (18:1)
    		(90:1.75) -- (126:2.5) -- (162:1.75) -- (162:1)
		(234:1) -- (234:1.75) -- (198:2.5) -- (162:1.75)
		(-54:1) -- (-54:1.75) -- (-18:2.5) -- (18:1.75)
		(54:2.5) -- (54:3.3) -- (126:3.3) -- (126:2.5);

    \node at (0:0) {\small $1$};
    \node at (54:1.5) {\small $2$};
    \node at (126:1.5) {\small $3$};
    \node at (198:1.5) {\small $4$};
    \node at (-18:1.5) {\small $6$};
    \node at (90:2.4) {\small $7$};
    \node at (162:2.4) {\small $8$};
    \node at (18:2.4) {\small $11$};
    
    \node at (18:0.7) {\small $\beta$};
    \node at (8:1.1) {\small $\beta$};
    \node at (28:1.1) {\small $\gamma$};
    
    \node at (90:0.7) {\small $\gamma$};
    \node at (80:1.1) {\small $\beta$};
    \node at (100:1.1) {\small $\beta$};
    
    \node at (162:0.7) {\small $\beta$};
    \node at (152:1.1) {\small $\gamma$};
    \node at (172:1.1) {\small $\beta$};
    
    \node at (234:0.7) {\small $\alpha$};
    
    \node at (-54:0.7) {\small $\alpha$};
    
    \node at (90:1.9) {\small $\alpha$};
    \node at (84:1.6) {\small $\alpha$};
    \node at (96:1.6) {\small $\alpha$};
    
    \node at (28:1.6) {\small $\beta$};
    
    \node at (152:1.6) {\small $\beta$};
    
    \node at (54:2.2) {\small $\alpha$};
    \node at (50:2.6) {\small $\alpha$};
    \node at (58:2.6) {\small $\alpha$};
    
    \node at (126:2.2) {\small $\alpha$};
    \node at (122:2.6) {\small $\alpha$};
    \node at (130:2.6) {\small $\alpha$};

    \end{scope}

    \end{tikzpicture}
\caption{Tilings for the angle combination $\alpha^2\beta^2\gamma$.}
\label{2a2b1cA}
\end{figure}
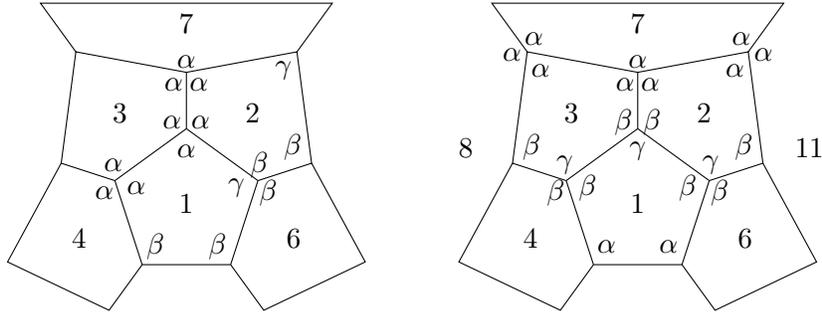

\medskip

{\bf Arrangement 2}

The case is the right of Figure \ref{2a2b1cA}. By $A_{1,23}=\gamma$ and AVC, we get $A_{2,13}=A_{3,12}=\beta$. By $A_{1,26}=\beta$ and AVC, we get $V_{126}=\beta^2\gamma$, so that $A_{2,16}=\beta$ or $\gamma$. Since the two $\beta$ in $P_2$ are not adjacent, we get $A_{2,16}=\gamma$. Together with $A_{2,13}=\beta$, we get all the angles of $P_2$. By the same reason, we get all the angles of $P_3$. Then by AVC, we get $V_{237}=V_{378}=V_{27\overline{11}}=\alpha^3$, so that there are three $\alpha$ in $P_7$, a contradiction.

\medskip

{\bf Arrangement 3}

The case is the left of Figure \ref{2a2b1cB}. By AVC, we get all the angles at $V_{123}$, $V_{134}$, $V_{126}$. By $A_{2,16}=A_{3,14}=\alpha$, $A_{2,13}=A_{3,12}=\beta$, we get all the angles of $P_2$, $P_3$. By AVC, we further get all the angles at all the vertices of $P_2$, $P_3$. The angles at these vertices then further determine all the angles of $P_4$, $P_6$, $P_7$, $P_8$, $P_{\overline{11}}$. By AVC, we again get all the angles at all the vertices of these tiles. Then we further get all the angles of $P_5$, $P_9$, $P_{10}$. Finally we get all the angles of $P_{12}$ by AVC.

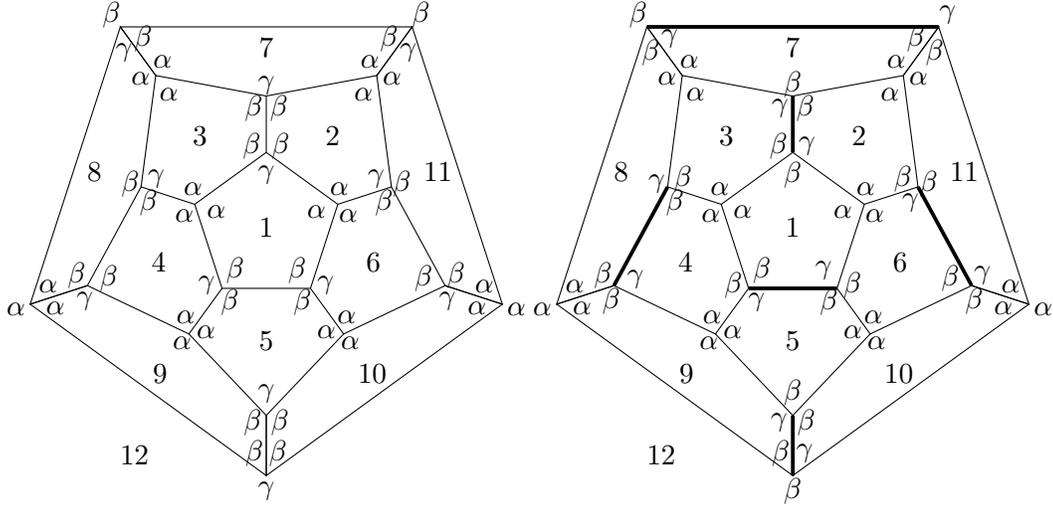
\begin{figure}[htp]
\centering
    \begin{tikzpicture}[scale=1]


	\foreach \x in {1,...,5}
    \draw 
    		(-54+72*\x:1) -- (18+72*\x:1)
    		(-54+72*\x:1) -- (-54+72*\x:1.75) -- (-18+72*\x:2.5) -- (18+72*\x:1.75)
		(-18+72*\x:2.5) -- (-18+72*\x:3.3) -- (54+72*\x:3.3) -- (54+72*\x:2.5);

    \node at (0:0) {\small $1$};
    \node at (54:1.5) {\small $2$};
    \node at (126:1.5) {\small $3$};
    \node at (198:1.5) {\small $4$};
    \node at (-90:1.5) {\small $5$};
    \node at (-18:1.5) {\small $6$};
    \node at (90:2.4) {\small $7$};
    \node at (162:2.4) {\small $8$};
    \node at (234:2.4) {\small $9$};
    \node at (-54:2.4) {\small $10$};
    \node at (18:2.4) {\small $11$};
    \node at (240:3.5){\small $12$};

    \node at (18:0.7) {\small $\alpha$};
    \node at (8:1.1) {\small $\alpha$};
    \node at (28:1.1) {\small $\alpha$};
    
    \node at (90:0.7) {\small $\gamma$};
    \node at (80:1.1) {\small $\beta$};
    \node at (100:1.1) {\small $\beta$};
    
    \node at (162:0.7) {\small $\alpha$};
    \node at (152:1.1) {\small $\alpha$};
    \node at (172:1.1) {\small $\alpha$};
    
    \node at (234:0.7) {\small $\beta$};
    \node at (224:1.1) {\small $\gamma$};
    \node at (244:1.1) {\small $\beta$};
    
    \node at (-54:0.7) {\small $\beta$};
    \node at (-44:1.1) {\small $\gamma$};
    \node at (-64:1.1) {\small $\beta$};

    \node at (18:1.9) {\small $\beta$};
    \node at (12:1.6) {\small $\beta$};
    \node at (24:1.6) {\small $\gamma$};
    
    \node at (90:1.9) {\small $\gamma$};
    \node at (84:1.6) {\small $\beta$};
    \node at (96:1.6) {\small $\beta$};
    
    \node at (162:1.9) {\small $\beta$};
    \node at (156:1.6) {\small $\gamma$};
    \node at (168:1.6) {\small $\beta$};
    
    \node at (234:1.9) {\small $\alpha$};
    \node at (228:1.6) {\small $\alpha$};
    \node at (240:1.6) {\small $\alpha$};
    
    \node at (-54:1.9) {\small $\alpha$};
    \node at (-48:1.6) {\small $\alpha$};
    \node at (-60:1.6) {\small $\alpha$};

    \node at (54:2.2) {\small $\alpha$};
    \node at (50:2.6) {\small $\alpha$};
    \node at (58:2.6) {\small $\alpha$};
    
    \node at (126:2.2) {\small $\alpha$};
    \node at (122:2.6) {\small $\alpha$};
    \node at (130:2.6) {\small $\alpha$};
    
    \node at (198:2.2) {\small $\beta$};
    \node at (194:2.6) {\small $\beta$};
    \node at (202:2.6) {\small $\gamma$};

    \node at (-90:2.2) {\small $\gamma$};
    \node at (-94:2.6) {\small $\beta$};
    \node at (-86:2.6) {\small $\beta$};
    
    \node at (-18:2.2) {\small $\beta$};
    \node at (-22:2.6) {\small $\gamma$};
    \node at (-14:2.6) {\small $\beta$};

    \node at (54:3.5) {\small $\beta$};
    \node at (51:3) {\small $\gamma$};
    \node at (57:3) {\small $\beta$};

    \node at (126:3.5) {\small $\beta$};
    \node at (123:3) {\small $\beta$};
    \node at (129:3) {\small $\gamma$};

    \node at (198:3.5) {\small $\alpha$};
    \node at (195:3) {\small $\alpha$};
    \node at (201:3) {\small $\alpha$};

    \node at (-90:3.5) {\small $\gamma$};
    \node at (-93:3) {\small $\beta$};
    \node at (-87:3) {\small $\beta$};

    \node at (-18:3.5) {\small $\alpha$};
    \node at (-15:3) {\small $\alpha$};
    \node at (-21:3) {\small $\alpha$};

    \begin{scope}[xshift=7cm]
    
    \foreach \x in {1,...,5}
    \draw 
    		(-54+72*\x:1) -- (18+72*\x:1)
    		(-54+72*\x:1) -- (-54+72*\x:1.75) -- (-18+72*\x:2.5) -- (18+72*\x:1.75)
		(-18+72*\x:2.5) -- (-18+72*\x:3.3) -- (54+72*\x:3.3) -- (54+72*\x:2.5);

    \node at (0:0) {\small $1$};
    \node at (54:1.5) {\small $2$};
    \node at (126:1.5) {\small $3$};
    \node at (198:1.5) {\small $4$};
    \node at (-90:1.5) {\small $5$};
    \node at (-18:1.5) {\small $6$};
    \node at (90:2.4) {\small $7$};
    \node at (162:2.4) {\small $8$};
    \node at (234:2.4) {\small $9$};
    \node at (-54:2.4) {\small $10$};
    \node at (18:2.4) {\small $11$};
    \node at (240:3.5){\small $12$};

    \node at (18:0.7) {\small $\alpha$};
    \node at (8:1.1) {\small $\alpha$};
    \node at (28:1.1) {\small $\alpha$};
    
    \node at (90:0.7) {\small $\beta$};
    \node at (80:1.1) {\small $\gamma$};
    \node at (100:1.1) {\small $\beta$};
    
    \node at (162:0.7) {\small $\alpha$};
    \node at (152:1.1) {\small $\alpha$};
    \node at (172:1.1) {\small $\alpha$};
    
    \node at (234:0.7) {\small $\beta$};
    \node at (224:1.1) {\small $\beta$};
    \node at (244:1.1) {\small $\gamma$};
    
    \node at (-54:0.7) {\small $\gamma$};
    \node at (-64:1.1) {\small $\beta$};
    \node at (-44:1.1) {\small $\beta$};

    \node at (18:1.9) {\small $\beta$};
    \node at (12:1.6) {\small $\gamma$};
    \node at (24:1.6) {\small $\beta$};
    
    \node at (90:1.9) {\small $\beta$};
    \node at (84:1.6) {\small $\beta$};
    \node at (96:1.6) {\small $\gamma$};
    
    \node at (162:1.9) {\small $\gamma$};
    \node at (156:1.6) {\small $\beta$};
    \node at (168:1.6) {\small $\beta$};
    
    \node at (234:1.9) {\small $\alpha$};
    \node at (228:1.6) {\small $\alpha$};
    \node at (240:1.6) {\small $\alpha$};
    
    \node at (-54:1.9) {\small $\alpha$};
    \node at (-48:1.6) {\small $\alpha$};
    \node at (-60:1.6) {\small $\alpha$};

    \node at (54:2.2) {\small $\alpha$};
    \node at (50:2.6) {\small $\alpha$};
    \node at (58:2.6) {\small $\alpha$};
    
    \node at (126:2.2) {\small $\alpha$};
    \node at (122:2.6) {\small $\alpha$};
    \node at (130:2.6) {\small $\alpha$};
    
    \node at (198:2.2) {\small $\gamma$};
    \node at (194:2.6) {\small $\beta$};
    \node at (202:2.6) {\small $\beta$};

    \node at (-90:2.2) {\small $\beta$};
    \node at (-94:2.6) {\small $\gamma$};
    \node at (-86:2.6) {\small $\beta$};
    
    \node at (-18:2.2) {\small $\beta$};
    \node at (-22:2.6) {\small $\beta$};
    \node at (-14:2.6) {\small $\gamma$};

    \node at (54:3.5) {\small $\gamma$};
    \node at (51:3) {\small $\beta$};
    \node at (57:3) {\small $\beta$};

    \node at (126:3.5) {\small $\beta$};
    \node at (123:3) {\small $\gamma$};
    \node at (129:3) {\small $\beta$};

    \node at (198:3.5) {\small $\alpha$};
    \node at (195:3) {\small $\alpha$};
    \node at (201:3) {\small $\alpha$};

    \node at (-90:3.5) {\small $\beta$};
    \node at (-93:3) {\small $\beta$};
    \node at (-87:3) {\small $\gamma$};

    \node at (-18:3.5) {\small $\alpha$};
    \node at (-15:3) {\small $\alpha$};
    \node at (-21:3) {\small $\alpha$};


	\draw[line width=0.05cm]
		(90:1) -- (90:1.75)
		(54:3.3) -- (126:3.3)
		(18:1.75) -- (-18:2.5)
		(162:1.75) -- (198:2.5)
		(-90:2.5) -- (-90:3.3)
        (-54:1) -- (234:1);
			
	\end{scope}

    \end{tikzpicture}
\caption{Tilings for the angle combination $\alpha^2\beta^2\gamma$, continued.}
\label{2a2b1cB}
\end{figure}

\medskip

{\bf Arrangement 4}

The case is the right of Figure \ref{2a2b1cB}. By AVC, we get all the angles at $V_{134}$, $V_{126}$, $V_{156}$. We also find $V_{145}=\beta^2\gamma$, so that $A_{5,14}=\beta$ or $\gamma$. Since $A_{5,16}=\beta$ and the two $\beta$ in $P_5$ are not adjacent, we get $A_{5,14}=\gamma$ and then all the angles of $P_5$. By AVC, we further get all the angles at $V_{145}$, $V_{459}$, $V_{56\overline{10}}$. We note that we may exchange the four angles $\beta,\gamma$ around $E_{15}$ without affecting the subsequent discussion.

Now we experiment with the angle $A_{2,13}$. By AVC, it must be either $\beta$ or $\gamma$. If $A_{2,13}=\gamma$, then $A_{3,12}=\beta$ and we get all the angles of $P_2$. By AVC, we get $V_{237}=\beta^2\gamma$. Since $A_{3,12}=\beta$ and the two $\beta$ in $P_3$ are not adjacent, we get $A_{3,27}=\gamma$ and then all the angles of $P_3$. By AVC, we further get all the angles at $V_{237}$, $V_{27\overline{11}}$, $V_{378}$. If $A_{2,13}=\beta$, then $A_{3,12}=\gamma$ and we may carry out the similar argument (first on $P_3$ and then on $P_2$). The result is the same picture with the four angles $\beta,\gamma$ around $E_{23}$ exchanged. 

We may successively carry out the same experiment around each of $E_{48}$, $E_{6\overline{11}}$, $E_{9\overline{10}}$, $E_{7\overline{12}}$. We find that we can make similar independent choices of $\beta,\gamma$, and each choice determines all the angles of the pair of tiles on the two sides of the edge. 
\end{proof}

\begin{proposition}\label{angle_pattern4}
The spherical tiling by $12$ angle congruent pentagons with angles $\alpha,\alpha,\beta,\gamma,\delta$, where $\alpha=\frac{2\pi}{3}$ and $\alpha,\beta,\gamma,\delta$ are distinct, is given by the right of Figure \ref{2a1b1c1dA} up to symmetry. Note that in the tiling on the right, we can exchange the four angles $\beta$ and $\gamma$ around any thick line and still get a tiling.
\end{proposition}

\begin{proof}
We are in the fourth case of Proposition \ref{angle}. The anglewise vertex combination is $\{8\alpha^3,12\beta\gamma\delta\}$, which implies the following AVC condition.

\medskip

{\bf AVC}(for $\alpha^2\beta\gamma\delta$): Any vertex is either $\alpha^3$ or $\beta\gamma\delta$. 

\medskip

Up to symmetry, there are two possible ways of arranging the angles in the pentagon. 

\medskip

{\bf Arrangement 1}

In this arrangement, the two $\alpha$ are adjacent. See the tile $P_1$ on the left of Figure \ref{2a1b1c1dA}. By AVC, we get $V_{123}=V_{134}=\alpha^3$ and $V_{126}=\beta\gamma\delta$. Thus $A_{2,16}=\beta$ or $\gamma$, and is adjacent to $\alpha$ in $P_2$. Therefore $A_{2,16}=\beta$. Together with $A_{2,13}=\alpha$, we get all the angles of $P_2$. Then $A_{2,37}=\alpha$ implies $V_{237}=\alpha^3$, and there are three $\alpha$ in $P_3$, a contradiction.

\begin{figure}[htp]
\centering
    \begin{tikzpicture}[scale=1]

	\foreach \x in {1,...,5}
    \draw (-54+72*\x:1) -- (18+72*\x:1);

    \draw
    		(90:1) -- (90:1.75) -- (54:2.5) -- (18:1.75) -- (18:1)
    		(90:1.75) -- (126:2.5) -- (162:1.75) -- (162:1)
		(234:1) -- (234:1.75) -- (198:2.5) -- (162:1.75)
		(-54:1) -- (-54:1.75) -- (-18:2.5) -- (18:1.75)
		(54:2.5) -- (54:3.3) -- (126:3.3) -- (126:2.5);

    \node at (0:0) {\small $1$};
    \node at (54:1.5) {\small $2$};
    \node at (126:1.5) {\small $3$};
    \node at (198:1.5) {\small $4$};
    \node at (-18:1.5) {\small $6$};
    \node at (90:2.4) {\small $7$};

    \node at (18:0.7) {\small $\delta$};
    \node at (28:1.1) {\small $\beta$};

    \node at (90:0.7) {\small $\alpha$};
    \node at (80:1.1) {\small $\alpha$};
    \node at (100:1.1) {\small $\alpha$};

    \node at (162:0.7) {\small $\alpha$};
    \node at (152:1.1) {\small $\alpha$};
    \node at (172:1.1) {\small $\alpha$};

    \node at (234:0.7) {\small $\beta$};
    \node at (-54:0.7) {\small $\gamma$};

    \node at (90:1.9) {\small $\alpha$};
    \node at (84:1.6) {\small $\alpha$};
    \node at (96:1.6) {\small $\alpha$};
    
    \node at (24:1.6) {\small $\gamma$};
    
    \node at (54:2.2) {\small $\delta$};
    
    \begin{scope}[xshift=7cm]

    \foreach \x in {1,...,5}
    \draw
    		(-54+72*\x:1) -- (18+72*\x:1)
    		(-54+72*\x:1) -- (-54+72*\x:1.75) -- (-18+72*\x:2.5) -- (18+72*\x:1.75)
		(-18+72*\x:2.5) -- (-18+72*\x:3.3) -- (54+72*\x:3.3) -- (54+72*\x:2.5);

    \node at (0:0) {\small $1$};
    \node at (54:1.5) {\small $2$};
    \node at (126:1.5) {\small $3$};
    \node at (198:1.5) {\small $4$};
    \node at (-90:1.5) {\small $5$};
    \node at (-18:1.5) {\small $6$};
    \node at (90:2.4) {\small $7$};
    \node at (162:2.4) {\small $8$};
    \node at (234:2.4) {\small $9$};
    \node at (-54:2.4) {\small $10$};
    \node at (18:2.4) {\small $11$};
    \node at (240:3.5){\small $12$};

    \node at (18:0.7) {\small $\alpha$};
    \node at (8:1.1) {\small $\alpha$};
    \node at (28:1.1) {\small $\alpha$};
    
    \node at (90:0.7) {\small $\delta$};
    \node at (80:1.1) {\small $\gamma$};
    \node at (100:1.1) {\small $\beta$};
    
    \node at (162:0.7) {\small $\alpha$};
    \node at (152:1.1) {\small $\alpha$};
    \node at (172:1.1) {\small $\alpha$};
    
    \node at (234:0.7) {\small $\beta$};
    \node at (224:1.1) {\small $\delta$};
    \node at (244:1.1) {\small $\gamma$};
    
    \node at (-54:0.7) {\small $\gamma$};
    \node at (-64:1.1) {\small $\beta$};
    \node at (-44:1.1) {\small $\delta$};

    \node at (18:1.9) {\small $\beta$};
    \node at (12:1.6) {\small $\gamma$};
    \node at (24:1.6) {\small $\delta$};
    
    \node at (90:1.9) {\small $\delta$};
    \node at (84:1.6) {\small $\beta$};
    \node at (96:1.6) {\small $\gamma$};
    
    \node at (162:1.9) {\small $\gamma$};
    \node at (156:1.6) {\small $\delta$};
    \node at (168:1.6) {\small $\beta$};
    
    \node at (234:1.9) {\small $\alpha$};
    \node at (228:1.6) {\small $\alpha$};
    \node at (240:1.6) {\small $\alpha$};
    
    \node at (-54:1.9) {\small $\alpha$};
    \node at (-48:1.6) {\small $\alpha$};
    \node at (-60:1.6) {\small $\alpha$};

    \node at (54:2.2) {\small $\alpha$};
    \node at (50:2.6) {\small $\alpha$};
    \node at (58:2.6) {\small $\alpha$};
    
    \node at (126:2.2) {\small $\alpha$};
    \node at (122:2.6) {\small $\alpha$};
    \node at (130:2.6) {\small $\alpha$};
    
    \node at (198:2.2) {\small $\gamma$};
    \node at (194:2.6) {\small $\beta$};
    \node at (202:2.6) {\small $\delta$};

    \node at (-90:2.2) {\small $\delta$};
    \node at (-94:2.6) {\small $\gamma$};
    \node at (-86:2.6) {\small $\beta$};
    
    \node at (-18:2.2) {\small $\beta$};
    \node at (-22:2.6) {\small $\delta$};
    \node at (-14:2.6) {\small $\gamma$};

    \node at (54:3.5) {\small $\gamma$};
    \node at (51:3) {\small $\delta$};
    \node at (57:3) {\small $\beta$};

    \node at (126:3.5) {\small $\beta$};
    \node at (123:3) {\small $\gamma$};
    \node at (129:3) {\small $\delta$};

    \node at (198:3.5) {\small $\alpha$};
    \node at (195:3) {\small $\alpha$};
    \node at (201:3) {\small $\alpha$};

    \node at (-90:3.5) {\small $\delta$};
    \node at (-93:3) {\small $\beta$};
    \node at (-87:3) {\small $\gamma$};

    \node at (-18:3.5) {\small $\alpha$};
    \node at (-15:3) {\small $\alpha$};
    \node at (-21:3) {\small $\alpha$};


	\draw[line width=0.05cm]
		(90:1) -- (90:1.75)
		(54:3.3) -- (126:3.3)
		(18:1.75) -- (-18:2.5)
		(162:1.75) -- (198:2.5)
		(-90:2.5) -- (-90:3.3)
        	(-54:1) -- (234:1);
	
		\end{scope}

    \end{tikzpicture}
\caption{Tilings for the angle combination $\alpha^2\beta\gamma\delta$.}
\label{2a1b1c1dA}
\end{figure}
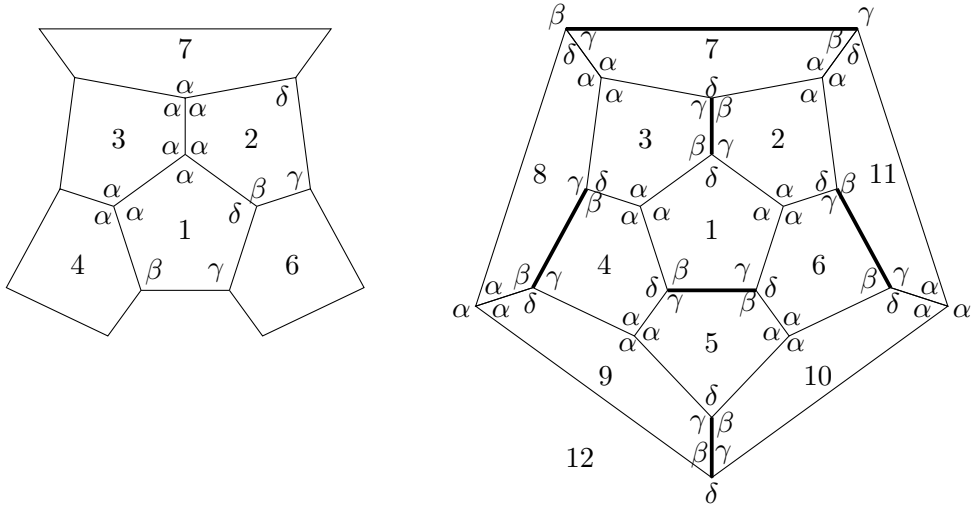

\medskip

{\bf Arrangement 2}

In this arrangement, the two $\alpha$ are not adjacent. See the tile $P_1$ on the right of Figure \ref{2a1b1c1dA}. By AVC, we get $V_{126}=V_{134}=\alpha^3$ and $V_{123}=\beta\gamma\delta$. If $A_{2,13}=\gamma$, $A_{3,12}=\beta$, then we get all the angles of $P_2$, $P_3$ and all the angles at $V_{237}$, $V_{27\overline{11}}$, $V_{378}$. If $A_{2,13}=\beta$, $A_{3,12}=\gamma$, then we get the similar result, except that the four angles $\beta,\gamma$ around $E_{23}$ are exchanged. Such exchange does not affect the subsequent discussion. 

By $A_{3,48}=\delta$, we know $A_{4,38}$, $A_{8,34}$ are $\beta,\gamma$. We may make either choice for $A_{4,38}$, $A_{8,34}$ and then get all the angles of $P_4$, $P_8$. Similar discussions can then be successively carried out around $E_{6\overline{11}}$, $E_{9\overline{10}}$, $E_{7\overline{12}}$ and determine all the angles.
\end{proof}

\begin{proposition}\label{angle_pattern5}
If the angles in a spherical tiling by $12$ angle congruent pentagons are $\alpha,\beta,\gamma,\delta,\epsilon$, with $\alpha=\frac{2\pi}{3}$ and all angles distinct, then up to symmetry, the tiling must have anglewise vertex combination $\{2\alpha^3,6\alpha\gamma\epsilon,6\beta^2\gamma,6\delta^2\epsilon\}$ and is given by one of the tilings on the left of Figures \ref{abcde2A}, \ref{abcde3B}, \ref{abcde7B}, \ref{abcde8B}.
\end{proposition}

\begin{proof}
We are in the last case of Proposition \ref{angle}. The following AVC condition applies to all three possible anglewise vertex combinations in the case.

\medskip

{\bf AVC}(for $\alpha\beta\gamma\delta\epsilon$): Any vertex is one of $\alpha^3$, $\alpha\beta\delta$, $\alpha\gamma\epsilon$, $\beta^2\gamma$, $\delta^2\epsilon$. 

\medskip

\noindent A consequence we will often use is that $\beta\epsilon$, $\gamma\delta$, $\gamma^2$, $\epsilon^2$ are forbidden at any vertex.

All three anglewise vertex combinations are symmetric with respect to the simultaneous exchange of $\beta$ with $\delta$ and $\gamma$ with $\epsilon$. Modulo such symmetry, there are eight possible arrangements of the angles in the pentagon, described in Figure \ref{abcdeangle}. We discuss each case by assuming the angles of $P_1$ in Figure \ref{tile_pattern_pic} are arranged as in Figure \ref{abcdeangle}. Since the vertex $\alpha\gamma\epsilon$ appears in all three anglewise vertex combinations, we will also assume $V_{123}=\alpha\gamma\epsilon$. This leads to the possibilities $A_{2,13}=\gamma$, $A_{3,12}=\epsilon$ and $A_{2,13}=\epsilon$, $A_{3,12}=\gamma$.

\begin{figure}[htp]
\centering
    \begin{tikzpicture}[scale=1]

	\foreach \x in {1,...,5}
	\foreach \y in {0,...,3}
	\foreach \z in {0,-1}
    \draw[shift={(3*\y cm, 2.5*\z cm)}]
    		(-54+72*\x:1) -- (18+72*\x:1);


	\node at (18:0.7) {\small $\epsilon$};
	\node at (90:0.7) {\small $\alpha$};
	\node at (162:0.7) {\small $\beta$};
	\node at (234:0.7) {\small $\gamma$};
	\node at (-54:0.7) {\small $\delta$};

	\begin{scope}[xshift=3cm]
	\node at (18:0.7) {\small $\gamma$};
	\node at (90:0.7) {\small $\alpha$};
	\node at (162:0.7) {\small $\beta$};
	\node at (234:0.7) {\small $\epsilon$};
	\node at (-54:0.7) {\small $\delta$};
	\end{scope}	
		
	
	\begin{scope}[xshift=6cm]
	\node at (18:0.7) {\small $\gamma$};
	\node at (90:0.7) {\small $\alpha$};
	\node at (162:0.7) {\small $\beta$};
	\node at (234:0.7) {\small $\delta$};
	\node at (-54:0.7) {\small $\epsilon$};
	\end{scope}
		
	\begin{scope}[xshift=9cm]
	\node at (18:0.7) {\small $\epsilon$};
	\node at (90:0.7) {\small $\alpha$};
	\node at (162:0.7) {\small $\beta$};
	\node at (234:0.7) {\small $\delta$};
	\node at (-54:0.7) {\small $\gamma$};
	\end{scope}

	\begin{scope}[yshift=-2.5cm]

	\node at (18:0.7) {\small $\delta$};
	\node at (90:0.7) {\small $\alpha$};
	\node at (162:0.7) {\small $\beta$};
	\node at (234:0.7) {\small $\gamma$};
	\node at (-54:0.7) {\small $\epsilon$};
	
	\begin{scope}[xshift=3cm]
	\node at (18:0.7) {\small $\delta$};
	\node at (90:0.7) {\small $\alpha$};
	\node at (162:0.7) {\small $\beta$};
	\node at (234:0.7) {\small $\epsilon$};
	\node at (-54:0.7) {\small $\gamma$};
	\end{scope}	
	
	\begin{scope}[xshift=6cm]
	\node at (18:0.7) {\small $\epsilon$};
	\node at (90:0.7) {\small $\alpha$};
	\node at (162:0.7) {\small $\gamma$};
	\node at (234:0.7) {\small $\beta$};
	\node at (-54:0.7) {\small $\delta$};
	\end{scope}
	
    \begin{scope}[xshift=9cm]
	\node at (18:0.7) {\small $\gamma$};
	\node at (90:0.7) {\small $\alpha$};
	\node at (162:0.7) {\small $\epsilon$};
	\node at (234:0.7) {\small $\beta$};
	\node at (-54:0.7) {\small $\delta$};
	\end{scope}
	
	\end{scope}

    \end{tikzpicture}
\caption{Eight possible arrangements for the angle combination $\alpha\beta\gamma\delta\epsilon$.}
\label{abcdeangle}
\end{figure}
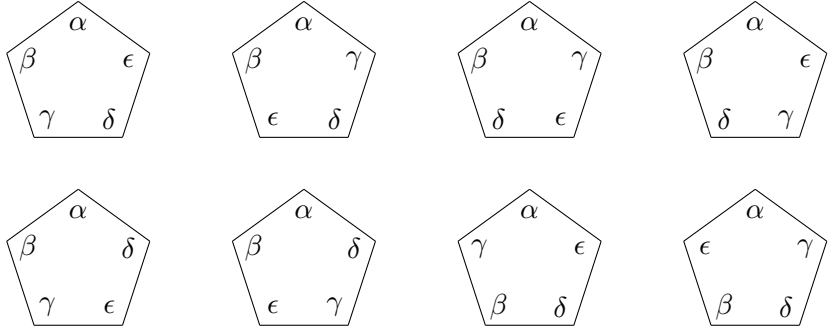

Since all angles are distinct, there is no ambiguity for us to call the orientation of angles given by Figure \ref{abcdeangle} {\em counterclockwise}, and call the other orientation {\em clockwise}.

\medskip

{\bf Arrangement 1, Case 1}: $A_{2,13}=\gamma$, $A_{3,12}=\epsilon$.

The case is the left of Figure \ref{abcde1}. Since $A_{2,16}$ is adjacent to $A_{2,13}=\gamma$ in $P_2$, we get $A_{2,16}=\beta$ or $\delta$. By AVC, we get $A_{2,16}=\delta$ and all the angles of $P_2$. 

By $A_{1,26}=\epsilon$, $A_{2,16}=\delta$ and AVC, we get $A_{6,12}=\delta$. Since $A_{6,15}$ is adjacent to $\delta$ in $P_6$, we get $A_{6,15}=\gamma$ or $\epsilon$. By AVC, we get $A_{6,15}=\epsilon$ and all the angles of $P_6$. 

By $A_{1,56}=\delta$, $A_{6,15}=\epsilon$ and AVC, we get $A_{5,16}=\delta$. Since $A_{5,14}$ is adjacent to $\delta$ in $P_5$, we get $A_{5,14}=\gamma$ or $\epsilon$. By AVC, we get $A_{5,14}=\epsilon$ and all the angles of $P_5$.

By $A_{1,45}=\gamma$, $A_{5,14}=\epsilon$ and AVC, we get $A_{4,15}=\alpha$. Since $A_{4,13}$ is adjacent to $\alpha$ in $P_4$, we get $A_{4,13}=\beta$ or $\epsilon$. By AVC, we get $A_{4,13}=\beta$ and all the angles of $P_4$.

By AVC, we further get $A_{3,14}=\gamma$, which is adjacent to $A_{3,12}=\epsilon$ in $P_3$. This is a contradiction.

We remark that the conclusion of the case is the following: If the angles in a tile (of first arrangement) are counterclockwise oriented, and the vertex at $\alpha$ is $\alpha\gamma\epsilon$, then $\alpha,\gamma,\epsilon$ must be clockwise oriented at the vertex. By symmetry, we also know that if the angles in a tile are clockwise oriented, then $\alpha,\gamma,\epsilon$ must be counterclockwise oriented at the vertex.

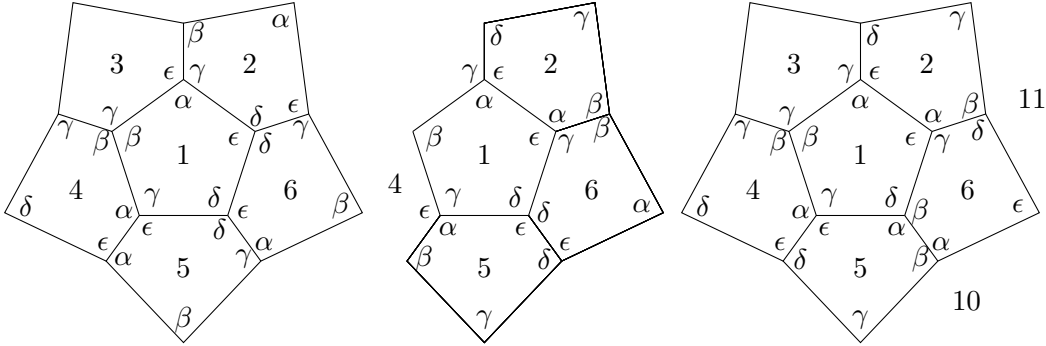
\begin{figure}[htp]
\centering
    \begin{tikzpicture}[scale=1]


	\foreach \x in {1,...,5}
    \draw 
    		(-54+72*\x:1) -- (18+72*\x:1)
    		(-54+72*\x:1) -- (-54+72*\x:1.75) -- (-18+72*\x:2.5) -- (18+72*\x:1.75);

    \node at (0:0) {\small $1$};
    \node at (54:1.5) {\small $2$};
    \node at (126:1.5) {\small $3$};
    \node at (198:1.5) {\small $4$};
    \node at (-90:1.5) {\small $5$};
    \node at (-18:1.5) {\small $6$};

    \node at (18:0.7) {\small $\epsilon$};
    \node at (8:1.1) {\small $\delta$};
    \node at (28:1.1) {\small $\delta$};
    
    \node at (90:0.7) {\small $\alpha$};
    \node at (80:1.1) {\small $\gamma$};
    \node at (100:1.1) {\small $\epsilon$};
    
    \node at (162:0.7) {\small $\beta$};
    \node at (152:1.1) {\small $\gamma$};
    \node at (172:1.1) {\small $\beta$};
    
    \node at (234:0.7) {\small $\gamma$};
    \node at (224:1.1) {\small $\alpha$};
    \node at (244:1.1) {\small $\epsilon$};
    
    \node at (-54:0.7) {\small $\delta$};
    \node at (-44:1.1) {\small $\epsilon$};
    \node at (-64:1.1) {\small $\delta$};

    \node at (12:1.6) {\small $\gamma$};
    \node at (24:1.6) {\small $\epsilon$};
    
    \node at (84:1.6) {\small $\beta$};
    
    \node at (168:1.6) {\small $\gamma$};
    
    \node at (228:1.6) {\small $\epsilon$};
    \node at (240:1.6) {\small $\alpha$};
    
    \node at (-48:1.6) {\small $\alpha$};
    \node at (-60:1.6) {\small $\gamma$};

    \node at (54:2.2) {\small $\alpha$};
    
    
    \node at (198:2.2) {\small $\delta$};

    \node at (-90:2.2) {\small $\beta$};
    
    \node at (-18:2.2) {\small $\beta$};

\begin{scope}[xshift=4cm]
	

	\foreach \x in {1,...,5}
    \draw 
    		(-54+72*\x:1) -- (18+72*\x:1)
		(-126:1) -- (-126:1.75) -- (-90:2.5) -- (-54:1.75)
		(-54:1) -- (-54:1.75) -- (-18:2.5) -- (18:1.75)
		(18:1) -- (18:1.75) -- (54:2.5) -- (90:1.75) -- (90:1);

    \node at (0:0) {\small $1$};
    \node at (54:1.5) {\small $2$};
    \node at (198:1.25) {\small $4$};
    \node at (-90:1.5) {\small $5$};
    \node at (-18:1.5) {\small $6$};
        
    \node at (18:0.7) {\small $\epsilon$};
    \node at (8:1.1) {\small $\gamma$};
    \node at (28:1.1) {\small $\alpha$};
    
    \node at (90:0.7) {\small $\alpha$};
    \node at (80:1.1) {\small $\epsilon$};
    \node at (100:1.1) {\small $\gamma$};
    
    \node at (162:0.7) {\small $\beta$};
    
    \node at (234:0.7) {\small $\gamma$};
    \node at (224:1.1) {\small $\epsilon$};
    \node at (244:1.1) {\small $\alpha$};
    
    \node at (-54:0.7) {\small $\delta$};
    \node at (-44:1.1) {\small $\delta$};
    \node at (-64:1.1) {\small $\epsilon$};

    \node at (12:1.6) {\small $\beta$};
    \node at (24:1.6) {\small $\beta$};
    
    \node at (84:1.6) {\small $\delta$};
    
    
    \node at (240:1.6) {\small $\beta$};
    
    \node at (-48:1.6) {\small $\epsilon$};
    \node at (-60:1.6) {\small $\delta$};

    \node at (54:2.2) {\small $\gamma$};
    
    

    \node at (-90:2.2) {\small $\gamma$};
    
    \node at (-18:2.2) {\small $\alpha$};

    \end{scope}
    
	\begin{scope}[xshift=9cm]
	

	\foreach \x in {1,...,5}
    \draw 
    		(-54+72*\x:1) -- (18+72*\x:1)
    		(-54+72*\x:1) -- (-54+72*\x:1.75) -- (-18+72*\x:2.5) -- (18+72*\x:1.75);

    \node at (0:0) {\small $1$};
    \node at (54:1.5) {\small $2$};
    \node at (126:1.5) {\small $3$};
    \node at (198:1.5) {\small $4$};
    \node at (-90:1.5) {\small $5$};
    \node at (-18:1.5) {\small $6$};
    \node at (-54:2.4) {\small $10$};
    \node at (18:2.4) {\small $11$};
    
    \node at (18:0.7) {\small $\epsilon$};
    \node at (8:1.1) {\small $\gamma$};
    \node at (28:1.1) {\small $\alpha$};
    
    \node at (90:0.7) {\small $\alpha$};
    \node at (80:1.1) {\small $\epsilon$};
    \node at (100:1.1) {\small $\gamma$};
    
    \node at (162:0.7) {\small $\beta$};
    \node at (152:1.1) {\small $\gamma$};
    \node at (172:1.1) {\small $\beta$};
    
    \node at (234:0.7) {\small $\gamma$};
    \node at (224:1.1) {\small $\alpha$};
    \node at (244:1.1) {\small $\epsilon$};
    
    \node at (-54:0.7) {\small $\delta$};
    \node at (-44:1.1) {\small $\beta$};
    \node at (-64:1.1) {\small $\alpha$};

    \node at (12:1.6) {\small $\delta$};
    \node at (24:1.6) {\small $\beta$};
    
    \node at (84:1.6) {\small $\delta$};
    
    \node at (168:1.6) {\small $\gamma$};
    
    \node at (228:1.6) {\small $\epsilon$};
    \node at (240:1.6) {\small $\delta$};
    
    \node at (-48:1.6) {\small $\alpha$};
    \node at (-60:1.6) {\small $\beta$};

    \node at (54:2.2) {\small $\gamma$};
    
    
    \node at (198:2.2) {\small $\delta$};

    \node at (-90:2.2) {\small $\gamma$};
    
    \node at (-18:2.2) {\small $\epsilon$};

    \end{scope}

    \end{tikzpicture}
\caption{Tilings for the first arrangement.}
\label{abcde1}
\end{figure}

\medskip

{\bf Arrangement 1, Case 2}: $A_{2,13}=\epsilon$, $A_{3,12}=\gamma$.

The case is the middle and the right of Figure \ref{abcde1}. Since $A_{2,16}$ is adjacent to $\epsilon$ in $P_2$, we get $A_{2,16}=\alpha$ or $\delta$. 

If $A_{2,16}=\delta$, then we get all the angles of $P_2$ and $A_{2,6\overline{11}}=\gamma$ in particular. By AVC, we also get $A_{6,12}=\delta$. Being adjacent to $A_{6,12}=\delta$ in $P_6$, one of $A_{6,15}$, $A_{6,2\overline{11}}$ is $\gamma$. This contradicts to AVC at $V_{156}$ or $V_{26\overline{11}}$.

Therefore we must have $A_{2,16}=\alpha$. Then we get all the angles of $P_2$ and $A_{6,12}=\gamma$. Now consider the two possible orientations of $P_6$. 

In the middle of Figure \ref{abcde1}, $P_6$ is counterclockwise oriented. By AVC, we get $A_{5,16}=\epsilon$, and the angle $A_{5,14}$ adjacent to $\epsilon$ in $P_5$ must be $\alpha$. This determines all the angles of $P_5$. Now $P_5$ is counterclockwise oriented, and $\alpha,\gamma,\epsilon$ is also counterclockwise oriented at $V_{145}$. This contradicts to the conclusion of Case 1.

On the right of Figure \ref{abcde1}, $P_6$ is clockwise oriented. By AVC, we get $A_{5,16}=\alpha$. Then the angle $A_{5,6\overline{10}}$ adjacent to $\alpha$ in $P_5$ is $\beta$ or $\epsilon$. If $A_{5,6\overline{10}}=\epsilon$, then $P_6$ is clockwise oriented, and $\alpha,\gamma,\epsilon$ are also clockwise oriented at $V_{56\overline{10}}$. This contradicts to the conclusion of Case 1. Therefore we must have $A_{5,6\overline{10}}=\beta$. Then we get all the angles of $P_5$ and by AVC, we further get $A_{4,15}=\alpha$. The clockwise orientation of $\alpha,\gamma,\epsilon$ at $V_{145}$ and the conclusion of Case 1 imply that $P_4$ must be counterclockwise oriented. This determines all the angles of $P_4$. Then by AVC, we get $A_{3,14}=\gamma$, and there are two $\gamma$ in $P_3$, a contradiction.

\medskip

{\bf Arrangement 2, Case 1}: $A_{2,13}=\gamma$, $A_{3,12}=\epsilon$.

Consider the two possible orientations of $P_3$. The counterclockwise case is the left of Figure \ref{abcde2A}. By AVC, we may successively determine all the angles of $P_2$, $P_7$, $P_{11}$, $P_6$, $P_5$, $P_{10}$, $P_{12}$, $P_8$, $P_4$, $P_9$. There is no contradiction, and we get an angle congruent tiling with anglewise vertex combination $\{2\alpha^3,6\alpha\gamma\epsilon,6\beta^2\gamma,6\delta^2\epsilon\}$.

On the right of Figure \ref{abcde2A}, $P_3$ is clockwise oriented. By AVC, we get $A_{4,13}=\alpha$, then get all the angles of $P_4$, and further get $A_{5,14}=\alpha$. Now one of $A_{5,16}$, $A_{5,49}$ adjacent to $\alpha$ in $P_5$ is $\gamma$. This contradicts to AVC at $V_{156}$ or $V_{459}$.

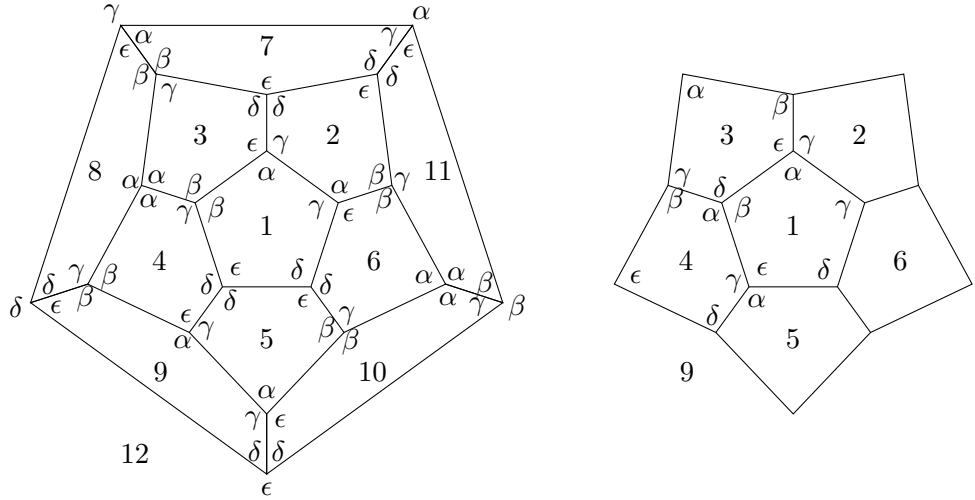
\begin{figure}[htp]
\centering
    \begin{tikzpicture}[scale=1]

\foreach \x in {1,...,5}
    \draw
    		(-54+72*\x:1) -- (18+72*\x:1)
    		(-54+72*\x:1) -- (-54+72*\x:1.75) -- (-18+72*\x:2.5) -- (18+72*\x:1.75)
		(-18+72*\x:2.5) -- (-18+72*\x:3.3) -- (54+72*\x:3.3) -- (54+72*\x:2.5);

    \node at (0:0) {\small $1$};
    \node at (54:1.5) {\small $2$};
    \node at (126:1.5) {\small $3$};
    \node at (198:1.5) {\small $4$};
    \node at (-90:1.5) {\small $5$};
    \node at (-18:1.5) {\small $6$};
    \node at (90:2.4) {\small $7$};
    \node at (162:2.4) {\small $8$};
    \node at (234:2.4) {\small $9$};
    \node at (-54:2.4) {\small $10$};
    \node at (18:2.4) {\small $11$};
    \node at (240:3.5){\small $12$};

    \node at (18:0.7) {\small $\gamma$};
    \node at (8:1.1) {\small $\epsilon$};
    \node at (28:1.1) {\small $\alpha$};

    \node at (90:0.7) {\small $\alpha$};
    \node at (80:1.1) {\small $\gamma$};
    \node at (100:1.1) {\small $\epsilon$};

    \node at (162:0.7) {\small $\beta$};
    \node at (152:1.1) {\small $\beta$};
    \node at (172:1.1) {\small $\gamma$};

    \node at (234:0.7) {\small $\epsilon$};
    \node at (224:1.1) {\small $\delta$};
    \node at (244:1.1) {\small $\delta$};

    \node at (-54:0.7) {\small $\delta$};
    \node at (-44:1.1) {\small $\delta$};
    \node at (-64:1.1) {\small $\epsilon$};

    \node at (18:1.9) {\small $\gamma$};
    \node at (12:1.6) {\small $\beta$};
    \node at (24:1.6) {\small $\beta$};

    \node at (90:1.9) {\small $\epsilon$};
    \node at (84:1.6) {\small $\delta$};
    \node at (96:1.6) {\small $\delta$};

    \node at (162:1.9) {\small $\alpha$};
    \node at (156:1.6) {\small $\alpha$};
    \node at (168:1.6) {\small $\alpha$};

    \node at (234:1.9) {\small $\alpha$};
    \node at (228:1.6) {\small $\epsilon$};
    \node at (240:1.6) {\small $\gamma$};

    \node at (-54:1.9) {\small $\beta$};
    \node at (-48:1.6) {\small $\gamma$};
    \node at (-60:1.6) {\small $\beta$};

    \node at (54:2.2) {\small $\epsilon$};
    \node at (50:2.6) {\small $\delta$};
    \node at (58:2.6) {\small $\delta$};

    \node at (126:2.2) {\small $\gamma$};
    \node at (122:2.6) {\small $\beta$};
    \node at (130:2.6) {\small $\beta$};

    \node at (198:2.2) {\small $\beta$};
    \node at (194:2.6) {\small $\gamma$};
    \node at (202:2.6) {\small $\beta$};

    \node at (-90:2.2) {\small $\alpha$};
    \node at (-94:2.6) {\small $\gamma$};
    \node at (-86:2.6) {\small $\epsilon$};

    \node at (-18:2.2) {\small $\alpha$};
    \node at (-22:2.6) {\small $\alpha$};
    \node at (-14:2.6) {\small $\alpha$};

    \node at (54:3.5) {\small $\alpha$};
    \node at (51:3) {\small $\epsilon$};
    \node at (57:3) {\small $\gamma$};

    \node at (126:3.5) {\small $\gamma$};
    \node at (123:3) {\small $\alpha$};
    \node at (129:3) {\small $\epsilon$};

    \node at (198:3.5) {\small $\delta$};
    \node at (195:3) {\small $\delta$};
    \node at (201:3) {\small $\epsilon$};

    \node at (-90:3.5) {\small $\epsilon$};
    \node at (-93:3) {\small $\delta$};
    \node at (-87:3) {\small $\delta$};

    \node at (-18:3.5) {\small $\beta$};
    \node at (-15:3) {\small $\beta$};
    \node at (-21:3) {\small $\gamma$};

	\begin{scope}[xshift=7cm]
	
	\foreach \x in {1,...,5}
    \draw
    		(-54+72*\x:1) -- (18+72*\x:1)
    		(-54+72*\x:1) -- (-54+72*\x:1.75) -- (-18+72*\x:2.5) -- (18+72*\x:1.75);

    \node at (0:0) {\small $1$};
    \node at (54:1.5) {\small $2$};
    \node at (126:1.5) {\small $3$};
    \node at (198:1.5) {\small $4$};
    \node at (-90:1.5) {\small $5$};
    \node at (-18:1.5) {\small $6$};
    \node at (234:2.4) {\small $9$};

    \node at (18:0.7) {\small $\gamma$};

    \node at (90:0.7) {\small $\alpha$};
    \node at (80:1.1) {\small $\gamma$};
    \node at (100:1.1) {\small $\epsilon$};

    \node at (162:0.7) {\small $\beta$};
    \node at (152:1.1) {\small $\delta$};
    \node at (172:1.1) {\small $\alpha$};

    \node at (234:0.7) {\small $\epsilon$};
    \node at (224:1.1) {\small $\gamma$};
    \node at (244:1.1) {\small $\alpha$};

    \node at (-54:0.7) {\small $\delta$};


    \node at (96:1.6) {\small $\beta$};

    \node at (156:1.6) {\small $\gamma$};
    \node at (168:1.6) {\small $\beta$};

    \node at (228:1.6) {\small $\delta$};



    \node at (126:2.2) {\small $\alpha$};

    \node at (198:2.2) {\small $\epsilon$};



    \end{scope}

    \end{tikzpicture}
\caption{Tilings for the second arrangement.}
\label{abcde2A}
\end{figure}

\medskip

{\bf Arrangement 2, Case 2}: $A_{2,13}=\epsilon$, $A_{3,12}=\gamma$.

By AVC, we get all the angles of $P_2$ and $A_{6,12}=\beta$. If $P_6$ is counterclockwise oriented, as on the left of Figure \ref{abcde2B}, then by AVC, we get $A_{11,26}=\alpha$. Now one of $A_{11,27}$, $A_{11,6\overline{10}}$ adjacent to $\alpha$ in $P_{11}$ is $\gamma$. This contradicts to AVC at $V_{27\overline{11}}$ or $V_{6\overline{10}\overline{11}}$. If $P_6$ is clockwise oriented, as on the right of Figure \ref{abcde2B}, then by AVC, we get all the angles of $P_5,P_4$ and further get $A_{3,14}=\alpha$. Together with $A_{3,12}=\gamma$, we get $A_{3,84}=\beta$. This contradicts to AVC at $V_{348}$.

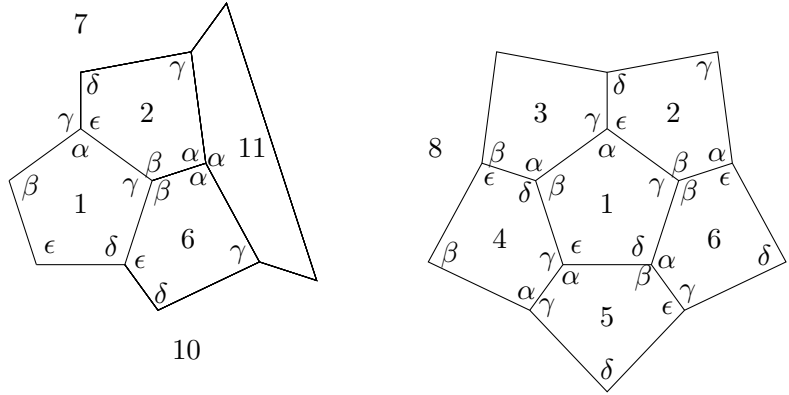
\begin{figure}[htp]
\centering
    \begin{tikzpicture}[scale=1]

\foreach \x in {1,...,5}
    \draw
    		(-54+72*\x:1) -- (18+72*\x:1)
		(-54:1) -- (-54:1.75) -- (-18:2.5) -- (18:1.75)
		(18:1) -- (18:1.75) -- (54:2.5) -- (90:1.75) -- (90:1)
		(-18:2.5) -- (-18:3.3) -- (54:3.3) -- (54:2.5);

    \node at (0:0) {\small $1$};
    \node at (54:1.5) {\small $2$};
    \node at (-18:1.5) {\small $6$};
    \node at (90:2.4) {\small $7$};
    \node at (-54:2.4) {\small $10$};
    \node at (18:2.4) {\small $11$};
    
    \node at (18:0.7) {\small $\gamma$};
    \node at (8:1.1) {\small $\beta$};
    \node at (28:1.1) {\small $\beta$};

    \node at (90:0.7) {\small $\alpha$};
    \node at (80:1.1) {\small $\epsilon$};
    \node at (100:1.1) {\small $\gamma$};

    \node at (162:0.7) {\small $\beta$};

    \node at (234:0.7) {\small $\epsilon$};

    \node at (-54:0.7) {\small $\delta$};
    \node at (-44:1.1) {\small $\epsilon$};

    \node at (18:1.9) {\small $\alpha$};
    \node at (12:1.6) {\small $\alpha$};
    \node at (24:1.6) {\small $\alpha$};

    \node at (84:1.6) {\small $\delta$};



    \node at (-48:1.6) {\small $\delta$};

    \node at (54:2.2) {\small $\gamma$};




    \node at (-18:2.2) {\small $\gamma$};

	\begin{scope}[xshift=7cm]

	\foreach \x in {1,...,5}
    \draw
    		(-54+72*\x:1) -- (18+72*\x:1)
    		(-54+72*\x:1) -- (-54+72*\x:1.75) -- (-18+72*\x:2.5) -- (18+72*\x:1.75);

    \node at (0:0) {\small $1$};
    \node at (54:1.5) {\small $2$};
    \node at (126:1.5) {\small $3$};
    \node at (198:1.5) {\small $4$};
    \node at (-90:1.5) {\small $5$};
    \node at (-18:1.5) {\small $6$};
    \node at (162:2.4) {\small $8$};

    \node at (18:0.7) {\small $\gamma$};
    \node at (8:1.1) {\small $\beta$};
    \node at (28:1.1) {\small $\beta$};

    \node at (90:0.7) {\small $\alpha$};
    \node at (80:1.1) {\small $\epsilon$};
    \node at (100:1.1) {\small $\gamma$};

    \node at (162:0.7) {\small $\beta$};
    \node at (152:1.1) {\small $\alpha$};
    \node at (172:1.1) {\small $\delta$};

    \node at (234:0.7) {\small $\epsilon$};
    \node at (224:1.1) {\small $\gamma$};
    \node at (244:1.1) {\small $\alpha$};

    \node at (-54:0.7) {\small $\delta$};
    \node at (-44:1.1) {\small $\alpha$};
    \node at (-64:1.1) {\small $\beta$};

    \node at (12:1.6) {\small $\epsilon$};
    \node at (24:1.6) {\small $\alpha$};

    \node at (84:1.6) {\small $\delta$};

    \node at (156:1.6) {\small $\beta$};
    \node at (168:1.6) {\small $\epsilon$};

    \node at (228:1.6) {\small $\alpha$};
    \node at (240:1.6) {\small $\gamma$};

    \node at (-48:1.6) {\small $\gamma$};
    \node at (-60:1.6) {\small $\epsilon$};

    \node at (54:2.2) {\small $\gamma$};


    \node at (198:2.2) {\small $\beta$};

    \node at (-90:2.2) {\small $\delta$};

    \node at (-18:2.2) {\small $\delta$};
    
    \end{scope}

    \end{tikzpicture}
\caption{Tilings for the second arrangement, continued.}
\label{abcde2B}
\end{figure}

\medskip

{\bf Arrangement 3, Case 1}: $A_{2,13}=\gamma$, $A_{3,12}=\epsilon$. 

If $P_2$ is clockwise oriented, as on the upper right of Figure \ref{abcde3B}, then by AVC, we may successively determine all the angles of $P_6$, $P_5$, $P_4$ and get $A_{3,14}=\beta$. Thus we find $\beta$ and $\epsilon$ adjacent in $P_3$, a contradiction. Similar augment shows that $P_3$ cannot be counterclockwise oriented. Thus we get all the angles of $P_2$, $P_3$ as described on the left of Figure \ref{abcde3B}. Then by AVC, we may successively determine all the angles of $P_7$, $P_{11}$, $P_6$, $P_{10}$, $P_5$, $P_4$, $P_8$, $P_9$, $P_{12}$. There is no contradiction, and we get an angle congruent tiling with anglewise vertex combination $\{2\alpha^3,6\alpha\gamma\epsilon,6\beta^2\gamma,6\delta^2\epsilon\}$.

\medskip

{\bf Arrangement 3, Case 2}: $A_{2,13}=\epsilon$, $A_{3,12}=\gamma$. 

The case is the lower right of Figure \ref{abcde3B}. By AVC, we get all the angles of $P_3$, $P_4$ and further get $A_{5,14}=\delta$. Then the angle $A_{5,16}$ adjacent to $\delta$ in $P_5$ is either $\beta$ or $\epsilon$. Either way contradicts to AVC at $V_{156}$. 

\begin{figure}[htp]
\centering
    \begin{tikzpicture}[scale=1]

    \foreach \x in {1,...,5}
    \draw
    	(-54+72*\x:1) -- (18+72*\x:1)
    	(-54+72*\x:1) -- (-54+72*\x:1.75) -- (-18+72*\x:2.5) -- (18+72*\x:1.75)
		(-18+72*\x:2.5) -- (-18+72*\x:3.3) -- (54+72*\x:3.3) -- (54+72*\x:2.5);

    \node at (0:0) {\small $1$};
    \node at (54:1.5) {\small $2$};
    \node at (126:1.5) {\small $3$};
    \node at (198:1.5) {\small $4$};
    \node at (-90:1.5) {\small $5$};
    \node at (-18:1.5) {\small $6$};
    \node at (90:2.4) {\small $7$};
    \node at (162:2.4) {\small $8$};
    \node at (234:2.4) {\small $9$};
    \node at (-54:2.4) {\small $10$};
    \node at (18:2.4) {\small $11$};
    \node at (240:3.5){\small $12$};

    \node at (18:0.7) {\small $\gamma$};
    \node at (8:1.1) {\small $\epsilon$};
    \node at (28:1.1) {\small $\alpha$};

    \node at (90:0.7) {\small $\alpha$};
    \node at (80:1.1) {\small $\gamma$};
    \node at (100:1.1) {\small $\epsilon$};

    \node at (162:0.7) {\small $\beta$};
    \node at (152:1.1) {\small $\gamma$};
    \node at (172:1.1) {\small $\beta$};

    \node at (234:0.7) {\small $\delta$};
    \node at (224:1.1) {\small $\delta$};
    \node at (244:1.1) {\small $\epsilon$};

    \node at (-54:0.7) {\small $\epsilon$};
    \node at (-44:1.1) {\small $\delta$};
    \node at (-64:1.1) {\small $\delta$};

    \node at (18:1.9) {\small $\beta$};
    \node at (12:1.6) {\small $\gamma$};
    \node at (24:1.6) {\small $\beta$};

    \node at (90:1.9) {\small $\delta$};
    \node at (84:1.6) {\small $\epsilon$};
    \node at (96:1.6) {\small $\delta$};

    \node at (162:1.9) {\small $\alpha$};
    \node at (156:1.6) {\small $\alpha$};
    \node at (168:1.6) {\small $\alpha$};

    \node at (234:1.9) {\small $\alpha$};
    \node at (228:1.6) {\small $\epsilon$};
    \node at (240:1.6) {\small $\gamma$};

    \node at (-54:1.9) {\small $\gamma$};
    \node at (-48:1.6) {\small $\beta$};
    \node at (-60:1.6) {\small $\beta$};

    \node at (54:2.2) {\small $\delta$};
    \node at (50:2.6) {\small $\delta$};
    \node at (58:2.6) {\small $\epsilon$};

    \node at (126:2.2) {\small $\beta$};
    \node at (122:2.6) {\small $\beta$};
    \node at (130:2.6) {\small $\gamma$};

    \node at (198:2.2) {\small $\gamma$};
    \node at (194:2.6) {\small $\beta$};
    \node at (202:2.6) {\small $\beta$};

    \node at (-90:2.2) {\small $\alpha$};
    \node at (-94:2.6) {\small $\gamma$};
    \node at (-86:2.6) {\small $\epsilon$};

    \node at (-18:2.2) {\small $\alpha$};
    \node at (-22:2.6) {\small $\alpha$};
    \node at (-14:2.6) {\small $\alpha$};

    \node at (54:3.5) {\small $\alpha$};
    \node at (51:3) {\small $\epsilon$};
    \node at (57:3) {\small $\gamma$};

    \node at (126:3.5) {\small $\gamma$};
    \node at (123:3) {\small $\alpha$};
    \node at (129:3) {\small $\epsilon$};

    \node at (198:3.5) {\small $\epsilon$};
    \node at (195:3) {\small $\delta$};
    \node at (201:3) {\small $\delta$};

    \node at (-90:3.5) {\small $\delta$};
    \node at (-93:3) {\small $\epsilon$};
    \node at (-87:3) {\small $\delta$};

    \node at (-18:3.5) {\small $\beta$};
    \node at (-15:3) {\small $\gamma$};
    \node at (-21:3) {\small $\beta$};

\begin{scope}[shift={(7cm,1.5cm)}]
    
	\foreach \x in {1,...,5}
    \draw
    		(-54+72*\x:1) -- (18+72*\x:1)
    		(-54+72*\x:1) -- (-54+72*\x:1.75) -- (-18+72*\x:2.5) -- (18+72*\x:1.75);

    \node at (0:0) {\small $1$};
    \node at (54:1.5) {\small $2$};
    \node at (126:1.5) {\small $3$};
    \node at (198:1.5) {\small $4$};
    \node at (-90:1.5) {\small $5$};
    \node at (-18:1.5) {\small $6$};

    \node at (18:0.7) {\small $\gamma$};
    \node at (8:1.1) {\small $\alpha$};
    \node at (28:1.1) {\small $\epsilon$};

    \node at (90:0.7) {\small $\alpha$};
    \node at (80:1.1) {\small $\gamma$};
    \node at (100:1.1) {\small $\epsilon$};

    \node at (162:0.7) {\small $\beta$};
    \node at (152:1.1) {\small $\beta$};
    \node at (172:1.1) {\small $\gamma$};

    \node at (234:0.7) {\small $\delta$};
    \node at (224:1.1) {\small $\alpha$};
    \node at (244:1.1) {\small $\beta$};

    \node at (-54:0.7) {\small $\epsilon$};
    \node at (-44:1.1) {\small $\gamma$};
    \node at (-64:1.1) {\small $\alpha$};

    \node at (12:1.6) {\small $\beta$};
    \node at (24:1.6) {\small $\delta$};

    \node at (84:1.6) {\small $\alpha$};

    \node at (168:1.6) {\small $\epsilon$};

    \node at (228:1.6) {\small $\beta$};
    \node at (240:1.6) {\small $\delta$};

    \node at (-48:1.6) {\small $\epsilon$};
    \node at (-60:1.6) {\small $\gamma$};

    \node at (54:2.2) {\small $\beta$};


    \node at (198:2.2) {\small $\delta$};

    \node at (-90:2.2) {\small $\epsilon$};

    \node at (-18:2.2) {\small $\delta$};
     
    \end{scope}    
    
    \begin{scope}[shift={(7cm,-3cm)}]

	\foreach \x in {1,...,5}
    \draw
    	(-54+72*\x:1) -- (18+72*\x:1)
    	(90:1) -- (90:1.75) -- (126:2.5) -- (162:1.75)
		(162:1) -- (162:1.75) -- (198:2.5) -- (234:1.75) -- (234:1)
		(306:1.75) -- (306:1);

    \node at (0:0) {\small $1$};
    \node at (126:1.5) {\small $3$};
    \node at (198:1.5) {\small $4$};
    \node at (-90:1.5) {\small $5$};
    \node at (-18:1.5) {\small $6$};
    
    \node at (18:0.7) {\small $\gamma$};

    \node at (90:0.7) {\small $\alpha$};
    \node at (80:1.1) {\small $\epsilon$};
    \node at (100:1.1) {\small $\gamma$};

    \node at (162:0.7) {\small $\beta$};
    \node at (152:1.1) {\small $\alpha$};
    \node at (172:1.1) {\small $\delta$};

    \node at (234:0.7) {\small $\delta$};
    \node at (224:1.1) {\small $\epsilon$};
    \node at (244:1.1) {\small $\delta$};

    \node at (-54:0.7) {\small $\epsilon$};


    \node at (96:1.6) {\small $\epsilon$};

    \node at (156:1.6) {\small $\beta$};
    \node at (168:1.6) {\small $\beta$};

    \node at (228:1.6) {\small $\gamma$};



    \node at (126:2.2) {\small $\delta$};

    \node at (198:2.2) {\small $\alpha$};



    \end{scope}

    \end{tikzpicture}
\caption{Tilings for the third arrangement.}
\label{abcde3B}
\end{figure}
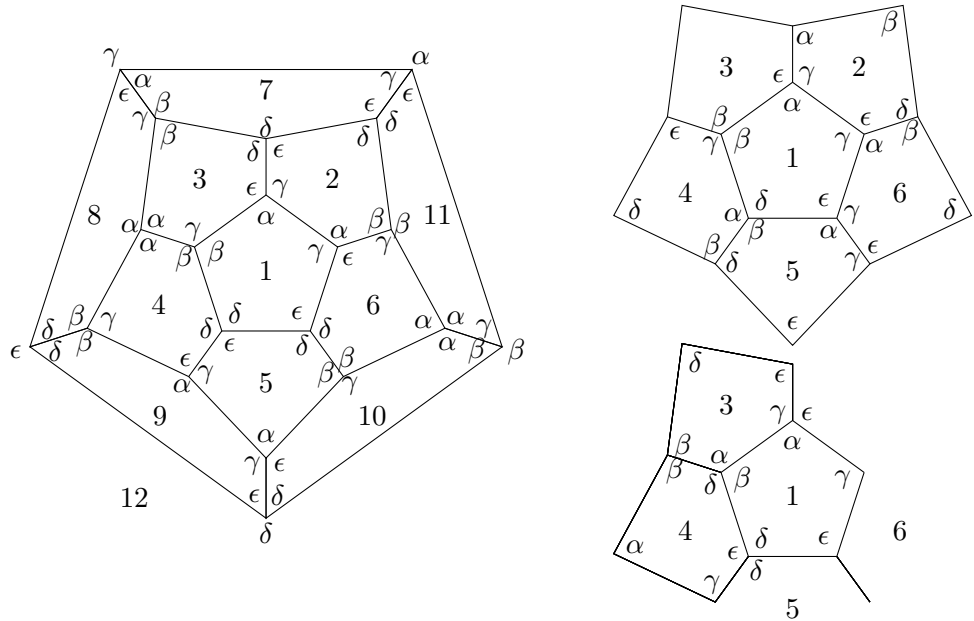

\medskip

{\bf Arrangement 4, Case 1}: $A_{2,13}=\gamma$, $A_{3,12}=\epsilon$. 

By AVC, we get all the angles of $P_2$, $P_6$ and $A_{5,16}=\beta$. If $P_5$ is counterclockwise oriented, as on the left of Figure \ref{abcde4A}, then by AVC, we get all the angles of $P_4$ and $A_{3,14}=\beta$. Then $\beta$ and $\epsilon$ are adjacent in $P_3$, a contradiction. If $P_5$ is clockwise oriented, as on the right of Figure \ref{abcde4A}, then by AVC, we get all the angles of $P_{10}$ and $A_{11,6\overline{10}}=\delta$. Then one of $A_{11,26}$, $A_{11,\overline{10}\,\overline{12}}$ adjacent to $A_{11,6\overline{10}}$ in $P_{11}$ must be $\gamma$. This contradicts to AVC at $V_{26\overline{11}}$ or $A_{\overline{10}\,\overline{11}\,\overline{12}}$.

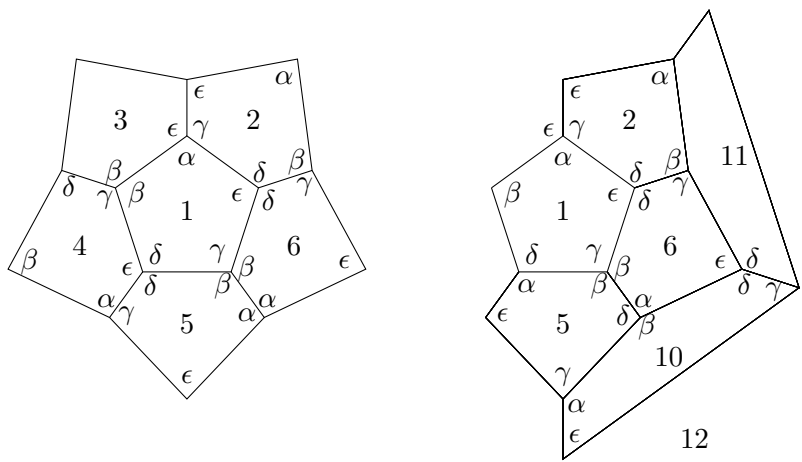
\begin{figure}[htp]
\centering
    \begin{tikzpicture}[scale=1]

\foreach \x in {1,...,5}
    \draw
    		(-54+72*\x:1) -- (18+72*\x:1)
    		(-54+72*\x:1) -- (-54+72*\x:1.75) -- (-18+72*\x:2.5) -- (18+72*\x:1.75);

    \node at (0:0) {\small $1$};
    \node at (54:1.5) {\small $2$};
    \node at (126:1.5) {\small $3$};
    \node at (198:1.5) {\small $4$};
    \node at (-90:1.5) {\small $5$};
    \node at (-18:1.5) {\small $6$};

    \node at (18:0.7) {\small $\epsilon$};
    \node at (8:1.1) {\small $\delta$};
    \node at (28:1.1) {\small $\delta$};

    \node at (90:0.7) {\small $\alpha$};
    \node at (80:1.1) {\small $\gamma$};
    \node at (100:1.1) {\small $\epsilon$};

    \node at (162:0.7) {\small $\beta$};
    \node at (152:1.1) {\small $\beta$};
    \node at (172:1.1) {\small $\gamma$};

    \node at (234:0.7) {\small $\delta$};
    \node at (224:1.1) {\small $\epsilon$};
    \node at (244:1.1) {\small $\delta$};

    \node at (-54:0.7) {\small $\gamma$};
    \node at (-44:1.1) {\small $\beta$};
    \node at (-64:1.1) {\small $\beta$};

    \node at (12:1.6) {\small $\gamma$};
    \node at (24:1.6) {\small $\beta$};

    \node at (84:1.6) {\small $\epsilon$};

    \node at (168:1.6) {\small $\delta$};

    \node at (228:1.6) {\small $\alpha$};
    \node at (240:1.6) {\small $\gamma$};

    \node at (-48:1.6) {\small $\alpha$};
    \node at (-60:1.6) {\small $\alpha$};

    \node at (54:2.2) {\small $\alpha$};


    \node at (198:2.2) {\small $\beta$};

    \node at (-90:2.2) {\small $\epsilon$};

    \node at (-18:2.2) {\small $\epsilon$};

	\begin{scope}[xshift=5cm]

	\foreach \x in {1,...,5}
    \draw
    	(-54+72*\x:1) -- (18+72*\x:1)
    	(-126:1) -- (-126:1.75) -- (-90:2.5) -- (-54:1.75)
		(-54:1) -- (-54:1.75) -- (-18:2.5) -- (18:1.75)
		(18:1) -- (18:1.75) -- (54:2.5) -- (90:1.75) -- (90:1)
		(-18:2.5) -- (-18:3.3) -- (54:3.3) -- (54:2.5)
		(-90:2.5) -- (-90:3.3) -- (-18:3.3);

    \node at (0:0) {\small $1$};
    \node at (54:1.5) {\small $2$};
    \node at (-90:1.5) {\small $5$};
    \node at (-18:1.5) {\small $6$};
    \node at (-54:2.4) {\small $10$};
    \node at (18:2.4) {\small $11$};
    \node at (300:3.5){\small $12$};

    \node at (18:0.7) {\small $\epsilon$};
    \node at (8:1.1) {\small $\delta$};
    \node at (28:1.1) {\small $\delta$};

    \node at (90:0.7) {\small $\alpha$};
    \node at (80:1.1) {\small $\gamma$};
    \node at (100:1.1) {\small $\epsilon$};

    \node at (162:0.7) {\small $\beta$};

    \node at (234:0.7) {\small $\delta$};
    \node at (244:1.1) {\small $\alpha$};

    \node at (-54:0.7) {\small $\gamma$};
    \node at (-44:1.1) {\small $\beta$};
    \node at (-64:1.1) {\small $\beta$};

    \node at (12:1.6) {\small $\gamma$};
    \node at (24:1.6) {\small $\beta$};

    \node at (84:1.6) {\small $\epsilon$};


    \node at (240:1.6) {\small $\epsilon$};

    \node at (-54:1.9) {\small $\beta$};
    \node at (-48:1.6) {\small $\alpha$};
    \node at (-60:1.6) {\small $\delta$};

    \node at (54:2.2) {\small $\alpha$};



    \node at (-90:2.2) {\small $\gamma$};
    \node at (-86:2.6) {\small $\alpha$};

    \node at (-18:2.2) {\small $\epsilon$};
    \node at (-22:2.6) {\small $\delta$};
    \node at (-14:2.6) {\small $\delta$};




    \node at (-87:3) {\small $\epsilon$};

    \node at (-21:3) {\small $\gamma$};

    \end{scope}

    \end{tikzpicture}
\caption{Tilings for the fourth arrangement.}
\label{abcde4A}
\end{figure}

\medskip

{\bf Arrangement 4, Case 2}: $A_{2,13}=\epsilon$, $A_{3,12}=\gamma$. 

If $P_2$ is counterclockwise oriented, as on the left of Figure \ref{abcde4B}, then by AVC, we successively get all the angles of $P_3$, $P_6$, $P_{11}$ and $A_{7,23}=\alpha$, $A_{7,2\overline{11}}=\delta$. We find $\alpha$ and $\delta$ adjacent in $P_7$, a contradiction. If $P_2$ is clockwise oriented, as on the right of Figure \ref{abcde4B}, then by AVC, we successively get all the angles of $P_3$, $P_7$, $P_8$ and $A_{4,13}=A_{4,38}=\alpha$. We find two $\alpha$ in $P_4$, a contradiction.

\begin{figure}[htp]
\centering
    \begin{tikzpicture}[scale=1]

	\foreach \x in {1,...,5}
    \draw
    		(-54+72*\x:1) -- (18+72*\x:1)
    		(-54:1) -- (-54:1.75) -- (-18:2.5) -- (18:1.75)
		(18:1) -- (18:1.75) -- (54:2.5) -- (90:1.75)
		(90:1) -- (90:1.75) -- (126:2.5) -- (162:1.75) -- (162:1)
		(-18:2.5) -- (-18:3.3) -- (54:3.3)
		(54:2.5) -- (54:3.3) -- (126:3.3) -- (126:2.5);

    \node at (0:0) {\small $1$};
    \node at (54:1.5) {\small $2$};
    \node at (126:1.5) {\small $3$};
    \node at (-18:1.5) {\small $6$};
    \node at (90:2.4) {\small $7$};
    \node at (18:2.4) {\small $11$};

    \node at (18:0.7) {\small $\epsilon$};
    \node at (8:1.1) {\small $\gamma$};
    \node at (28:1.1) {\small $\alpha$};

    \node at (90:0.7) {\small $\alpha$};
    \node at (80:1.1) {\small $\epsilon$};
    \node at (100:1.1) {\small $\gamma$};

    \node at (162:0.7) {\small $\beta$};
    \node at (152:1.1) {\small $\delta$};

    \node at (234:0.7) {\small $\delta$};

    \node at (-54:0.7) {\small $\gamma$};
    \node at (-44:1.1) {\small $\epsilon$};

    \node at (18:1.9) {\small $\alpha$};
    \node at (12:1.6) {\small $\delta$};
    \node at (24:1.6) {\small $\beta$};

    \node at (90:1.9) {\small $\alpha$};
    \node at (84:1.6) {\small $\gamma$};
    \node at (96:1.6) {\small $\epsilon$};

    \node at (156:1.6) {\small $\beta$};


    \node at (-48:1.6) {\small $\alpha$};

    \node at (54:2.2) {\small $\delta$};
    \node at (50:2.6) {\small $\epsilon$};
    \node at (58:2.6) {\small $\delta$};

    \node at (126:2.2) {\small $\alpha$};



    \node at (-18:2.2) {\small $\beta$};
    \node at (-14:2.6) {\small $\beta$};

    \node at (51:3) {\small $\gamma$};




    \node at (-15:3) {\small $\delta$};

	\begin{scope}[xshift=7cm]

	\foreach \x in {1,...,5}
    \draw
    		(-54+72*\x:1) -- (18+72*\x:1)
		(18:1) -- (18:1.75) -- (54:2.5) -- (90:1.75)
		(90:1) -- (90:1.75) -- (126:2.5) -- (162:1.75)
		(162:1) -- (162:1.75) -- (198:2.5) -- (234:1.75) -- (234:1)
		(54:2.5) -- (54:3.3) -- (126:3.3) 
		(126:2.5) -- (126:3.3) -- (198:3.3) -- (198:2.5);

    \node at (0:0) {\small $1$};
    \node at (54:1.5) {\small $2$};
    \node at (126:1.5) {\small $3$};
    \node at (198:1.5) {\small $4$};
    \node at (90:2.4) {\small $7$};
    \node at (162:2.4) {\small $8$};

    \node at (18:0.7) {\small $\epsilon$};
    \node at (28:1.1) {\small $\gamma$};

    \node at (90:0.7) {\small $\alpha$};
    \node at (80:1.1) {\small $\epsilon$};
    \node at (100:1.1) {\small $\gamma$};

    \node at (162:0.7) {\small $\beta$};
    \node at (152:1.1) {\small $\delta$};
    \node at (172:1.1) {\small $\alpha$};

    \node at (234:0.7) {\small $\delta$};

    \node at (-54:0.7) {\small $\gamma$};

    \node at (24:1.6) {\small $\delta$};

    \node at (90:1.9) {\small $\gamma$};
    \node at (84:1.6) {\small $\alpha$};
    \node at (96:1.6) {\small $\epsilon$};

    \node at (162:1.9) {\small $\delta$};
    \node at (156:1.6) {\small $\beta$};
    \node at (168:1.6) {\small $\alpha$};



    \node at (54:2.2) {\small $\beta$};
    \node at (58:2.6) {\small $\delta$};

    \node at (126:2.2) {\small $\alpha$};
    \node at (122:2.6) {\small $\epsilon$};
    \node at (130:2.6) {\small $\gamma$};

    \node at (194:2.6) {\small $\beta$};



    \node at (57:3) {\small $\beta$};

    \node at (123:3) {\small $\alpha$};
    \node at (129:3) {\small $\epsilon$};

    \node at (195:3) {\small $\alpha$};



    \end{scope}

    \end{tikzpicture}

\caption{Tilings for the fourth arrangement, continued.}
\label{abcde4B}
\end{figure}
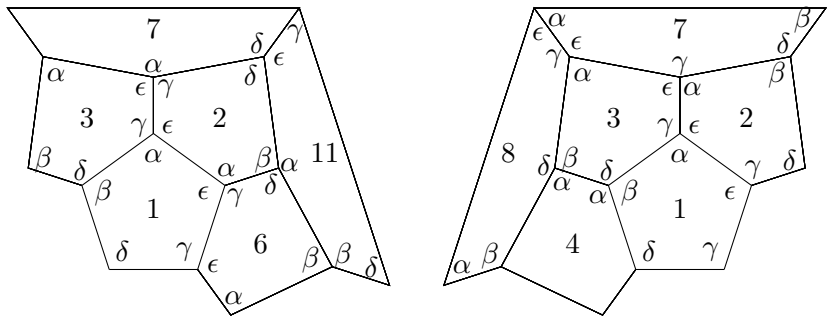

\medskip

{\bf Arrangement 5, Case 2}: $A_{2,13}=\epsilon$, $A_{3,12}=\gamma$. 

We treat the second case first because the conclusion will be useful for the first case. By AVC, we get all the angles of $P_2$, $P_3$. If $P_4$ is clockwise oriented, as on the left of Figure \ref{abcde5A}, then by AVC, we successively get all the angles of $P_5$, $P_6$, $P_9$, $P_{10}$. Then we get a contradiction to AVC at $V_{6\overline{10}\,\overline{11}}$. 

In the argument above, we started from clockwise $P_4$ and derived counterclockwise $P_6$. The same argument can be reversed, so that starting from counterclockwise $P_6$, we can derive clockwise $P_4$. On the other hand, if in the left of Figure \ref{abcde5A}, we exchange $\beta$ with $\delta$, $\epsilon$ with $\gamma$, and then flip the picture horizontally, we see that the argument also tells us that $P_6$ cannot be clockwise. We conclude that $P_6$ can be neither counterclockwise nor clockwise, a contradiction.

We remark that the conclusion of the case is the following: If the angles in a tile (of fifth arrangement) are counterclockwise oriented, and the vertex at $\alpha$ is $\alpha\gamma\epsilon$, then $\alpha,\gamma,\epsilon$ must be counterclockwise oriented at the vertex.

\medskip

{\bf Arrangement 5, Case 1}: $A_{2,13}=\gamma$, $A_{3,12}=\epsilon$. 

If $P_2$ is counterclockwise oriented, as on the upper right of Figure \ref{abcde5A}, then by AVC, we get all the angles of $P_6$ and $A_{5,16}=\gamma$. Now $P_6$ is counterclockwise oriented and $\alpha,\gamma,\epsilon$ are clockwise oriented at $V_{156}$. This contradicts to the conclusion of Case 2. If $P_2$ is clockwise oriented, as on the lower right of Figure \ref{abcde5A}, then by AVC, we successively get all the angles of $P_6$, $P_5$, $P_4$ and $A_{3,14}=\alpha$. We find $\alpha,\epsilon$ adjacent in $P_3$, a contradiction.

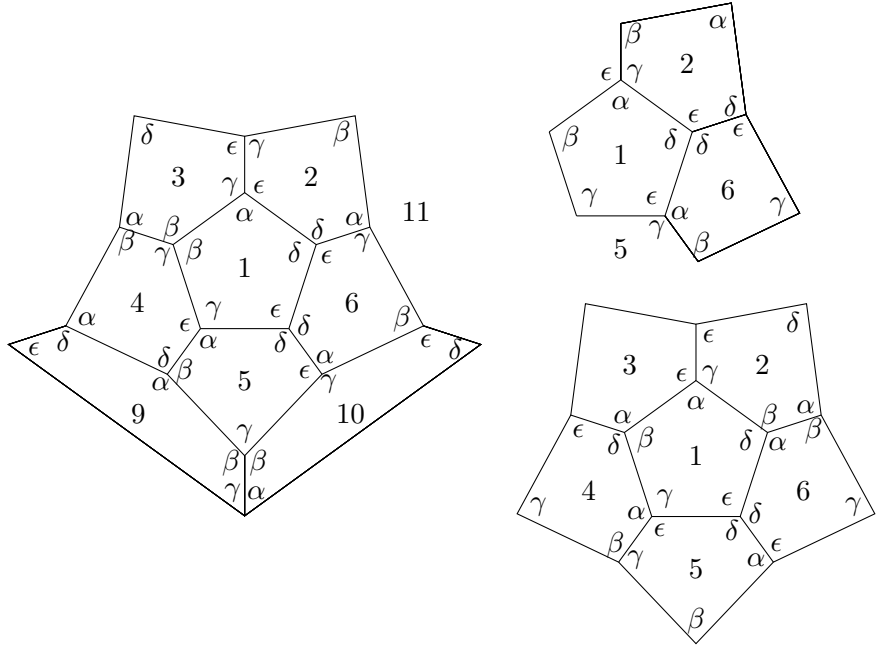
\begin{figure}[htp]
\centering
    \begin{tikzpicture}[scale=1]

	\foreach \x in {1,...,5}
    \draw
    		(-54+72*\x:1) -- (18+72*\x:1)
    		(-54+72*\x:1) -- (-54+72*\x:1.75) -- (-18+72*\x:2.5) -- (18+72*\x:1.75)
		(198:2.5) -- (198:3.3) -- (270:3.3) 
		(270:2.5) -- (270:3.3) -- (-18:3.3) -- (-18:2.5);

    \node at (0:0) {\small $1$};
    \node at (54:1.5) {\small $2$};
    \node at (126:1.5) {\small $3$};
    \node at (198:1.5) {\small $4$};
    \node at (-90:1.5) {\small $5$};
    \node at (-18:1.5) {\small $6$};
    \node at (234:2.4) {\small $9$};
    \node at (-54:2.4) {\small $10$};
    \node at (18:2.4) {\small $11$};

    \node at (18:0.7) {\small $\delta$};
    \node at (8:1.1) {\small $\epsilon$};
    \node at (28:1.1) {\small $\delta$};

    \node at (90:0.7) {\small $\alpha$};
    \node at (80:1.1) {\small $\epsilon$};
    \node at (100:1.1) {\small $\gamma$};

    \node at (162:0.7) {\small $\beta$};
    \node at (152:1.1) {\small $\beta$};
    \node at (172:1.1) {\small $\gamma$};

    \node at (234:0.7) {\small $\gamma$};
    \node at (224:1.1) {\small $\epsilon$};
    \node at (244:1.1) {\small $\alpha$};

    \node at (-54:0.7) {\small $\epsilon$};
    \node at (-44:1.1) {\small $\delta$};
    \node at (-64:1.1) {\small $\delta$};

    \node at (12:1.6) {\small $\gamma$};
    \node at (24:1.6) {\small $\alpha$};

    \node at (84:1.6) {\small $\gamma$};
    \node at (96:1.6) {\small $\epsilon$};

    \node at (156:1.6) {\small $\alpha$};
    \node at (168:1.6) {\small $\beta$};

    \node at (234:1.9) {\small $\alpha$};
    \node at (228:1.6) {\small $\delta$};
    \node at (240:1.6) {\small $\beta$};

    \node at (-54:1.9) {\small $\gamma$};
    \node at (-48:1.6) {\small $\alpha$};
    \node at (-60:1.6) {\small $\epsilon$};

    \node at (54:2.2) {\small $\beta$};

    \node at (126:2.2) {\small $\delta$};

    \node at (198:2.2) {\small $\alpha$};
    \node at (202:2.6) {\small $\delta$};

    \node at (-90:2.2) {\small $\gamma$};
    \node at (-94:2.6) {\small $\beta$};
    \node at (-86:2.6) {\small $\beta$};

    \node at (-18:2.2) {\small $\beta$};
    \node at (-22:2.6) {\small $\epsilon$};



    \node at (201:3) {\small $\epsilon$};

    \node at (-93:3) {\small $\gamma$};
    \node at (-87:3) {\small $\alpha$};

    \node at (-21:3) {\small $\delta$};

\begin{scope}[shift={(5cm,1.5cm)}]
    
    	\foreach \x in {1,...,5}
    \draw
    		(-54+72*\x:1) -- (18+72*\x:1)
    		(-54:1) -- (-54:1.75) -- (-18:2.5) -- (18:1.75)
		(18:1) -- (18:1.75) -- (54:2.5) -- (90:1.75) -- (90:1);

    \node at (0:0) {\small $1$};
    \node at (54:1.5) {\small $2$};
    \node at (-90:1.25) {\small $5$};
    \node at (-18:1.5) {\small $6$};

    \node at (18:0.7) {\small $\delta$};
    \node at (8:1.1) {\small $\delta$};
    \node at (28:1.1) {\small $\epsilon$};

    \node at (90:0.7) {\small $\alpha$};
    \node at (80:1.1) {\small $\gamma$};
    \node at (100:1.1) {\small $\epsilon$};

    \node at (162:0.7) {\small $\beta$};

    \node at (234:0.7) {\small $\gamma$};

    \node at (-54:0.7) {\small $\epsilon$};
    \node at (-44:1.1) {\small $\alpha$};
    \node at (-64:1.1) {\small $\gamma$};

    \node at (12:1.6) {\small $\epsilon$};
    \node at (24:1.6) {\small $\delta$};

    \node at (84:1.6) {\small $\beta$};



    \node at (-48:1.6) {\small $\beta$};

    \node at (54:2.2) {\small $\alpha$};




    \node at (-18:2.2) {\small $\gamma$};

    \end{scope}

	\begin{scope}[shift={(6cm,-2.5cm)}]
	
	\foreach \x in {1,...,5}
    \draw
    		(-54+72*\x:1) -- (18+72*\x:1)
    		(-54+72*\x:1) -- (-54+72*\x:1.75) -- (-18+72*\x:2.5) -- (18+72*\x:1.75);

    \node at (0:0) {\small $1$};
    \node at (54:1.5) {\small $2$};
    \node at (126:1.5) {\small $3$};
    \node at (198:1.5) {\small $4$};
    \node at (-90:1.5) {\small $5$};
    \node at (-18:1.5) {\small $6$};

    \node at (18:0.7) {\small $\delta$};
    \node at (8:1.1) {\small $\alpha$};
    \node at (28:1.1) {\small $\beta$};

    \node at (90:0.7) {\small $\alpha$};
    \node at (80:1.1) {\small $\gamma$};
    \node at (100:1.1) {\small $\epsilon$};

    \node at (162:0.7) {\small $\beta$};
    \node at (152:1.1) {\small $\alpha$};
    \node at (172:1.1) {\small $\delta$};

    \node at (234:0.7) {\small $\gamma$};
    \node at (224:1.1) {\small $\alpha$};
    \node at (244:1.1) {\small $\epsilon$};

    \node at (-54:0.7) {\small $\epsilon$};
    \node at (-44:1.1) {\small $\delta$};
    \node at (-64:1.1) {\small $\delta$};

    \node at (12:1.6) {\small $\beta$};
    \node at (24:1.6) {\small $\alpha$};

    \node at (84:1.6) {\small $\epsilon$};

    \node at (168:1.6) {\small $\epsilon$};

    \node at (228:1.6) {\small $\beta$};
    \node at (240:1.6) {\small $\gamma$};

    \node at (-48:1.6) {\small $\epsilon$};
    \node at (-60:1.6) {\small $\alpha$};

    \node at (54:2.2) {\small $\delta$};


    \node at (198:2.2) {\small $\gamma$};

    \node at (-90:2.2) {\small $\beta$};

    \node at (-18:2.2) {\small $\gamma$};

    \end{scope}

    \end{tikzpicture}

\caption{Tilings for the fifth arrangement.}
\label{abcde5A}
\end{figure}

\medskip

{\bf Arrangement 6, Case 1}: $A_{2,13}=\gamma$, $A_{3,12}=\epsilon$. 

If $P_2$ is counterclockwise oriented, as on the left of Figure \ref{abcde6}, then by AVC, we successively get all the angles of $P_6$, $P_5$, $P_4$ and $A_{3,14}=\alpha$. We find $\alpha,\epsilon$ adjacent in $P_3$, a contradiction. If $P_2$ is clockwise oriented, as in the middle of Figure \ref{abcde6}, then by AVC, we successively get all the angles of $P_6$, $P_5$, $P_{10}$ and $A_{11,26}=\beta$, $A_{11,6\overline{10}}=\delta$. We find $\beta,\delta$ adjacent in $P_{11}$, a contradiction.

\begin{figure}[htp]
\centering
    \begin{tikzpicture}[scale=1]

\foreach \x in {1,...,5}
    \draw
    		(-54+72*\x:1) -- (18+72*\x:1)
    		(-54+72*\x:1) -- (-54+72*\x:1.75) -- (-18+72*\x:2.5) -- (18+72*\x:1.75);

    \node at (0:0) {\small $1$};
    \node at (54:1.5) {\small $2$};
    \node at (126:1.5) {\small $3$};
    \node at (198:1.5) {\small $4$};
    \node at (-90:1.5) {\small $5$};
    \node at (-18:1.5) {\small $6$};

    \node at (18:0.7) {\small $\delta$};
    \node at (8:1.1) {\small $\epsilon$};
    \node at (28:1.1) {\small $\delta$};

    \node at (90:0.7) {\small $\alpha$};
    \node at (80:1.1) {\small $\gamma$};
    \node at (100:1.1) {\small $\epsilon$};

    \node at (162:0.7) {\small $\beta$};
    \node at (152:1.1) {\small $\alpha$};
    \node at (172:1.1) {\small $\delta$};

    \node at (234:0.7) {\small $\epsilon$};
    \node at (224:1.1) {\small $\gamma$};
    \node at (244:1.1) {\small $\alpha$};

    \node at (-54:0.7) {\small $\gamma$};
    \node at (-44:1.1) {\small $\beta$};
    \node at (-64:1.1) {\small $\beta$};

    \node at (12:1.6) {\small $\gamma$};
    \node at (24:1.6) {\small $\alpha$};

    \node at (84:1.6) {\small $\epsilon$};

    \node at (168:1.6) {\small $\alpha$};

    \node at (228:1.6) {\small $\epsilon$};
    \node at (240:1.6) {\small $\delta$};

    \node at (-48:1.6) {\small $\alpha$};
    \node at (-60:1.6) {\small $\epsilon$};

    \node at (54:2.2) {\small $\beta$};


    \node at (198:2.2) {\small $\beta$};

    \node at (-90:2.2) {\small $\gamma$};

    \node at (-18:2.2) {\small $\delta$};

	\begin{scope}[xshift=4cm]

	\foreach \x in {1,...,5}
    \draw
    		(-54+72*\x:1) -- (18+72*\x:1)
    		(-126:1) -- (-126:1.75) -- (-90:2.5) -- (-54:1.75)
		(-54:1) -- (-54:1.75) -- (-18:2.5) -- (18:1.75)
		(18:1) -- (18:1.75) -- (54:2.5) -- (90:1.75) -- (90:1)
		(-18:2.5) -- (-18:3.3) -- (54:3.3) -- (54:2.5)
		(-90:2.5) -- (-90:3.3) -- (-18:3.3);

    \node at (0:0) {\small $1$};
    \node at (54:1.5) {\small $2$};
    \node at (-90:1.5) {\small $5$};
    \node at (-18:1.5) {\small $6$};
    \node at (-54:2.4) {\small $10$};
    \node at (18:2.4) {\small $11$};
    \node at (300:3.5){\small $12$};

    \node at (18:0.7) {\small $\delta$};
    \node at (8:1.1) {\small $\delta$};
    \node at (28:1.1) {\small $\epsilon$};

    \node at (90:0.7) {\small $\alpha$};
    \node at (80:1.1) {\small $\gamma$};
    \node at (100:1.1) {\small $\epsilon$};

    \node at (162:0.7) {\small $\beta$};

    \node at (234:0.7) {\small $\epsilon$};
    \node at (244:1.1) {\small $\gamma$};

    \node at (-54:0.7) {\small $\gamma$};
    \node at (-44:1.1) {\small $\alpha$};
    \node at (-64:1.1) {\small $\epsilon$};

    \node at (18:1.9) {\small $\beta$};
    \node at (12:1.6) {\small $\gamma$};
    \node at (24:1.6) {\small $\beta$};

    \node at (84:1.6) {\small $\delta$};


    \node at (240:1.6) {\small $\delta$};

    \node at (-54:1.9) {\small $\gamma$};
    \node at (-48:1.6) {\small $\beta$};
    \node at (-60:1.6) {\small $\beta$};

    \node at (54:2.2) {\small $\alpha$};



    \node at (-90:2.2) {\small $\alpha$};
    \node at (-86:2.6) {\small $\epsilon$};

    \node at (-18:2.2) {\small $\epsilon$};
    \node at (-22:2.6) {\small $\delta$};
    \node at (-14:2.6) {\small $\delta$};




    \node at (-87:3) {\small $\beta$};

    \node at (-21:3) {\small $\alpha$};

    \end{scope}

    \begin{scope}[xshift=10cm]

	\foreach \x in {1,...,5}
    \draw
    		(-54+72*\x:1) -- (18+72*\x:1)
    		(-54+72*\x:1) -- (-54+72*\x:1.75) -- (-18+72*\x:2.5) -- (18+72*\x:1.75);

    \node at (0:0) {\small $1$};
    \node at (54:1.5) {\small $2$};
    \node at (126:1.5) {\small $3$};
    \node at (198:1.5) {\small $4$};
    \node at (-90:1.5) {\small $5$};
    \node at (-18:1.5) {\small $6$};

    \node at (18:0.7) {\small $\delta$};
    \node at (8:1.1) {\small $\alpha$};
    \node at (28:1.1) {\small $\beta$};

    \node at (90:0.7) {\small $\alpha$};
    \node at (80:1.1) {\small $\epsilon$};
    \node at (100:1.1) {\small $\gamma$};

    \node at (162:0.7) {\small $\beta$};
    \node at (152:1.1) {\small $\delta$};
    \node at (172:1.1) {\small $\alpha$};

    \node at (234:0.7) {\small $\epsilon$};
    \node at (224:1.1) {\small $\delta$};
    \node at (244:1.1) {\small $\delta$};

    \node at (-54:0.7) {\small $\gamma$};
    \node at (-44:1.1) {\small $\beta$};
    \node at (-64:1.1) {\small $\beta$};

    \node at (12:1.6) {\small $\delta$};
    \node at (24:1.6) {\small $\alpha$};

    \node at (84:1.6) {\small $\gamma$};
    \node at (96:1.6) {\small $\epsilon$};

    \node at (156:1.6) {\small $\alpha$};
    \node at (168:1.6) {\small $\beta$};

    \node at (228:1.6) {\small $\gamma$};

    \node at (-48:1.6) {\small $\epsilon$};

    \node at (54:2.2) {\small $\delta$};

    \node at (126:2.2) {\small $\beta$};

    \node at (198:2.2) {\small $\epsilon$};


    \node at (-18:2.2) {\small $\gamma$};

    \end{scope}

    \end{tikzpicture}
\caption{Tilings for the sixth arrangement.}
\label{abcde6}
\end{figure}
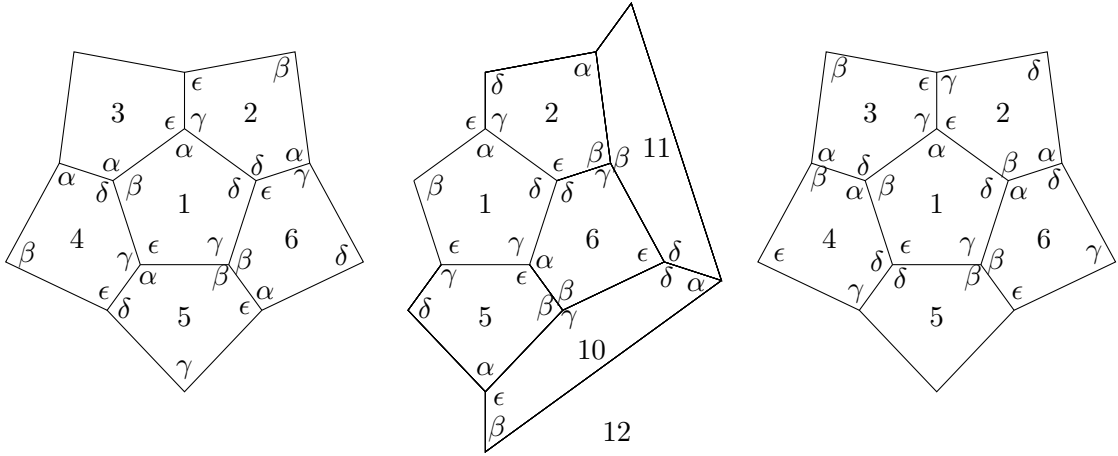

\medskip

{\bf Arrangement 6, Case 2}: $A_{2,13}=\epsilon$, $A_{3,12}=\gamma$. 

The case is the right of Figure \ref{abcde6}. By AVC, we successively get all the angles of $P_2$, $P_3$, $P_4$, $P_6$ and $A_{5,14}=\delta$, $A_{5,16}=\beta$. We find $\beta,\delta$ adjacent in $P_5$, a contradiction.

\medskip

{\bf Arrangement 7, Case 1}: $A_{2,13}=\gamma$, $A_{3,12}=\epsilon$.

The case is the upper right of Figure \ref{abcde7B}. By AVC, we successively get all the angles of $P_2$, $P_3$, $P_4$, $P_6$ and $A_{5,14}=A_{5,16}=\alpha$. We find two $\alpha$ in $P_5$, a contradiction.

\medskip

{\bf Arrangement 7, Case 2}: $A_{2,13}=\epsilon$, $A_{3,12}=\gamma$.

If $P_2$ is clockwise oriented, as on the left of Figure \ref{abcde7B}, then by AVC, we get all the angles of $P_6$, $P_5$ and $A_{11,26}=A_{10,56}=\gamma$. Since $\gamma,\epsilon$ are not adjacent, neither $A_{10,6\overline{11}}$ nor $A_{11,6\overline{10}}$ can be $\epsilon$. Then by $A_{6,\overline{10}\,\overline{11}}=\gamma$ and AVC, we get $A_{10,6\overline{11}}=A_{11,6\overline{10}}=\beta$ and then all the angles of $P_{10}$, $P_{11}$. By the angles we know so far and AVC, we successively get all the angles of $P_9$, $P_4$, $P_3$, $P_7$, $P_8$, $P_{12}$. There is no contradiction, and we get an angle congruent tiling with anglewise vertex combination $\{2\alpha^3,6\alpha\gamma\epsilon,6\beta^2\gamma,6\delta^2\epsilon\}$.

By the same argument (in the left picture, exchanging $\beta$ with $\delta$, $\epsilon$ with $\gamma$, and then flipping horizontally), we also conclude that, if $P_3$ is clockwise oriented, then we get the same angle congruent tiling on the left of Figure \ref{abcde7B}. Therefore the only remaining case is that both $P_2,P_3$ are counterclockwise oriented, as on the lower right of Figure \ref{abcde7B}. By AVC, we get all the angles of $P_7$ and $A_{8,37}=\delta$. Since $A_{8,34}$ is adjacent to $\alpha$ in $P_8$, we get $A_{8,34}=\beta$ or $\epsilon$. Either way contradicts to AVC at $V_{348}$.

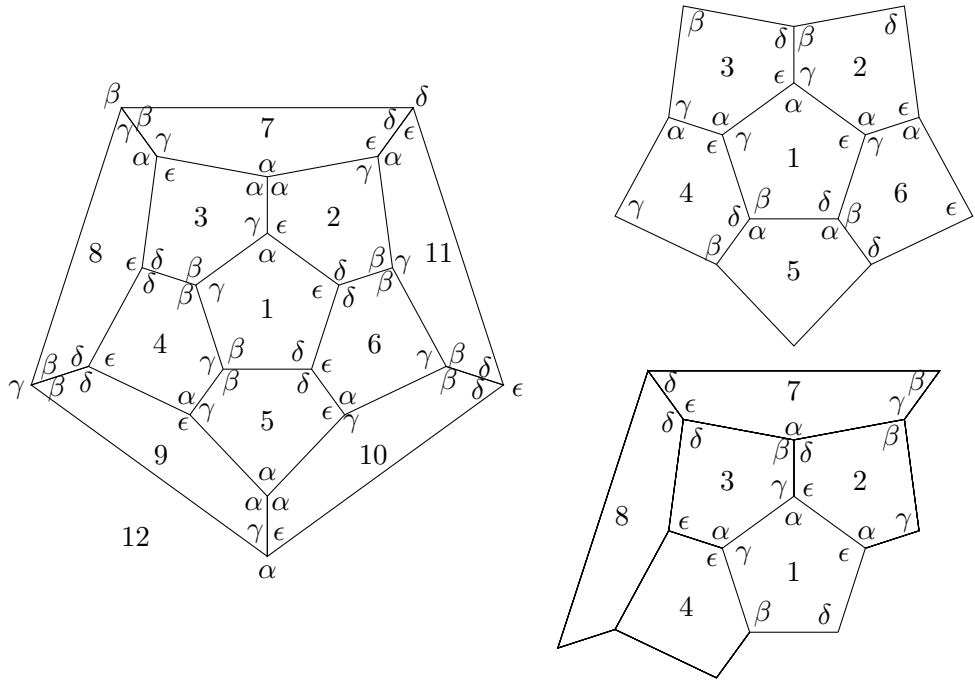
\begin{figure}[htp]
\centering
    \begin{tikzpicture}[scale=1]

	\foreach \x in {1,...,5}
    \draw
    		(-54+72*\x:1) -- (18+72*\x:1)
    		(-54+72*\x:1) -- (-54+72*\x:1.75) -- (-18+72*\x:2.5) -- (18+72*\x:1.75)
		(-18+72*\x:2.5) -- (-18+72*\x:3.3) -- (54+72*\x:3.3) -- (54+72*\x:2.5);

    \node at (0:0) {\small $1$};
    \node at (54:1.5) {\small $2$};
    \node at (126:1.5) {\small $3$};
    \node at (198:1.5) {\small $4$};
    \node at (-90:1.5) {\small $5$};
    \node at (-18:1.5) {\small $6$};
    \node at (90:2.4) {\small $7$};
    \node at (162:2.4) {\small $8$};
    \node at (234:2.4) {\small $9$};
    \node at (-54:2.4) {\small $10$};
    \node at (18:2.4) {\small $11$};
    \node at (240:3.5){\small $12$};

    \node at (18:0.7) {\small $\epsilon$};
    \node at (8:1.1) {\small $\delta$};
    \node at (28:1.1) {\small $\delta$};

    \node at (90:0.7) {\small $\alpha$};
    \node at (80:1.1) {\small $\epsilon$};
    \node at (100:1.1) {\small $\gamma$};

    \node at (162:0.7) {\small $\gamma$};
    \node at (152:1.1) {\small $\beta$};
    \node at (172:1.1) {\small $\beta$};

    \node at (234:0.7) {\small $\beta$};
    \node at (224:1.1) {\small $\gamma$};
    \node at (244:1.1) {\small $\beta$};

    \node at (-54:0.7) {\small $\delta$};
    \node at (-44:1.1) {\small $\epsilon$};
    \node at (-64:1.1) {\small $\delta$};

    \node at (18:1.9) {\small $\gamma$};
    \node at (12:1.6) {\small $\beta$};
    \node at (24:1.6) {\small $\beta$};

    \node at (90:1.9) {\small $\alpha$};
    \node at (84:1.6) {\small $\alpha$};
    \node at (96:1.6) {\small $\alpha$};

    \node at (162:1.9) {\small $\epsilon$};
    \node at (156:1.6) {\small $\delta$};
    \node at (168:1.6) {\small $\delta$};

    \node at (234:1.9) {\small $\epsilon$};
    \node at (228:1.6) {\small $\alpha$};
    \node at (240:1.6) {\small $\gamma$};

    \node at (-54:1.9) {\small $\gamma$};
    \node at (-48:1.6) {\small $\alpha$};
    \node at (-60:1.6) {\small $\epsilon$};

    \node at (54:2.2) {\small $\gamma$};
    \node at (50:2.6) {\small $\alpha$};
    \node at (58:2.6) {\small $\epsilon$};

    \node at (126:2.2) {\small $\epsilon$};
    \node at (122:2.6) {\small $\gamma$};
    \node at (130:2.6) {\small $\alpha$};

    \node at (198:2.2) {\small $\epsilon$};
    \node at (194:2.6) {\small $\delta$};
    \node at (202:2.6) {\small $\delta$};

    \node at (-90:2.2) {\small $\alpha$};
    \node at (-94:2.6) {\small $\alpha$};
    \node at (-86:2.6) {\small $\alpha$};

    \node at (-18:2.2) {\small $\gamma$};
    \node at (-22:2.6) {\small $\beta$};
    \node at (-14:2.6) {\small $\beta$};

    \node at (54:3.5) {\small $\delta$};
    \node at (51:3) {\small $\epsilon$};
    \node at (57:3) {\small $\delta$};

    \node at (126:3.5) {\small $\beta$};
    \node at (123:3) {\small $\beta$};
    \node at (129:3) {\small $\gamma$};

    \node at (198:3.5) {\small $\gamma$};
    \node at (195:3) {\small $\beta$};
    \node at (201:3) {\small $\beta$};

    \node at (-90:3.5) {\small $\alpha$};
    \node at (-93:3) {\small $\gamma$};
    \node at (-87:3) {\small $\epsilon$};

    \node at (-18:3.5) {\small $\epsilon$};
    \node at (-15:3) {\small $\delta$};
    \node at (-21:3) {\small $\delta$};

\begin{scope}[shift={(7cm,2cm)}]

    \foreach \x in {1,...,5}
    \draw
    		(-54+72*\x:1) -- (18+72*\x:1)
    		(-54+72*\x:1) -- (-54+72*\x:1.75) -- (-18+72*\x:2.5) -- (18+72*\x:1.75);

    \node at (0:0) {\small $1$};
    \node at (54:1.5) {\small $2$};
    \node at (126:1.5) {\small $3$};
    \node at (198:1.5) {\small $4$};
    \node at (-90:1.5) {\small $5$};
    \node at (-18:1.5) {\small $6$};

    \node at (18:0.7) {\small $\epsilon$};
    \node at (8:1.1) {\small $\gamma$};
    \node at (28:1.1) {\small $\alpha$};

    \node at (90:0.7) {\small $\alpha$};
    \node at (80:1.1) {\small $\gamma$};
    \node at (100:1.1) {\small $\epsilon$};

    \node at (162:0.7) {\small $\gamma$};
    \node at (152:1.1) {\small $\alpha$};
    \node at (172:1.1) {\small $\epsilon$};

    \node at (234:0.7) {\small $\beta$};
    \node at (224:1.1) {\small $\delta$};
    \node at (244:1.1) {\small $\alpha$};

    \node at (-54:0.7) {\small $\delta$};
    \node at (-44:1.1) {\small $\beta$};
    \node at (-64:1.1) {\small $\alpha$};

    \node at (12:1.6) {\small $\alpha$};
    \node at (24:1.6) {\small $\epsilon$};

    \node at (84:1.6) {\small $\beta$};
    \node at (96:1.6) {\small $\delta$};

    \node at (156:1.6) {\small $\gamma$};
    \node at (168:1.6) {\small $\alpha$};

    \node at (228:1.6) {\small $\beta$};

    \node at (-48:1.6) {\small $\delta$};

    \node at (54:2.2) {\small $\delta$};

    \node at (126:2.2) {\small $\beta$};

    \node at (198:2.2) {\small $\gamma$};


    \node at (-18:2.2) {\small $\epsilon$};
    
\end{scope}

\begin{scope}[shift={(7cm,-3.5cm)}]
    
    \foreach \x in {1,...,5}
    \draw
    		(-54+72*\x:1) -- (18+72*\x:1)
		(18:1) -- (18:1.75) -- (54:2.5) -- (90:1.75)
		(90:1) -- (90:1.75) -- (126:2.5) -- (162:1.75)
		(162:1) -- (162:1.75) -- (198:2.5) -- (234:1.75) -- (234:1)
		(54:2.5) -- (54:3.3) -- (126:3.3) 
		(126:2.5) -- (126:3.3) -- (198:3.3) -- (198:2.5);

    \node at (0:0) {\small $1$};
    \node at (54:1.5) {\small $2$};
    \node at (126:1.5) {\small $3$};
    \node at (198:1.5) {\small $4$};
    \node at (90:2.4) {\small $7$};
    \node at (162:2.4) {\small $8$};

    \node at (18:0.7) {\small $\epsilon$};
    \node at (28:1.1) {\small $\alpha$};

    \node at (90:0.7) {\small $\alpha$};
    \node at (80:1.1) {\small $\epsilon$};
    \node at (100:1.1) {\small $\gamma$};

    \node at (162:0.7) {\small $\gamma$};
    \node at (152:1.1) {\small $\alpha$};
    \node at (172:1.1) {\small $\epsilon$};

    \node at (234:0.7) {\small $\beta$};

    \node at (-54:0.7) {\small $\delta$};

    \node at (24:1.6) {\small $\gamma$};

    \node at (90:1.9) {\small $\alpha$};
    \node at (84:1.6) {\small $\delta$};
    \node at (96:1.6) {\small $\beta$};

    \node at (156:1.6) {\small $\epsilon$};



    \node at (54:2.2) {\small $\beta$};
    \node at (58:2.6) {\small $\gamma$};

    \node at (126:2.2) {\small $\delta$};
    \node at (122:2.6) {\small $\epsilon$};
    \node at (130:2.6) {\small $\delta$};




    \node at (57:3) {\small $\beta$};

    \node at (123:3) {\small $\delta$};




    \end{scope}

    \end{tikzpicture}
\caption{Tilings for the seventh arrangement.}
\label{abcde7B}
\end{figure}

\medskip

{\bf Arrangement 8, Case 1}: $A_{2,13}=\gamma$, $A_{3,12}=\epsilon$.

The case is the left of Figure \ref{abcde8A}. By AVC, we successively get all the angles of $P_2$, $P_3$, $P_4$, $P_6$ and $A_{5,14}=A_{5,16}=\alpha$. We find two $\alpha$ in $P_5$, a contradiction. 

\begin{figure}[htp]
\centering
    \begin{tikzpicture}[scale=1]

    \foreach \x in {1,...,5}
    \draw
    		(-54+72*\x:1) -- (18+72*\x:1)
    		(-54+72*\x:1) -- (-54+72*\x:1.75) -- (-18+72*\x:2.5) -- (18+72*\x:1.75);

    \node at (0:0) {\small $1$};
    \node at (54:1.5) {\small $2$};
    \node at (126:1.5) {\small $3$};
    \node at (198:1.5) {\small $4$};
    \node at (-90:1.5) {\small $5$};
    \node at (-18:1.5) {\small $6$};

    \node at (18:0.7) {\small $\gamma$};
    \node at (8:1.1) {\small $\epsilon$};
    \node at (28:1.1) {\small $\alpha$};

    \node at (90:0.7) {\small $\alpha$};
    \node at (80:1.1) {\small $\gamma$};
    \node at (100:1.1) {\small $\epsilon$};

    \node at (162:0.7) {\small $\epsilon$};
    \node at (152:1.1) {\small $\alpha$};
    \node at (172:1.1) {\small $\gamma$};

    \node at (234:0.7) {\small $\beta$};
    \node at (224:1.1) {\small $\delta$};
    \node at (244:1.1) {\small $\alpha$};

    \node at (-54:0.7) {\small $\delta$};
    \node at (-44:1.1) {\small $\beta$};
    \node at (-64:1.1) {\small $\alpha$};

    \node at (12:1.6) {\small $\alpha$};
    \node at (24:1.6) {\small $\epsilon$};

    \node at (84:1.6) {\small $\delta$};
    \node at (96:1.6) {\small $\beta$};

    \node at (156:1.6) {\small $\gamma$};
    \node at (168:1.6) {\small $\alpha$};

    \node at (228:1.6) {\small $\beta$};

    \node at (-48:1.6) {\small $\delta$};

    \node at (54:2.2) {\small $\beta$};

    \node at (126:2.2) {\small $\delta$};

    \node at (198:2.2) {\small $\epsilon$};


    \node at (-18:2.2) {\small $\gamma$};

\begin{scope}[xshift=7cm]
	

	\foreach \x in {1,...,5}
    \draw
    		(-54+72*\x:1) -- (18+72*\x:1)
    		(-54+72*\x:1) -- (-54+72*\x:1.75) -- (-18+72*\x:2.5) -- (18+72*\x:1.75)
		(-18:2.5) -- (-18:3.3) -- (54:3.3)
		(54:2.5) -- (54:3.3) -- (126:3.3) 
		(126:2.5) -- (126:3.3) -- (198:3.3) -- (198:2.5);

    \node at (0:0) {\small $1$};
    \node at (54:1.5) {\small $2$};
    \node at (126:1.5) {\small $3$};
    \node at (198:1.5) {\small $4$};
    \node at (-90:1.5) {\small $5$};
    \node at (-18:1.5) {\small $6$};
    \node at (90:2.4) {\small $7$};
    \node at (162:2.4) {\small $8$};
    \node at (18:2.4) {\small $11$};

    \node at (18:0.7) {\small $\gamma$};
    \node at (8:1.1) {\small $\beta$};
    \node at (28:1.1) {\small $\beta$};

    \node at (90:0.7) {\small $\alpha$};
    \node at (80:1.1) {\small $\epsilon$};
    \node at (100:1.1) {\small $\gamma$};

    \node at (162:0.7) {\small $\epsilon$};
    \node at (152:1.1) {\small $\alpha$};
    \node at (172:1.1) {\small $\gamma$};

    \node at (234:0.7) {\small $\beta$};
    \node at (224:1.1) {\small $\delta$};
    \node at (244:1.1) {\small $\alpha$};

    \node at (-54:0.7) {\small $\delta$};
    \node at (-44:1.1) {\small $\epsilon$};
    \node at (-64:1.1) {\small $\delta$};

    \node at (18:1.9) {\small $\epsilon$};
    \node at (12:1.6) {\small $\delta$};
    \node at (24:1.6) {\small $\delta$};

    \node at (90:1.9) {\small $\beta$};
    \node at (84:1.6) {\small $\alpha$};
    \node at (96:1.6) {\small $\delta$};

    \node at (162:1.9) {\small $\gamma$};
    \node at (156:1.6) {\small $\epsilon$};
    \node at (168:1.6) {\small $\alpha$};

    \node at (228:1.6) {\small $\beta$};

    \node at (-48:1.6) {\small $\alpha$};

    \node at (54:2.2) {\small $\gamma$};
    \node at (50:2.6) {\small $\alpha$};
    \node at (58:2.6) {\small $\epsilon$};

    \node at (126:2.2) {\small $\beta$};
    \node at (122:2.6) {\small $\delta$};
    \node at (130:2.6) {\small $\alpha$};

    \node at (198:2.2) {\small $\epsilon$};
    \node at (194:2.6) {\small $\delta$};


    \node at (-18:2.2) {\small $\gamma$};
    \node at (-14:2.6) {\small $\beta$};

    \node at (51:3) {\small $\gamma$};
    \node at (57:3) {\small $\alpha$};

    \node at (123:3) {\small $\gamma$};
    \node at (129:3) {\small $\epsilon$};

    \node at (195:3) {\small $\beta$};


    \node at (-15:3) {\small $\delta$};

    \end{scope}

    \end{tikzpicture}
\caption{Tilings for the eigth arrangement.}
\label{abcde8A}
\end{figure}

\medskip

{\bf Arrangement 8, Case 2}: $A_{2,13}=\epsilon$, $A_{3,12}=\gamma$.

If $P_2$ is counterclockwise oriented and $P_3$ is clockwise oriented, as on the right of Figure \ref{abcde8A}, then by AVC, we successively get all the angles of $P_7$, $P_8$, $P_{11}$, $P_4$, $P_6$ and $A_{5,14}=\alpha$, $A_{5,16}=\delta$. We find $\alpha,\delta$ adjacent in $P_5$, a contradiction. By the same argument (in the right of Figure \ref{abcde8A}, exchanging $\beta$ with $\delta$, $\epsilon$ with $\gamma$, and then flipping horizontally), we also conclude that it is impossible for $P_2$ to be clockwise oriented and $P_3$ to be counterclockwise oriented. Therefore $P_2$ and $P_3$ must have the same orientation.

If both $P_2$, $P_3$ are counterclockwise oriented, as on the left of Figure \ref{abcde8B}, then by AVC, we successively get all the angles of $P_7$, $P_8$, $P_{11}$, $P_4$, $P_6$, $P_5$, $P_9$, $P_{10}$, $P_{12}$. There is no contradiction, and we get an angle congruent tiling with anglewise vertex combination $\{2\alpha^3,6\alpha\gamma\epsilon,6\beta^2\gamma,6\delta^2\epsilon\}$.

If both $P_2$, $P_3$ are clockwise oriented, as on the right of Figure \ref{abcde8B}, then by AVC, we successively get all the angles of $P_7$, $P_8$, $P_{11}$, $P_4$, $P_6$, $P_5$ and $A_{9,45}=\alpha$, $A_{9,48}=\beta$. We find $\alpha,\beta$ adjacent in $P_9$, a contradiction.
\end{proof}

\begin{figure}[htp]
\centering
    \begin{tikzpicture}[scale=1]

    \foreach \x in {1,...,5}
    \draw
    		(-54+72*\x:1) -- (18+72*\x:1)
    		(-54+72*\x:1) -- (-54+72*\x:1.75) -- (-18+72*\x:2.5) -- (18+72*\x:1.75)
		(-18+72*\x:2.5) -- (-18+72*\x:3.3) -- (54+72*\x:3.3) -- (54+72*\x:2.5);

    \node at (0:0) {\small $1$};
    \node at (54:1.5) {\small $2$};
    \node at (126:1.5) {\small $3$};
    \node at (198:1.5) {\small $4$};
    \node at (-90:1.5) {\small $5$};
    \node at (-18:1.5) {\small $6$};
    \node at (90:2.4) {\small $7$};
    \node at (162:2.4) {\small $8$};
    \node at (234:2.4) {\small $9$};
    \node at (-54:2.4) {\small $10$};
    \node at (18:2.4) {\small $11$};
    \node at (240:3.5){\small $12$};

    \node at (18:0.7) {\small $\gamma$};
    \node at (8:1.1) {\small $\beta$};
    \node at (28:1.1) {\small $\beta$};

    \node at (90:0.7) {\small $\alpha$};
    \node at (80:1.1) {\small $\epsilon$};
    \node at (100:1.1) {\small $\gamma$};

    \node at (162:0.7) {\small $\epsilon$};
    \node at (152:1.1) {\small $\delta$};
    \node at (172:1.1) {\small $\delta$};

    \node at (234:0.7) {\small $\beta$};
    \node at (224:1.1) {\small $\gamma$};
    \node at (244:1.1) {\small $\beta$};

    \node at (-54:0.7) {\small $\delta$};
    \node at (-44:1.1) {\small $\epsilon$};
    \node at (-64:1.1) {\small $\delta$};

    \node at (18:1.9) {\small $\epsilon$};
    \node at (12:1.6) {\small $\delta$};
    \node at (24:1.6) {\small $\delta$};

    \node at (90:1.9) {\small $\alpha$};
    \node at (84:1.6) {\small $\alpha$};
    \node at (96:1.6) {\small $\alpha$};

    \node at (162:1.9) {\small $\gamma$};
    \node at (156:1.6) {\small $\beta$};
    \node at (168:1.6) {\small $\beta$};

    \node at (234:1.9) {\small $\gamma$};
    \node at (228:1.6) {\small $\alpha$};
    \node at (240:1.6) {\small $\epsilon$};

    \node at (-54:1.9) {\small $\epsilon$};
    \node at (-48:1.6) {\small $\alpha$};
    \node at (-60:1.6) {\small $\gamma$};

    \node at (54:2.2) {\small $\gamma$};
    \node at (50:2.6) {\small $\alpha$};
    \node at (58:2.6) {\small $\epsilon$};

    \node at (126:2.2) {\small $\epsilon$};
    \node at (122:2.6) {\small $\gamma$};
    \node at (130:2.6) {\small $\alpha$};

    \node at (198:2.2) {\small $\epsilon$};
    \node at (194:2.6) {\small $\delta$};
    \node at (202:2.6) {\small $\delta$};

    \node at (-90:2.2) {\small $\alpha$};
    \node at (-94:2.6) {\small $\alpha$};
    \node at (-86:2.6) {\small $\alpha$};

    \node at (-18:2.2) {\small $\gamma$};
    \node at (-22:2.6) {\small $\beta$};
    \node at (-14:2.6) {\small $\beta$};

    \node at (54:3.5) {\small $\beta$};
    \node at (51:3) {\small $\gamma$};
    \node at (57:3) {\small $\beta$};

    \node at (126:3.5) {\small $\delta$};
    \node at (123:3) {\small $\delta$};
    \node at (129:3) {\small $\epsilon$};

    \node at (198:3.5) {\small $\gamma$};
    \node at (195:3) {\small $\beta$};
    \node at (201:3) {\small $\beta$};

    \node at (-90:3.5) {\small $\alpha$};
    \node at (-93:3) {\small $\epsilon$};
    \node at (-87:3) {\small $\gamma$};

    \node at (-18:3.5) {\small $\epsilon$};
    \node at (-15:3) {\small $\delta$};
    \node at (-21:3) {\small $\delta$};

\begin{scope}[xshift=7cm]
	

	\foreach \x in {1,...,5}
    \draw
    		(-54+72*\x:1) -- (18+72*\x:1)
    		(-54+72*\x:1) -- (-54+72*\x:1.75) -- (-18+72*\x:2.5) -- (18+72*\x:1.75)
		(-18+72*\x:2.5) -- (-18+72*\x:3.3) -- (54+72*\x:3.3) -- (54+72*\x:2.5);

    \node at (0:0) {\small $1$};
    \node at (54:1.5) {\small $2$};
    \node at (126:1.5) {\small $3$};
    \node at (198:1.5) {\small $4$};
    \node at (-90:1.5) {\small $5$};
    \node at (-18:1.5) {\small $6$};
    \node at (90:2.4) {\small $7$};
    \node at (162:2.4) {\small $8$};
    \node at (234:2.4) {\small $9$};
    \node at (-54:2.4) {\small $10$};
    \node at (18:2.4) {\small $11$};

    \node at (18:0.7) {\small $\gamma$};
    \node at (8:1.1) {\small $\epsilon$};
    \node at (28:1.1) {\small $\alpha$};

    \node at (90:0.7) {\small $\alpha$};
    \node at (80:1.1) {\small $\epsilon$};
    \node at (100:1.1) {\small $\gamma$};

    \node at (162:0.7) {\small $\epsilon$};
    \node at (152:1.1) {\small $\alpha$};
    \node at (172:1.1) {\small $\gamma$};

    \node at (234:0.7) {\small $\beta$};
    \node at (224:1.1) {\small $\alpha$};
    \node at (244:1.1) {\small $\delta$};

    \node at (-54:0.7) {\small $\delta$};
    \node at (-44:1.1) {\small $\alpha$};
    \node at (-64:1.1) {\small $\beta$};

    \node at (18:1.9) {\small $\beta$};
    \node at (12:1.6) {\small $\beta$};
    \node at (24:1.6) {\small $\gamma$};

    \node at (90:1.9) {\small $\alpha$};
    \node at (84:1.6) {\small $\beta$};
    \node at (96:1.6) {\small $\delta$};

    \node at (162:1.9) {\small $\delta$};
    \node at (156:1.6) {\small $\epsilon$};
    \node at (168:1.6) {\small $\delta$};

    \node at (234:1.9) {\small $\alpha$};
    \node at (228:1.6) {\small $\epsilon$};
    \node at (240:1.6) {\small $\gamma$};

    \node at (-48:1.6) {\small $\gamma$};
    \node at (-60:1.6) {\small $\epsilon$};

    \node at (54:2.2) {\small $\delta$};
    \node at (50:2.6) {\small $\delta$};
    \node at (58:2.6) {\small $\epsilon$};

    \node at (126:2.2) {\small $\beta$};
    \node at (122:2.6) {\small $\gamma$};
    \node at (130:2.6) {\small $\beta$};

    \node at (198:2.2) {\small $\beta$};
    \node at (194:2.6) {\small $\gamma$};
    \node at (202:2.6) {\small $\beta$};

    \node at (-90:2.2) {\small $\alpha$};

    \node at (-18:2.2) {\small $\delta$};
    \node at (-14:2.6) {\small $\epsilon$};

    \node at (51:3) {\small $\gamma$};
    \node at (57:3) {\small $\beta$};

    \node at (123:3) {\small $\delta$};
    \node at (129:3) {\small $\epsilon$};

    \node at (195:3) {\small $\alpha$};


    \node at (-15:3) {\small $\alpha$};

    \end{scope}

    \end{tikzpicture}
\caption{Tilings for the eigth arrangement, continued.}
\label{abcde8B}
\end{figure}

\section{Proof of Main Theorem}
\label{proofmain}

We have studied the spherical tilings by $12$ congruent pentagons from the purely edge length viewpoint and the purely angle viewpoint. In this section, we combine the two together to prove the main theorem. 

The following result shows that certain combinations of edge lengths and angles cannot happen for geometrical reasons.

\begin{lemma}\label{pcombo}
Suppose in the spherical pentagon on the left of Figure \ref{pentcombo}, three of the following four equalities hold
\[
a=b,\quad
c=d,\quad
\beta=\gamma,\quad
\delta=\epsilon,
\]
then all four equalities hold.
\end{lemma}

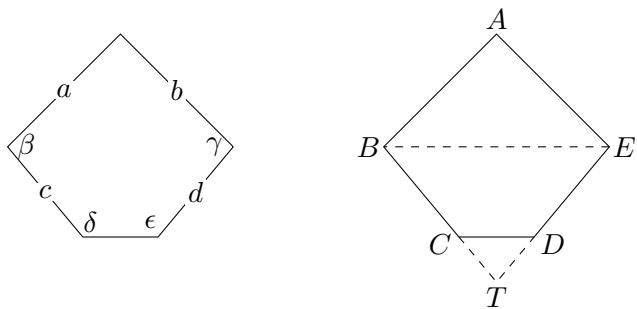
\begin{figure}[htp]
\centering
    \begin{tikzpicture}[scale=0.5]

	\draw (3,0) --node[fill=white,inner sep=1] {\small $b$} (0,3)
            --node[fill=white,inner sep=1] {\small $a$} (-3,0)
            --node[fill=white,inner sep=1] {\small $c$} (-1,-2.4)
            -- (1,-2.4)
            --node[fill=white,inner sep=1] {\small $d$} (3,0);
    
    \node at (-2.5,0) {\small $\beta$};
    \node at (2.5,0) {\small $\gamma$};
    \node at (-0.8,-2) {\small $\delta$};
    \node at (0.8,-2) {\small $\epsilon$};

    \begin{scope}[xshift=10cm]
   
    \draw (3,0) -- (0,3) -- (-3,0) -- (-1,-2.4) -- (1,-2.4) -- (3,0);
    \draw[dashed] 
    		(-1,-2.4) -- (0,-3.6) -- (1,-2.4)
    		(3,0) -- (-3,0);
    
    \node at (0,3.4) {\small $A$};
    \node at (-3.4,0) {\small $B$};
    \node at (3.4,0) {\small $E$};
    \node at (-1.5,-2.6) {\small $C$};
    \node at (1.5,-2.6) {\small $D$};    
    \node at (0,-4) {\small $T$}; 
    
    \end{scope}

    \end{tikzpicture}
\caption{Geometrical constraint.}
\label{pentcombo}
\end{figure}

\begin{proof}
Add auxiliary lines as on the right of Figure \ref{pentcombo}. 

Assume $a=b$ and $\beta=\gamma$. Then $a=b$ implies $\triangle ABE$ is an isosceles triangle, and $\beta=\gamma$ further implies that $\triangle TBE$ is also an isosceles triangle. Then $c=d$ if and only if $TC$ and $TD$ have the same length. This is further equivalent to $\triangle TCD$ being an isosceles triangle, which is the same as $\delta=\epsilon$.

Assume $c=d$ and $\epsilon=\delta$. Then $\epsilon=\delta$ implies $\triangle TCD$ is an isosceles triangle, and $c=d$ further implies $\triangle TBE$ is also  an isosceles triangle. Then $\beta=\gamma$ is equivalent to $\angle ABE=\angle AEB$, which means $a=b$.
\end{proof}

We recall the notations $P_i,E_{ij},V_{ijk}$ and the names $a$-edge, $abc$-vertex introduced in Section \ref{edgetile2}. We recall the names $\alpha\beta\gamma$-vertex, $\alpha\beta\gamma$-type vertex introduced in Section \ref{angletile1}. We also recall the notation $A_{i,jk}$ and the names $\alpha$-angle, $\alpha\beta$-edge in a tile introduced in Section \ref{angletile2}.

We can also name an angle by the two edges around the angle. For example, the angle $A_{1,23}$ in Figure \ref{generalclass} is an $ab$-angle, and $A_{1,34}$ is an $a^2$-angle.

We use $a,b,c,\dotsc$ to denote distinct edge lengths. We use $\alpha,\beta,\gamma,\dotsc$ to denote distinct angles. Moreover, $\alpha$ will always be $\frac{2\pi}{3}$.

\begin{proof}[Proof of Main Theorem]
The proof is divided according to the edge length combination in the pentagon. 

\medskip

{\bf Case} $a^5$.

If the angle combination is $\alpha^5$, then the pentagon is regular, and the tiling is the regular dodecahedron.

Now consider the middle three cases in Proposition \ref{angle}. Up to symmetry, the angle combination $\alpha^3\beta\gamma$ can be arranged as the first or the second in Figure \ref{case5a1}. By Proposition \ref{angle_pattern3} and up to symmetry, the angle combination $\alpha^2\beta^2\gamma$ can be arranged as the third or the fourth in Figure \ref{case5a1}. By Proposition \ref{angle_pattern4} and up to symmetry, the angle combination $\alpha^2\beta\gamma\delta$ can be arranged as the fifth in Figure \ref{case5a1}.

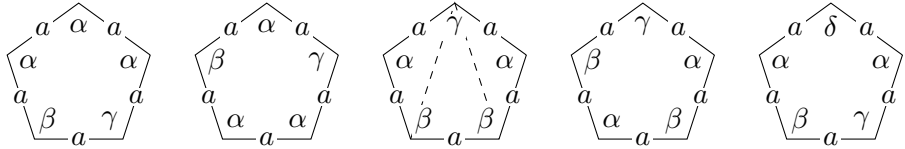
\begin{figure}[htp]
\centering
    \begin{tikzpicture}[scale=1]
    
\foreach \x in {0,...,4}
    \draw[xshift=2.5*\x cm] 
    	    (18:1) -- node[fill=white,inner sep=1] {\small $a$}
		(90:1) -- node[fill=white,inner sep=1] {\small $a$}
		(162:1) -- node[fill=white,inner sep=1] {\small $a$}
		(234:1) -- node[fill=white,inner sep=1] {\small $a$}
		(306:1) -- node[fill=white,inner sep=1] {\small $a$} (18:1);


    \node at (18:0.7) {\small $\alpha$};
    \node at (90:0.7) {\small $\alpha$};
    \node at (162:0.7) {\small $\alpha$};
    \node at (234:0.7) {\small $\beta$};
    \node at (306:0.7) {\small $\gamma$};

       
    \begin{scope}[xshift=2.5cm]
    
    \node at (18:0.7) {\small $\gamma$};
    \node at (90:0.7) {\small $\alpha$};
    \node at (162:0.7) {\small $\beta$};
    \node at (234:0.7) {\small $\alpha$};
    \node at (306:0.7) {\small $\alpha$};
    
    \end{scope}

       
    \begin{scope}[xshift=5cm]
    
    \draw[dashed] (234:1) -- (90:1) -- (306:1);
    
    \node at (18:0.7) {\small $\alpha$};
    \node[fill=white,inner sep=1] at (90:0.7) {\small $\gamma$};
    \node at (162:0.7) {\small $\alpha$};
    \node[fill=white,inner sep=1] at (234:0.7) {\small $\beta$};
    \node[fill=white,inner sep=1] at (306:0.7) {\small $\beta$};
    
  	\end{scope}

       
    \begin{scope}[xshift=7.5cm]
    
    \node at (18:0.7) {\small $\alpha$};
    \node at (90:0.7) {\small $\gamma$};
    \node at (162:0.7) {\small $\beta$};
    \node at (234:0.7) {\small $\alpha$};
    \node at (306:0.7) {\small $\beta$};
    
    \end{scope}

       
    \begin{scope}[xshift=10cm]
    
    \node at (18:0.7) {\small $\alpha$};
    \node at (90:0.7) {\small $\delta$};
    \node at (162:0.7) {\small $\alpha$};
    \node at (234:0.7) {\small $\beta$};
    \node at (306:0.7) {\small $\gamma$};
    
    \end{scope}

    \end{tikzpicture}
\caption{Combinations of $a^5$ with $\alpha^3\beta\gamma$, $\alpha^2\beta^2\gamma$, $\alpha^2\beta\gamma\delta$.}
\label{case5a1}
\end{figure}

By Lemma \ref{pcombo}, we have $\beta=\gamma$ in all except the third arrangement. In the third arrangement, divide the pentagon into three triangles. Note that the three triangles (and therefore the pentagon also) are completely determined by $a$. Moreover, the area of the two triangles on the side are strictly increasing in $a$. The length of the two dividing lines are strictly increasing in $a$, and the area of the triangle at the middle is also strictly increasing in $a$. Therefore the area of the pentagon is strictly increasing in $a$. So there is a unique $a$, such that the area of the pentagon is $\frac{1}{12}$ of the area of the sphere. Since this area is reached by the pentagon in the regular dodecahedron, we conclude that the third arrangement is the regular one.

Finally, for the angle combination $\alpha\beta\gamma\delta\epsilon$, we get four possible tilings from Proposition \ref{angle_pattern5}. Any one exists as long as there is a relevant pentagon with all sides equal. After suitable permutation among $\beta,\gamma,\delta,\epsilon$, these are included in the first four tilings in Figure \ref{completeclassify}. 

\medskip

{\bf Case} $a^4b$.

Recall that the edgewise vertex combination is $\{8a^3,12a^2b\}$. We will not use the complete classification in Proposition \ref{edge_pattern2}. Instead, we will study the edge-angle combinations directly.

\medskip

{\bf Subcase} $\alpha^5$.

In this case, $b$ is completely determined by $a$, and the area of the pentagon is strictly increasing in  $a$. Therefore there is only one $a$ such that the area is $\frac{1}{12}$ of the area of the sphere. This is the regular dodecahedron tiling. Therefore we get $a=b$, a contradiction.

\medskip

{\bf Subcase} $\alpha^3\beta\gamma$.

In addition to the edgewise vertex combination $\{8a^3,12a^2b\}$, Proposition \ref{angle} also tells us the anglewise vertex combination $\{8\alpha^3,12\alpha\beta\gamma\}$, which implies the following condition.

\medskip

{\bf AVC}(for $\alpha^3\beta\gamma$): Any vertex is either $\alpha^3$ or $\alpha\beta\gamma$. 

\medskip

If $\beta$ and $\gamma$ are separated in the pentagon, as on the left of Figure \ref{22combo}, then up to symmetry, there are three ways of assigning $b$, as the $\alpha^2$-edge or as one of two $\alpha\beta$-edges. Lemma \ref{pcombo} implies that two such assignments lead to $\beta=\gamma$, a contradiction. The only assignment that does not lead to a contradiction is the left of Figure \ref{22combo}. Similarly, if $\beta$ and $\gamma$ are adjacent in the pentagon, then among three possible ways (up to symmetry) of assigning $b$, one leads to a contradiction and we are left with the middle and the right of Figure \ref{22combo}.

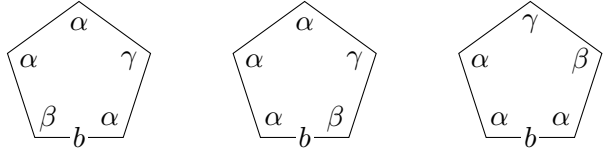
\begin{figure}[htp]
\centering
    \begin{tikzpicture}[scale=1]
	
	\foreach \x in {0,1,2}
    \draw[xshift=3*\x cm] (18:1) --	 (90:1) -- (162:1) -- (234:1) -- node[fill=white,inner sep=1] {\small $b$} (306:1) -- cycle;

	\node at (18:0.7) {\small $\gamma$};
	\node at (90:0.7) {\small $\alpha$};
	\node at (162:0.7) {\small $\alpha$};
	\node at (234:0.7) {\small $\beta$};
	\node at (-54:0.7) {\small $\alpha$};
	
	\begin{scope}[xshift=3cm]
	
	\node at (18:0.7) {\small $\gamma$};
	\node at (90:0.7) {\small $\alpha$};
	\node at (162:0.7) {\small $\alpha$};
	\node at (234:0.7) {\small $\alpha$};
	\node at (-54:0.7) {\small $\beta$};
	
	\end{scope}
	
	\begin{scope}[xshift=6cm]
	
	\node at (18:0.7) {\small $\beta$};
	\node at (90:0.7) {\small $\gamma$};
	\node at (162:0.7) {\small $\alpha$};
	\node at (234:0.7) {\small $\alpha$};
	\node at (-54:0.7) {\small $\alpha$};
		
	\end{scope}

    \end{tikzpicture}
\caption{Combinations of $a^4b$ and $\alpha^3\beta\gamma$.}
\label{22combo}
\end{figure}

We start the further discussion by assuming the tile $P_1$ in Figure \ref{tile_pattern_pic} is given by one of the pentagons in Figure \ref{22combo}.

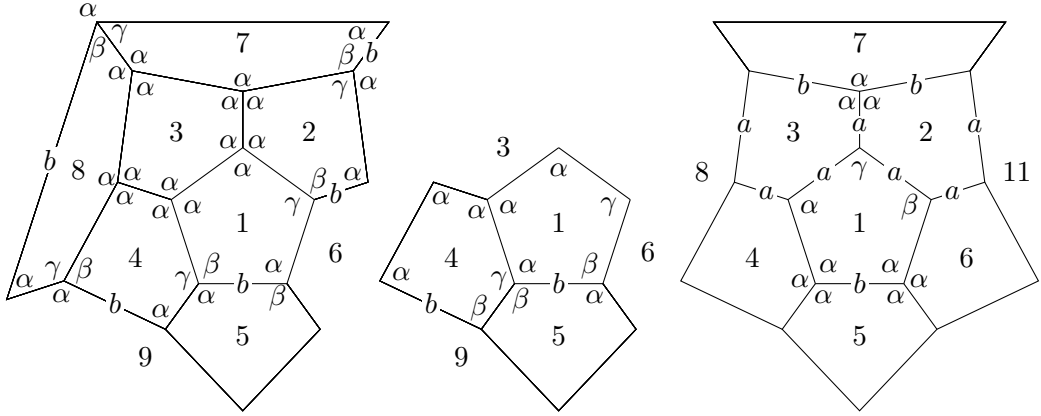
\begin{figure}[htp]
\centering
    \begin{tikzpicture}[scale=1]

\foreach \x in {1,...,5}
    \draw 
    		(-54+72*\x:1) -- (18+72*\x:1)
		(18:1) -- (18:1.75) -- (54:2.5) -- (90:1.75)
		(90:1) -- (90:1.75) -- (126:2.5) -- (162:1.75)
		(162:1) -- (162:1.75) -- (198:2.5) -- (234:1.75) 
		(234:1) -- (234:1.75) -- (-90:2.5) -- (-54:1.75) -- (-54:1) 
		(54:2.5) -- (54:3.3) -- (126:3.3) 
		(126:2.5) -- (126:3.3) -- (198:3.3) -- (198:2.5);

    \node at (0:0) {\small $1$};
    \node at (54:1.5) {\small $2$};
    \node at (126:1.5) {\small $3$};
    \node at (198:1.5) {\small $4$};
    \node at (-90:1.5) {\small $5$};
    \node at (-18:1.3) {\small $6$};
    \node at (90:2.4) {\small $7$};
    \node at (162:2.3) {\small $8$};
    \node at (234:2.2) {\small $9$};

	\node[fill=white,inner sep=1] at (-90:0.8) {\small $b$};
	\node[fill=white,inner sep=1] at (212:2) {\small $b$};
	\node[fill=white,inner sep=1] at (162:2.7) {\small $b$};
	\node[fill=white,inner sep=1] at (53:2.85) {\small $b$};
	\node[fill=white,inner sep=1] at (18:1.3) {\small $b$};

    \node at (18:0.7) {\small $\gamma$};
    \node at (29:1.15) {\small $\beta$};
    
    \node at (90:0.7) {\small $\alpha$};
    \node at (80:1.1) {\small $\alpha$};
    \node at (100:1.1) {\small $\alpha$};
    
    \node at (162:0.7) {\small $\alpha$};
    \node at (152:1.1) {\small $\alpha$};
    \node at (172:1.1) {\small $\alpha$};
    
    \node at (234:0.7) {\small $\beta$};
    \node at (224:1.1) {\small $\gamma$};
    \node at (244:1.1) {\small $\alpha$};
    
    \node at (-54:0.7) {\small $\alpha$};
    \node at (-65:1.1) {\small $\beta$};

    \node at (24:1.6) {\small $\alpha$};
    
    \node at (90:1.9) {\small $\alpha$};
    \node at (84:1.6) {\small $\alpha$};
    \node at (96:1.6) {\small $\alpha$};
    
    \node at (162:1.9) {\small $\alpha$};
    \node at (156:1.6) {\small $\alpha$};
    \node at (168:1.6) {\small $\alpha$};
    
    \node at (227:1.6) {\small $\alpha$};

    \node at (54:2.2) {\small $\gamma$};
    \node at (48:2.5) {\small $\alpha$};
    \node at (58:2.6) {\small $\beta$};
    
    \node at (126:2.2) {\small $\alpha$};
    \node at (122:2.6) {\small $\alpha$};
    \node at (130:2.6) {\small $\alpha$};
    
    \node at (197:2.2) {\small $\beta$};
    \node at (194:2.6) {\small $\gamma$};
    \node at (202:2.6) {\small $\alpha$};
   
    \node at (59:2.95) {\small $\alpha$};

    \node at (126:3.5) {\small $\alpha$};
    \node at (123:3) {\small $\gamma$};
    \node at (130:3) {\small $\beta$};

    \node at (195:3) {\small $\alpha$};

    \begin{scope}[xshift=4.2cm]
    
        \foreach \x in {1,...,5}
    \draw 
    		(-54+72*\x:1) -- (18+72*\x:1)
		(162:1) -- (162:1.75) -- (198:2.5) -- (234:1.75)
		(234:1) -- (234:1.75) -- (270:2.5) -- (306:1.75) -- (306:1);

    \node at (0:0) {\small $1$};
    \node at (126:1.25) {\small $3$};
    \node at (198:1.5) {\small $4$};
    \node at (-90:1.5) {\small $5$};
    \node at (-18:1.25) {\small $6$};
    \node at (234:2.2) {\small $9$};

	\node[fill=white,inner sep=1] at (-90:0.8) {\small $b$};
	\node[fill=white,inner sep=1] at (212:2) {\small $b$};

    \node at (18:0.7) {\small $\gamma$};
    
    \node at (90:0.7) {\small $\alpha$};
    
    \node at (162:0.7) {\small $\alpha$};
    \node at (172:1.1) {\small $\alpha$};
    
    \node at (234:0.7) {\small $\alpha$};
    \node at (224:1.1) {\small $\gamma$};
    \node at (244:1.15) {\small $\beta$};
    
    \node at (-54:0.7) {\small $\beta$};
    \node at (-64:1.1) {\small $\alpha$};
    
    \node at (168:1.6) {\small $\alpha$};
    
    \node at (227:1.55) {\small $\beta$};
    
    \node at (198:2.2) {\small $\alpha$};
    
    \end{scope}
  
    \begin{scope}[xshift=8.2cm]
    
	\foreach \x in {1,...,5}
    \draw 
    		(-54+72*\x:1) -- (18+72*\x:1)
    		(-54+72*\x:1) -- (-54+72*\x:1.75) -- (-18+72*\x:2.5) -- (18+72*\x:1.75)
		(54:2.5) -- (54:3.3) -- (126:3.3) -- (126:2.5);

    \node at (0:0) {\small $1$};
    \node at (54:1.5) {\small $2$};
    \node at (126:1.5) {\small $3$};
    \node at (198:1.5) {\small $4$};
    \node at (-90:1.5) {\small $5$};
    \node at (-18:1.5) {\small $6$};
    \node at (90:2.4) {\small $7$};
    \node at (162:2.2) {\small $8$};
    \node at (18:2.2) {\small $11$};

	\node[fill=white,inner sep=1] at (-90:0.8) {\small $b$};
	\node[fill=white,inner sep=1] at (54:0.8) {\small $a$};
	\node[fill=white,inner sep=1] at (126:0.8) {\small $a$};
	
	\node[fill=white,inner sep=1] at (18:1.3) {\small $a$};
	\node[fill=white,inner sep=1] at (90:1.3) {\small $a$};
	\node[fill=white,inner sep=1] at (162:1.3) {\small $a$};
	
	\node[fill=white,inner sep=1] at (40:2) {\small $a$};
	\node[fill=white,inner sep=1] at (68:2) {\small $b$};
	\node[fill=white,inner sep=1] at (112:2) {\small $b$};
	\node[fill=white,inner sep=1] at (140:2) {\small $a$};

    \node at (18:0.7) {\small $\beta$};
    
    \node at (90:0.7) {\small $\gamma$};
    
    \node at (162:0.7) {\small $\alpha$};
    
    \node at (234:0.7) {\small $\alpha$};
    \node at (224:1.1) {\small $\alpha$};
    \node at (244:1.1) {\small $\alpha$};
    
    \node at (-54:0.7) {\small $\alpha$};
    \node at (-64:1.1) {\small $\alpha$};
    \node at (-44:1.1) {\small $\alpha$};

    \node at (90:1.9) {\small $\alpha$};
    \node at (84:1.6) {\small $\alpha$};
    \node at (96:1.6) {\small $\alpha$};

    \end{scope}

    \end{tikzpicture}
\caption{Tilings for the combinations of $a^4b$ and $\alpha^3\beta\gamma$.}
\label{22comboA}
\end{figure}

The {\em first combination} is the left of Figure \ref{22comboA}. Since the $b$-edge is also an $\alpha\beta$-edge of $P_5$, the angles $A_{5,14}$, $A_{5,16}$ must be $\alpha,\beta$. By the AVC condition above, we cannot have $\beta^2$ at a vertex. Therefore we get $A_{5,14}=\alpha$, $A_{5,16}=\beta$. By AVC, we further get $A_{4,15}=\gamma$, so that the $b$-edge of $P_4$ is either $E_{34}$ or $E_{49}$. We note that, for the current combination, the choice of the $b$-edge together with the location of $\gamma$ determine all the angles of a tile. The choice $E_{34}=b$ gives all the angles of $P_4$ and $A_{4,13}=\alpha$, $A_{4,38}=\beta$. By AVC, we get $A_{3,14}=\alpha$, and since $b$ is an $\alpha\beta$-edge of $P_3$, we further get $A_{3,48}=\beta$. But $A_{4,38}=A_{3,48}=\beta$ contradicts to AVC at $V_{348}$. Therefore we must have $E_{49}=b$. Together with $A_{4,15}=\gamma$, we get all the angles of $P_4$, as described in the picture. 

So we see $P_1$ completely determines $P_4$. Note that the edge $E_{14}$ shared by the two tiles is opposite to $\gamma$ in $P_1$. Using this as a guidance, $P_4$ determines $P_8$, $P_8$ determines $P_7$, and $P_7$ determines $P_2$. Then by AVC, all vertices of $P_3$ are $\alpha^3$-vertices, so that all the angles of $P_3$ are $\alpha$, a contradiction.

The {\em second combination} is the middle of Figure \ref{22comboA}. Similar to the first arrangement, by AVC and the fact that the $b$-edge is an $\alpha\beta$-edge, we get $A_{5,16}=\alpha$, $A_{5,14}=\beta$. By AVC again, we get $A_{4,15}=\gamma$, and the $b$-edge of $P_4$ is either $E_{34}$ or $E_{49}$. If $E_{34}=b$, then we get all the angles of $P_4$ and $A_{4,13}=\beta$. By AVC, we get $A_{3,14}=\gamma$, contradicting to the fact that $b$ is an $\alpha\beta$-edge in $P_3$. Therefore we must have $E_{49}=b$. This determines all the angles of $P_4$, as described in the picture.

So we see $P_1$ completely determines $P_4$, just like the first arrangement. By the same guidance, we can successively get all the angles of $P_8,P_7,P_2$ and find all the angles of $P_3$ to be $\alpha$, a contradiction.

The {\em third combination} is the right of Figure \ref{22comboA}. The combination has the special property that $\gamma$ and $b$ are opposite to each other in a tile and determine the location of each other. Moreover, since the $b$-edge is an $\alpha^2$-edge in $P_5$, we get $A_{5,14}=A_{5,16}=\alpha$. By AVC, this implies that both ends of any $b$-edge are $\alpha^3$-vertices. By the edgewise vertex combination, this implies that all $\alpha\beta\gamma$-vertices are $a^3$-vertices. By AVC, we get $V_{123}=V_{126}=\alpha\beta\gamma$, so that $V_{123}=V_{126}=a^3$. Moreover, the angles $A_{2,13}$, $A_{3,12}$ are $\alpha,\beta$, so that the opposite edges $E_{2\overline{11}}=E_{38}=a$. From the four known $a$-edges of $P_2$, we get $E_{27}=b$. As an end of this $b$-edge, we get $V_{237}=\alpha^3$. As the edge opposite to $A_{3,27}=\alpha$, we get $E_{34}=a$. The four known $a$-edges of $P_3$ imply $E_{37}=b$. Then we find two $b$-edges in $P_7$, a contradiction.

\medskip

{\bf Subcase} $\alpha^2\beta^2\gamma$.

We need to assign the $b$-edges to the two tilings in Figure \ref{2a2b1cB}. The $b$-edge has to be a $\rho\tau$-edge for some (not necessarily distinct) angles $\rho$ and $\tau$. Therefore the angles around the $b$-edge has to appear as one of the two scenarios in Figure \ref{b_edge} (different $b$-edges may follow different scenario).

\begin{figure}[htp]
\centering
    \begin{tikzpicture}[scale=1]
    
\foreach \x in {0,4}
\draw[xshift=\x cm]
	(-1.3,0.5) -- (-1,0) -- (-1.3,-0.5)
	(1.3,0.5) -- (1,0) -- (1.3,-0.5)
	(-1,0) -- node[fill=white] {$b$} (1,0);
	
\node at (-0.8,0.2) {$\rho$};
\node at (-0.8,-0.2) {$\rho$};
\node at (0.8,0.2) {$\tau$};
\node at (0.8,-0.2) {$\tau$};

\node[xshift=4cm] at (-0.8,0.2) {$\rho$};
\node[xshift=4cm] at (-0.8,-0.2) {$\tau$};
\node[xshift=4cm] at (0.8,0.2) {$\tau$};
\node[xshift=4cm] at (0.8,-0.2) {$\rho$};

    \end{tikzpicture}
\caption{Angles around $b$-edge.}
\label{b_edge}
\end{figure}
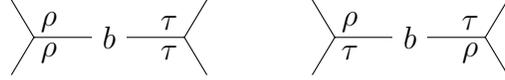

For the tiling on the left of Figure \ref{2a2b1cB}, the only edge in $P_1$ fitting the scenarios in Figure \ref{b_edge} is the $\beta^2$-edge. By assigning the $\beta^2$-edge as $b$ in all the tiles, we get a special case of the tiling $T_5$ in Figure \ref{completeclassify} and Figure \ref{generalclass}, after suitable change of symbols.

For the tiling on the right of Figure \ref{2a2b1cB}, the thick edge (which is the $\beta\gamma$-edge) cannot be the $b$-edge because this would imply $\beta=\gamma$ by Lemma \ref{pcombo}. On the other hand, any ordinary edge is adjacent to the thick edge in some tile. Therefore in some tile, the $b$-edge is adjacent to the thick edge. Since the congruence between tiles preserves the thick edge (because it is the $\beta\gamma$-edge) and the $b$-edge, the $b$-edge is adjacent to the thick edge in every tile. Then $E_{12}$, $E_{13}$ cannot be the $b$-edge because they are not adjacent to the thick edge $E_{15}$ of $P_1$. Moreover, $E_{14}$, $E_{16}$ cannot be the $b$-edge because they are not adjacent to the thick edges in $P_4$, $P_6$. Therefore no edge of $P_1$ can be the $b$-edge, a contradiction.

\medskip

{\bf Subcase} $\alpha^2\beta\gamma\delta$.

We need to assign the $b$-edges to the right of Figure \ref{2a1b1c1dA}. Again the angles around the $b$-edge must fit Figure \ref{b_edge}. The only such possible edge is the $\beta\gamma$-edge. However, $E_{15}=b$ implies $\beta=\gamma$ by Lemma \ref{pcombo}. We get a contradiction.

\medskip

{\bf Subcase} $\alpha\beta\gamma\delta\epsilon$.

We need to assign the $b$-edges to the four tilings given by Proposition \ref{angle_pattern5}. The angles around the $b$-edge must fit Figure \ref{b_edge}.

In the tiling on the left of Figure \ref{abcde2A}, no edge of $P_3$ fits Figure \ref{b_edge}. 

In the tilings on the left of Figures \ref{abcde3B}, \ref{abcde7B}, \ref{abcde8B}, by looking for edges of $P_3$ in the first two cases and of $P_1$ in the third case, we find the only edge fitting Figure \ref{b_edge} is the $\beta\delta$-edge. By assigning $b$ to all the $\beta\delta$-edges and suitably changing the symbols for the angles, we get the tilings $T_2$, $T_3$, $T_4$ of Figure \ref{completeclassify}. 

\medskip

{\bf Case} $a^3b^2$.

The tiling is given on the left of Figure \ref{case3a2b} and has the edgewise vertex combination $\{4a^3,4b^3,12a^2b\}$. The pentagon has $2$ $a^2$-angles $\theta_1$ and $\theta_2$, $2$ $ab$-angles $\rho_1$ and $\rho_2$, and $1$ $b^2$-angle $\tau$, described on the left of Figure \ref{a3b2_angle}. Then $a^3$-vertices, $a^2b$-vertices, $b^3$-vertices are respectively $\theta_i\theta_j\theta_k$-vertices, $\theta_i\rho_j\rho_k$-vertices, $\tau^3$-vertices. All such vertices exist.

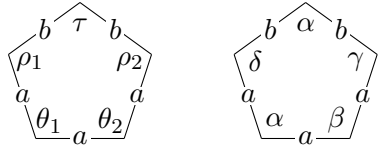
\begin{figure}[htp]
\centering
    \begin{tikzpicture}[scale=1]
   
    \foreach \y in {0,3}
    {
    \begin{scope}[xshift=\y cm]
    
    \foreach \x in {1,...,5}
    \draw 
    		(-54+72*\x:1) -- (18+72*\x:1);
		
    	\node[fill=white,inner sep=1] at (54:0.8) {\small $b$};
	\node[fill=white,inner sep=1] at (126:0.8) {\small $b$};
	\node[fill=white,inner sep=1] at (198:0.8) {\small $a$};
	\node[fill=white,inner sep=1] at (-90:0.8) {\small $a$};
	\node[fill=white,inner sep=1] at (-18:0.8) {\small $a$};
	
	\end{scope}
	}

    \node at (18:0.7) {\small $\rho_2$};  
    \node at (90:0.7) {\small $\tau$};
    \node at (162:0.7) {\small $\rho_1$};    
    \node at (234:0.7) {\small $\theta_1$};     
    \node at (-54:0.7) {\small $\theta_2$};
    
    \node[xshift=3cm] at (18:0.7) {\small $\gamma$};  
    \node[xshift=3cm] at (90:0.7) {\small $\alpha$};
    \node[xshift=3cm] at (162:0.7) {\small $\delta$};    
    \node[xshift=3cm] at (234:0.7) {\small $\alpha$};     
    \node[xshift=3cm] at (-54:0.7) {\small $\beta$};

    \end{tikzpicture}
\caption{Pentagon for the combination of $a^3b^2$ and $\alpha^2\beta\gamma\delta$.}
\label{a3b2_angle}
\end{figure}

We have $\theta_1+\theta_2+\rho_1+\rho_2+\tau=5\alpha$. Since there are $\tau^3$-vertices, we have $3\tau=3\alpha$. Therefore 
\[
\tau=\alpha,\quad
\theta_1+\theta_2+\rho_1+\rho_2=4\alpha.
\]

By Lemma \ref{pcombo}, we have $\theta_1=\theta_2$ if and only if $\rho_1=\rho_2$. In this case, we have $\theta_1+2\rho_1=3\alpha$ at a $\theta_i\rho_j\rho_k$-vertex. Combined with $2\theta_1+2\rho_1=4\alpha$, we conclude that all the angles are equal to $\alpha$. By Lemma \ref{pcombo} again, we must have $a=b$, a contradiction. Therefore we have $\theta_1\ne\theta_2$ and $\rho_1\ne\rho_2$.

At a $\theta_i\theta_j\theta_k$-vertex, we have $\theta_i+\theta_j+\theta_k=3\alpha$. Without loss of generality, we may assume the equality is either $2\theta_1+\theta_2=3\alpha$ or $3\theta_1=3\alpha$.

\medskip

{\bf Subcase} $2\theta_1+\theta_2=3\alpha$, $\theta_1\ne \theta_2$.

Since $\theta_1\ne \theta_2$, the only $a^3$-vertices (same as $\theta_i\theta_j\theta_k$-vertices) are $\theta_1^2\theta_2$-vertices. There are $4$ $a^3$-vertices, which involve $8$ angle $\theta_1$ and $4$ angle $\theta_2$. Among the total of $12$ angle $\theta_1$ and $12$ angle $\theta_2$, the remaining $4$ angle $\theta_1$ and $8$ angle $\theta_2$ must appear in the remaining $12$ $a^2b$-vertices (same as $\theta_i\rho_j\rho_k$-vertices). 

Note that a $\theta_1\rho_1\rho_2$-vertex implies $\theta_1+\rho_1+\rho_2=3\alpha$. Combined with $\theta_1+\theta_2+\rho_1+\rho_2=4\alpha$ and $2\theta_1+\theta_2=3\alpha$, we get $\theta_1=\alpha=\theta_2$, a contradiction. Similarly, there is no $\theta_2\rho_1\rho_2$-vertex. Therefore $a^2b$-vertices are $\theta_i\rho_j^2$-vertices, and we have $4$ $\theta_1\rho_j^2$-vertices and $8$ $\theta_2\rho_k^2$-vertices. Since $j=k$ will lead to $\theta_1=3\alpha-2\rho_j=3\alpha-2\rho_k=\theta_2$, we must have $j\ne k$. Without loss of generality, we may assume $j=1$, $k=2$ and get $\theta_1+2\rho_1=3\alpha$ and $\theta_2+2\rho_2=3\alpha$. Combined with $\theta_1+\theta_2+\rho_1+\rho_2=4\alpha$ and $2\theta_1+\theta_2=3\alpha$, we get all the angles equal to $\alpha$, a contradiction.

\medskip

{\bf Subcase} $3\theta_1=3\alpha$, $\theta_1\ne \theta_2$.

We have $\theta_1=\alpha$ and $\theta_2+\rho_1+\rho_2=3\alpha$. If any of $\theta_1\theta_2^2$, $\theta_1^2\theta_2$, $\theta_2^3$ is a vertex, then we will always get $\theta_2=\alpha=\theta_1$, a contradiction. Therefore the only $a^3$-vertices (same as $\theta_i\theta_j\theta_k$-vertices) are $\theta_1^3$-vertices. Since all $12$ angle $\theta_1$ are concentrated at the $4$ $a^3$-vertices, the $12$ $a^2b$-vertices (same as $\theta_i\rho_j\rho_k$-vertices) must be $\theta_2\rho_j\rho_k$-vertices. The possible $\theta_2\rho_j\rho_k$-vertices are $\theta_2\rho_1^2$, $\theta_2\rho_2^2$, $\theta_2\rho_1\rho_2$. If any two appear as vertices, then we get $\rho_1=\rho_2$, a contradiction. Since both $\rho_1$ and $\rho_2$ must appear, we conclude that all $12$ $a^2b$-vertices are $\theta_2\rho_1\rho_2$-vertices.

Denote $\theta_2=\beta$, $\rho_1=\delta$, $\rho_2=\gamma$. Then $\alpha\ne \beta$ and $\delta\ne\gamma$, and the pentagon is given by the right of Figure \ref{a3b2_angle}. The discussion in the last paragraph tells us that, as far as vertices are concerned, we have
\[
a^3=\alpha^3,\quad 
a^2b=\beta\gamma\delta. \qquad (\text{note that } \theta_2\rho_1\rho_2=\beta\gamma\delta)
\]
Now we try to fit the right of Figure \ref{a3b2_angle} into the left of Figure \ref{case3a2b}, the only edge congruent tiling for $a^3b^2$. The two $a^2$-angles $A_{1,45}$ and $A_{1,56}$ in $P_1$ are $\theta_1=\alpha$ and $\theta_2=\beta$. Since $V_{145}=a^3=\alpha^3$ and $V_{156}=a^2b=\beta\gamma\delta$, we get $A_{1,45}=\alpha$ and $A_{1,56}=\beta$. We also know the $b^2$-angle of $P_1$ is $A_{1,23}=\tau=\alpha$. Knowing three angles of $P_1$ determines the other two angles. By the same argument, we get all the angles of all tiles. After rotating the result by angle $\frac{\pi}{5}$ in clockwise orientation and letting $c=a$, we get the tiling $T_5$ in Figure \ref{completeclassify} and Figure \ref{generalclass}.

\medskip

{\bf Case} $a^2b^2c$.

The tiling is given on the left of Figure \ref{case2a2bc}. Since $a,b,c$ are distinct, we can uniquely define $\alpha_1,\alpha_2,\beta,\gamma,\delta$ as the $a^2$-angle, the $b^2$-angle, the $ac$-angle, the $bc$-angle, the $ab$-angle in the pentagon. After assigning these angles, we find $3\alpha_1=3\alpha_2=2\pi$ at $b^3$-vertices. Therefore $\alpha_1=\alpha_2=\alpha$, and we get the tiling $T_5$ in Figure \ref{completeclassify} and Figure \ref{generalclass}.

This completes the proof of the main classification theorem.
\end{proof}

\end{document}